\documentclass[a4paper]{amsart}

\usepackage[colorlinks=true, linkcolor=black, citecolor=black, urlcolor=black]{hyperref}
\usepackage{cleveref}
\usepackage[centering]{geometry}
\usepackage{filecontents}
\usepackage{float}
\usepackage{tcolorbox}

\usepackage[english]{babel}
\usepackage[utf8]{inputenc}
\usepackage[T1]{fontenc}

\usepackage{amsmath}
\usepackage{amssymb}

\usepackage{xcolor}
\usepackage{bbm}

\usepackage{enumerate}
\usepackage{scalerel} 
\usepackage{subfigure}

\usepackage{graphicx}

\usepackage{amsthm}
\newtheorem{theorem}{Theorem}
\numberwithin{theorem}{section}
\newtheorem{lemma}[theorem]{Lemma}
\newtheorem{proposition}[theorem]{Proposition}
\newtheorem{corollary}[theorem]{Corollary}
\theoremstyle{definition}
\newtheorem{definition}[theorem]{Definition}
\newtheorem{example}[theorem]{Example}
\newtheorem{remark}[theorem]{Remark}

\newcommand{\N}{\mathbb{N}}

\newcommand{\R}{\mathbb{R}}

\newcommand{\X}{\mathcal{X}}
\newcommand{\Xs}{\mathcal X}
\newcommand{\Ys}{\mathcal Y}

\newcommand{\E}{\mathbb{E}}

\newcommand{\F}{\mathcal{F}}
\newcommand{\G}{\mathcal{G}}
\newcommand{\B}{\mathcal{B}}

\newcommand{\supp}{\textup{supp}}

\newcommand{\pr}{\textup{pr}}
\newcommand{\id}{\textup{id}}

\newcommand{\proj}{\textup{proj}}

\newcommand{\law}{\mathcal{L}}
\newcommand{\dom}{\textup{dom}}

\newcommand{\cpl}{\textup{Cpl}}

\newcommand{\W}{\mathcal{W}}
\newcommand{\WW}{\mathbb{W}}
\newcommand{\AW}{\mathcal{AW}}

\newcommand{\Pc}{\mathcal P}

\renewcommand{\epsilon}{\varepsilon}
\renewcommand{\subset}{\subseteq}

\renewcommand{\P}{\mathbb{P}}

\newcommand{\FFP}{\mathcal{FP}}

\newcommand{\FP}{\textup{FP}}

\newcommand{\cont}{\textup{cont}}
\newcommand{\cplopt}{\cpl_{\rm opt}}
\newcommand{\cplbc}{\cpl_{\rm bc}}
\newcommand{\mean}[1]{\mathrm{mean}(#1)}
\newcommand{\var}[1]{\mathrm{Var}(#1)}
\newcommand{\simad}{\sim_{\mathrm{ad}}}
\renewcommand{\H}{H}

\renewcommand{\phi}{\varphi}

\newcommand{\fp}[1]{{\mathbb #1}}

\newcommand{\mc}{\textup{MC}}
\newcommand{\MC}{\textup{MC}}
\newcommand{\lemc}{\preccurlyeq_\mc}
\newcommand{\gemc}{\succcurlyeq_\mc}
\newcommand{\ip}{\textup{ip}}
\newcommand{\lawad}{\law^{\textup{ad}}}

\newcommand{\Zc}{\mathcal{Z}}
\newcommand{\ipp}{\mathbf{ip}}
\newcommand{\lawadd}{\law^{\textnormal{\textbf{ad}}}}
\newcommand{\AMC}{\textup{AMC}}

\let\oldmarginpar\marginpar
\renewcommand\marginpar[1]{\-\oldmarginpar[\raggedleft\footnotesize #1]{\raggedright\footnotesize\color{red} #1}}
\begin{document}

\begin{abstract}
We develop Brenier theorems on iterated Wasserstein spaces. For a separable Hilbert space $H$ and $N\ge 1$, we construct a full-support probability $\Lambda$ on $\mathcal{P}_2^{N}(H)= \Pc_2(\ldots \Pc_2(\H)\ldots) $ that is transport regular: for every $Q$ with finite second moment, transporting $\Lambda$ to $Q$ with cost $\W_2^2$ admits a unique optimizer, and this optimizer is of Monge type. The analysis rests on a characterization of optimal couplings on $\mathcal{P}_2(H)$ and, more generally, on $\mathcal{P}_2^{\,N}(H)$ via convex potentials on the Lions lift; in the latter case we employ a new adapted version of the lift tailored to the $N$-step structure. A key idea is a new identification between optimal-transport $c$-conjugation (with $c$ given by maximal covariance) and classical convex conjugation on the lift.

A primary motivation comes from the adapted Wasserstein distance $\AW_2$: our results yield a first Brenier theorem for $\AW_2$ and characterize $\AW_2^2$-optimal couplings through convex functionals on the space of $L_2$-processes.

\medskip

\noindent\emph{keywords: optimal transport for random measures, Brenier's theorem, Lions lift, convex conjugate, Monge problem, adapted optimal transport.}
\end{abstract}

\title[A Brenier Theorem on $(\Pc_2(\ldots \Pc_2(\H)\ldots), \W_2 )$  and  Adapted Transport]{A Brenier Theorem for measures on  $ (\Pc_2(\ldots \Pc_2(\H)\ldots), \W_2 )$ and Applications to Adapted Transport}
\author{Mathias Beiglböck}
\author{Gudmund Pammer}
\author{Stefan Schrott}
\thanks{We thank Francois Delarue, Lorenzo Dello Schiavo, Martin Huesmann and Daniel Lacker for insightful discussions. 
}

\maketitle

\section{Introduction}
Let  $(\X,d)$ be a Polish metric space and write $\Pc_2(\X)$ for the probabilities with finite second moments. For  $\mu, \nu\in \Pc_2(\X)$ we write $\cpl(\mu, \nu)$ for the set of couplings or transport plans 
and \begin{align}\label{eq:W2} \W_2^2(\mu, \nu)= \inf_{\pi \in\cpl(\mu, \nu) } \int d^2\, d\pi = \inf_{X,Y:(\Omega, \P) \to \X, X\sim \mu, Y\sim \nu} \E[d(X,Y)^2]\end{align}
 for the squared Wasserstein-2   distance.
 Brenier's  theorem \cite{Br91}, see also \cite{KnSm84, RaRu90},  asserts that if $\X$  is  the Euclidean space and $\mu\ll \lambda$, the minimization problem \eqref{eq:W2} admits a unique optimizer $\pi^*$ and that $\pi^*$ is of Monge type, i.e.\ concentrated on the graph of a function. Moreover, the optimal Monge-transport is the gradient a convex function. This result stands at the beginning of the breathtaking development of optimal transport over the last decades and has inspired a number of powerful generalizations and refinements, see \cite{Gi11, Mc95, GaMc96, Mc01, AmKiPr04, AgCa11, ChPa2011, EmPa25} among others. 

We extend it to iterated Wasserstein spaces, i.e.\ we consider $(\X, d)= (\Pc_2(H), \W_2)$ and more generally   $ (\Pc_2(\ldots\Pc_2(
H)\ldots), \W_2) = (\Pc_2^N(H), \W_2)$ where $\H$ denotes a separable Hilbert space.  

\begin{theorem}[Monge solutions]\label{Thm:MainMongeIntro}
    There exists a full support probability $\Lambda$ on $\Pc_2^N(H)$  which is \emph{transport regular}. That is, for every probability $Q$ on $\Pc_2^N(H)$ with finite second moment 
    \begin{align}\label{eq:IteratedW2Intro}  \inf_{\Pi \in\cpl(\Lambda, Q) } \int \W_2^2 (\mu, \nu) \, d\Pi(\mu,\nu) = \inf_{X,Y:(\Omega, \P) \to \Pc_2^N(H),\, X\sim \Lambda,\, Y\sim Q} \E[W_2^2( X, Y)]\end{align}  admits a unique minimizer $\Pi^*\sim (X,Y)$ and $\Pi^*$ is of Monge type.  
\end{theorem} 
In fact we obtain here transport regularity of all orders, i.e.\ $\Lambda $ is concentrated on probabilities $\mu$ which are again transport regular, etc. Since every $P\ll \Lambda$ is also transport regular, we have

\begin{corollary}\label{cor:MainMongeIntro}
     The set of $\W_2$-transport regular measures on $\Pc_2^N(H)$ is dense. 
\end{corollary}

Most closely related to our work are the concurrently written article of Savare--Pinzi \cite{PiSa25} (see \Cref{sec:intro_literature} below) and the work of Emami--Pass \cite{EmPa25} which is concerned with $\Pc_2(M)$, where $M$ is a smooth Riemannian manifold. As pointed out in \cite{EmPa25} in the setting $X = \Pc_2(M)$, the choice of an appropriate reference measure is non-trivial. Recent deep work of Dello Schiavo \cite{De20} establishes a version of Rademacher’s theorem on Wasserstein spaces for a class of
reference measures satisfying certain assumptions which allows Emami--Pass to conclude transport regularity. It can be shown that probabilities satisfying these assumptions exist in the case where $M$ is the 1-dimensional torus but it seems difficult to extend the apporach to higher dimensions and in particular the case $\Pc_2^N(M),$ for $N\geq 2$. Rather the approach put forward here and in \cite{PiSa25} exploits a connection to the Lions lift which allows to use tools from convex analysis. 

We also note that while our main interest lies in results for $\Pc_2^N(\R^d)$ and their applications to the adapted Wasserstein distance, we will work with a separable Hilbert space $H$ throughout. The reason is that the proofs for both cases are identical, in fact our argument for the existence of transport regular measures on $\Pc_2^2(\R^d)$ already requires the construction of a transport regular measure on $\Pc_2(H)$ for infinite dimensional $H$.

Next we describe the construction of the measure $\Lambda$ and the characteriziation of optimizers in terms convex functions on the Lions' lift.

\subsection{Characterization of optimal transport maps via convex Lions lifts}\label{sec:BrenierAndLiftsIntro}

To set the stage we recall the formulation of Brenier's theorem for measures on $H$ in terms of couplings of random variables. 
We use  $\Lambda$ to denote  Lebesgue measure in the classical case $H=\R^d$ (see \cite{Br91}) or a regular Gaussian measure in the case of infinite-dimensional $H$ (see  \cite{AmGiSa08}). 
If $X:(\Omega, \F, \P)\to H$ is a random variable with $\law_\P X\ll \Lambda$ and $\phi: H\to (-\infty, \infty]$ is convex with  $\partial \phi(X) \neq \emptyset$ a.s., then  $\partial \phi(X)$ consists of exactly one element which we denote by $\nabla \phi(X)$.

\begin{theorem}[Brenier's Theorem \cite{Br91, AmGiSa08}]\label{thm:coupling_Brenier} 
Let $\pi\in \cpl(\mu,\nu)$ where  $\mu, \nu$ are probability measures on $H$ with finite second moment,  $\mu \ll \Lambda $, and let $X:(\Omega, \F, \P)\to H$ be a random variable with   $X \sim \mu$.
Then $\pi$ is an  optimal coupling for squared-distance cost if and only if 
\begin{align} 
\pi \sim (X,\nabla\phi (X))
\end{align}
for some convex lsc $\phi: H \to (-\infty, \infty]$ with $\partial \phi(X) \neq \emptyset$ a.s.

Moreover we have uniqueness, i.e.\ there is exactly one such coupling.
\end{theorem}

Following classical transport theory, optimal plans are characterized in terms of dual potentials. A key finding (see Theorem \ref{thm:MCiConvConjIntro} below) is that for optimal transport of measures on $\Pc_2(H)$, optimal dual potentials $\phi:\Pc_2(H)\to (-\infty, \infty]$ are best understood through their \emph{Lions lift} given by 
$$\overline\phi: L_2([0,1]; H) \to (-\infty, \infty], \quad \overline\phi(X):= \phi(\law_\lambda(X) ).$$
Here we use $L_2([0,1]; H)$ to denote the $L_2$-space of $H$-valued square-integrable functions on 
$([0,1], \lambda)$, where $\lambda$ is the  Lebesgue measure on $[0,1]$.
Evidently a functional $\overline\phi$ on  $L_2([0,1]; H)$ appears as the Lions lift of some function $\phi$ on $\Pc_2(H)$ if and only if it is \emph{law invariant}, i.e.\ $\overline\phi(X)= \overline\phi(Y)$ provided that $\law_\lambda (X) = \law_\lambda(Y)$.

 In the case $H= \R^d$ we take $\Lambda$ to be the \emph{occupation  measure} of $d$-dimensional Brownian motion $B=(B_t)_{t\in [0,1]}$. Recall that for a path $b:[0,1]\to \R^d$ of Brownian motion, the occupation  measure of $b$ is given by $\law_\lambda b = b_{\#}(\lambda)\in \Pc_2(\R^d)$ and  the occupation  measure of $B$ has law 
\begin{align}
    \Lambda:= (\omega \mapsto \law_\lambda(B(\omega) )_{\#}(\P)  = \law_{\P} (\law_\lambda B) \in \Pc_2(\Pc_2(\R^d)). 
\end{align}
In the case where $H$ has infinite dimension, we take $\Lambda$ to be the law of the occupation  measure of a standard Wiener process, see \Cref{sec:Wiener_process_reg} below. 

The desired Brenier theorem for measures on $\Pc_2(H)$ is then obtained from \Cref{thm:coupling_Brenier} by replacing $H$ with $L_2([0,1]; H)$  in the domain of the convex potential and the state space of random variables, resp.

\begin{theorem}[Brenier theorem on $\Pc_2(\H)$]\label{thm:1-coupling_Brenier} 
Let $\Pi \in \cpl(P,Q)$, where $P$ and $Q$ are probability measures on $\mathcal{P}_2(H)$ with finite second moments, $P \ll \Lambda$, and let $X:(\Omega,\mathcal{F},\mathbb{P}) \to L_2([0,1]; H)$ be a random variable such that $X \sim P$.
Then $\Pi$ is a $\W_2^2$-optimal coupling if and only if 
\begin{align} 
\Pi \sim (\law_\lambda X,\law_\lambda \nabla \phi (X))
\end{align}
 for a convex, lsc, law-invariant  $\overline\phi:  L_2([0,1]; H) \to (-\infty, \infty]$ with $  \partial\phi(X) \neq \emptyset $ a.s.

Moreover we have uniqueness, i.e.\ there is exactly one such coupling.
\end{theorem}

To tackle the case of general $\Pc_2^N(H)$, $N\geq 1$, we introduce an adapted  version of the Lions' lift. We consider the filtered probability space $([0,1]^N, (\F_t)_{t=1}^N, \lambda)$, where $\lambda $ denotes the uniform distribution on $[0,1]^N$ and   $(\F_t)_{t=1}^N$ the filtration generated by the coordinate projections. For $X,Y\in L_2([0,1]^N; H)$ we define the \emph{adapted law}
$$\lawad(X):=\law(\law( \ldots \law(X| \F_{N-1}) \ldots | \F_1)) \in \Pc_2^N(H)$$
and write $X\simad Y$ if $\lawad(X)= \lawad(Y)$. A basic but crucial fact which renders this representation useful is that
\begin{align}\label{eq:quotient_property}
    \W_2(P,Q) = \min_{ X \simad P, \, Y \simad Q}    \|X-Y\|_2,
\end{align}
see \Cref{prop:iteratedSkorohod} below.\footnote{
By Lisini's work \cite{Li07}, any $P_0, P_1\in \Pc_2^N(H)$ are joined by a geodesic. 
 We note that \eqref{eq:quotient_property} yields this explicitly: Pick  $X_0\simad P_0,X_1\simad P_1$ with $\W_2(P_0, P_1) = \|X_0-X_1\|$ and set $P_t:= \lawad((1-t)X_0+tX_1)$, $t\in [0,1]$.}    
 The  \emph{adapted Lions lift} of $\phi:\Pc_2^N(H) \to (-\infty, \infty]$ is
$$   \overline\phi: L_2([0,1]^N; H) \to (-\infty, \infty], \quad \overline\phi(X):= \phi(\lawad X).$$
Evidently $\overline\phi$ is of this form if and only if it is adapted-law invariant in the sense that $\overline\phi(X)=\overline\phi(Y)$ for $X$, $Y$ with $X\simad Y$. 

Here $\Lambda $ is from an $H$-valued Wiener sheet  which is the $N$ parameter version of the standard Wiener process, see \Cref{sec:TR} below. Using these notions we obtain:

 \begin{theorem}[$N$-iterated Brenier]\label{thm:N-coupling_Brenier} 
Let $\Pi \in \cpl(P,Q)$, where $P$ and $Q$ are probability measures on $\mathcal{P}_2^N(H)$ with finite second moments, $P \ll \Lambda$, and let $X:(\Omega,\mathcal{F},\mathbb{P}) \to L_2([0,1]^N; H)$ be a random variable such that $X \sim P$.
Then $\Pi$ is a $\W_2^2$-optimal coupling if and only if 
\begin{align}
\Pi \sim (\lawad  (X),\lawad (\nabla \phi (X)))
\end{align}
for some convex, lsc, adapted-law invariant
$\bar{\phi}: L_2([0,1]^N; H) \to (-\infty, \infty]$ 
with $\partial \overline\phi(X) \neq \emptyset$ 
a.s.

Moreover we have uniqueness, i.e.\ there is exactly one such coupling.
\end{theorem}

A first ingredient to Theorems \ref{thm:1-coupling_Brenier} and \ref{thm:N-coupling_Brenier} is the fundamental theorem of optimal transport together with a new connection between $c$-conjugates in the sense of optimal transport and convex analysis on the Lions lift; we describe this in \Cref{sec:Adlift_Intro} below.

A further important role 
is played by the ability to switch back and forth between a stochastic process viewpoint and the Wasserstein on Wasserstein viewpoint. Specifically the map $\lawad$ turns out to be an isometry between martingales equipped with the adapted Wasserstein distance and $(\Pc_2^N, \W_2)$. We discuss this connection as well as Brenier type theorems for the adapted Wasserstein distance in \Cref{sec:AW_intro}.

The construction of the transport regular measure $\Lambda$ uses that convex functions on Hilbert spaces are almost surely differentiable with respect to Gaussian measures as well as path properties of Wiener processes. E.g.\ already for the case $N=1$ and $H= \R^d$ we use the fact the for a typical Brownian path $b:[0,1]\to \R^d$, the occupation measure $\law_\lambda b$ is a transport regular probability on $\R^d$. This appears to be somewhat remarkable since the support of $\law_\lambda b$ has Hausdorff dimension 2 for every $d\geq 2$.

\medskip

We conclude this section with a result that does not appeal to a particular reference measure. It provides  a characterization of primal optimizers in terms convex potentials on $L_2([0,1]^N;H)$ and asserts that \emph{typically} (in the Baire category sense) couplings are of Monge type.

 \begin{theorem}\label{thm:only-coupling_Brenier} 
Let $\Pi\in \cpl(P,Q)$ where  $P, Q$ are probability measures on $\Pc_2^N(H)$ with finite second moment.
Then $\Pi$ is a $\W_2^2$-optimal coupling if and only if for some $X,Y:(\Omega, \F, \P)\to L_2([0,1]^N;H)$ and  $\overline\phi:  L_2([0,1]^N; H) \to (-\infty, \infty]$ convex lsc adapted-law invariant 
\begin{align} 
\Pi \sim (\lawad X,\lawad Y), \quad Y\in \partial \overline\phi (X) \text{ a.s.} 
\end{align}
Moreover, the set of $(P,Q)$ for which there exists a unique $\W_2^2$-optimal coupling $\pi^*$ and $\pi^*$ is concentrated on the graph of a bijection, is comeager in $\Pc_2^{N+1}(H) \times \Pc_2^{N+1}(H)$.
\end{theorem}

\subsection{Fundamental theorem of optimal transport, $MC$-convexity and $L$-convexity}\label{sec:Adlift_Intro}

 As we are interested in the connection to convex analysis on $L_2([0,1]; H)$, it is convenient to switch from the minimization of the squared-distance cost to the equivalent problem of finding the `maximal covariance' 
\begin{align}\label{eq:MC1} 
\mc(\mu, \nu)=\sup_{\pi \in\cpl(\mu, \nu) } \int  x\cdot y\, d\pi(x,y) = \sup_{X\sim \mu, Y\sim \nu} \E[X \cdot Y] \end{align}
of probabilities $\mu, \nu$ on $H$ with finite second moment.
Likewise, for  probabilities $P, Q$ on $\Pc_2(H)$ with finite second moment we replace \eqref{eq:W2}   by
\begin{align}\label{eq:MC2} 
\mc(P, Q)=\sup_{\pi \in\cpl(P, Q) } \int \mc(\mu, \nu) \, d\pi(\mu,\nu). 
\end{align}
For $P,Q\in \Pc_2^N(\H)$, $N>2$ we define  $\mc(P, Q)$ through the natural iteration of \eqref{eq:MC2}.

Optimal transport plans can be characterized in terms of dual potentials. This is made precise in the  
`fundamental theorem of optimal transport' (see e.g.\ \cite{AmGiSa08} or \cite{Vi09}) which we recall in a form convenient for $\MC$ costs.
For continuous symmetric $c:\Xs\times \Xs\to \R$
the $c$-convex conjugate of $\phi:\Xs\to [-\infty, +\infty]$ is 
$$ \phi^c (y):=\sup_{x\in \Xs} c(x,y)-\phi(x).$$ 
Here, $\phi$ is called $c$-convex 
if $\phi^{cc} = \phi$ and each such $\phi$ is lsc. We also note that $\phi^{ccc}= \phi^c$.  
The subdifferential of a $c$-convex $\phi$ is $$\quad \partial_c \phi:=\{(x,y): \phi^c (x) + \phi(y)=c(x,y)\} = \{(x,y): c(x,y)- \phi(x) \geq c(z,y) - \phi(z) \text{ for all $z\in X$} \},$$
and  $\partial_c \phi(x)=\{y:(x,y)\in \partial_c \phi\}$.
The primal and dual transport problem w.r.t.\ $c$ are given by 
\begin{align}\label{eq:OTmaxversion}\text{OT}_{\text{primal}}:=\sup_{\pi \in\cpl(\mu, \nu) } \int c\, d\pi, \quad \text{OT}_{\text{dual}}:= \inf_{\phi\in L^1 (\mu), \phi \text{ lsc}} \int \phi \, d\mu + \int \phi^c \, d\nu.\end{align}
\begin{theorem}[Fundamental theorem of optimal transport] \label{thm:ftot}
    Assume that $c:\Xs\times \Xs\to \R$ is symmetric, continuous and satisfies $|c|\leq a\oplus b$, $a\in L^1 (\mu),b\in L^1(\nu)$. 
The  primal problem is attained and the dual problem is attained by a $c$-convex function $\phi$. Moreover for any such $\phi$ we have that $\pi \in \cpl(\mu, \nu)$ is optimal for the primal problem if and only if $\supp (\pi) \subseteq \partial_c \phi$.
\end{theorem}

In the setting of Brenier's theorem, $c(x, y)= x \cdot y$, $c$-convexity is just ordinary convexity and \Cref{thm:ftot} yields that transport plans are optimal if and only if they are concentrated on the subgradient of a convex function. 

In the case
$\X= \Pc_2(H)$ and $c(\mu, \nu)= \MC(\mu, \nu)$, \Cref{thm:ftot} asserts that transport plans are optimal if and only if they are concentrated on the $\MC$-subgradients of $\MC$-convex functions. 
The link to convex law invariant functions as in \Cref{thm:1-coupling_Brenier} is given in the following result which also appears to be of independent interest: 
\begin{theorem}\label{thm:MCiConvConjIntro}
Let $\phi : \Pc_2^N(\H) \to (-\infty, +\infty]$, $N\geq 1$ be proper. Then
\[
{\overline{\phi\,}}^\ast = \overline{\phi^{\mc}}.
\]
In particular, the convex conjugate of an adapted-law invariant function is adapted-law invariant.
\end{theorem}
Theorem~\ref{thm:MCiConvConjIntro} provides a bridge between 
the abstract conjugates on  $\Pc^N_2(H)$ and classical convex analysis on the Hilbert space 
$L_2([0,1]^N;H)$. In particular, convex functions on $L_2([0,1]^N;H)$ are significantly more tractable 
and amenable to the standard tools of convex analysis, making this identification conceptually 
and technically valuable. 

In the case $N=1$, convexity of the lift $\overline\phi$ is tantamount to convexity of $\phi$ along all curves of the form
 \begin{align}\label{eq:all_interpolations}
     (\mu_t)_{t\in [0,1]}, \mu_t:=((1-t)\proj_0 + t \pr_1)_{\#} (\pi), \pi\in \cpl(\mu, \nu).
 \end{align}
  Functionals with this property are called totally convex \cite{CaSaSo25} or acceleration-free convex \cite{Pa24} or simply convex (in the context of risk measures).
 Total convexity is equivalent to displacement convexity if $\phi:\Pc_2(H)\to \R$ is continuous \cite{CaSaSo25}, but not in general, see \cite{Pa24}. We will show
 \begin{proposition}\label{pro:ConvexEquivIntro} 
For a lsc functional $\phi : \Pc_2(H) \to (-\infty, \infty]$, the following are equivalent:
\begin{enumerate}
    \item $\phi$ is totally convex.
\item $\overline\phi$ is convex.
\item $\phi$ is $MC$-convex.
    \item $\phi$ has totally convex domain, is displacement convex and  increasing in convex order.
   
\end{enumerate}
\end{proposition}
Recall that  $\mu, \nu\in \Pc_2(H)$ are in convex order 
if $\int f\, d\mu \leq \int f\, d\nu$ for all convex $f$.

\subsection{Stochastic processes and  Brenier theorems for the adapted Wasserstein distance}\label{sec:AW_intro}

Authors from several fields have independently refined the weak topology on the laws of stochastic processes in order to take the temporal flow of information into account, see \cite{Al81, He96, HeSc02, BiTa19, PfPi12,PfPi14, NiSu20} among others. In finite discrete time, all these approaches define the same \emph{adapted weak topology} on $\Pc_2(H^N)$ which is metrized by the adapted  Wasserstein distance $\AW_2$, see \cite{BaBaBeEd19b, Pa22, BePaPo21}.

The Monge-formulation of the (squared) adapted Wasserstein distance of $\mu, \nu\in \Pc_2(H)$ is 
\begin{align}
   \AW^2_{2, \text{Monge}} (\mu, \nu) = \inf_{T: H^N  \to H^N, T_{\#}(\mu)=\nu, T \text{ bi-triangular}} \int |T(x)-x|^2 \, d\mu(x).     
\end{align}
Here $T:H^N \to H^N$ is called triangular (or adapted) if each $T_k$ does depend on $(x_1, \ldots, x_n)$ only through the first $k$ coordinates. It is bi-triangular if it is invertible and $T, T^{-1}$ are both triangular. Restricting to triangular / adapted mappings in the definition ensures that the adapted Wasserstein distance adequately respects the information structure inherent in stochastic processes. 

The Kantorovich type relaxation is given by allowing for bi-causal couplings, i.e.\ 
\begin{align}\label{eq:AW_intro_plain}
   \AW^2_2 (\mu, \nu) = \inf_{\pi \in \cplbc(\mu, \nu)} \int |x-y|^2  \, d\pi(x,y).     
\end{align}
 A coupling $\pi$  is called causal if for $t\leq N$ and Borel $B\subseteq H$
\begin{align}\label{eq:AWIntro}
    \pi(Y_t\in B| X_1, \ldots, X_N) = \pi(Y_t\in B| X_1, \ldots, X_t),  
\end{align}
where we use $(X_1,\ldots, X_N), (Y_1, \ldots, Y_N)$ to denote the projections onto the first and second (respectively) coordinate of $H\times H$. The coupling $\pi$ is called bi-causal, if this also holds when switching the roles of $X $ and $Y$.

Already in the first non-trivial instance  $H=\R$, $N=2$ it is easy to find absolutely continuous probability measures which are not transport regular, see \Cref{ex:NoPlainAdaptedBrenier} below. The deeper reason for this  is that via a specific isometry, adapted transport for processes with $N$ time intervals can be seen as a $\W_2$-transport problem on $\Pc_2^{N}(H^N)$ and the corresponding push-forward of the Lebesgue measure to   $\Pc_2^{N}(H^N)$ plays no particular role. 

However based on our results for optimal transport on $\Pc_2^{N}(H^N)$ we are able  to show that  \emph{typically} optimizers of \eqref{eq:AW_intro_plain} are unique and bi-trangular Monge.
\begin{theorem}[Baire--Brenier theorem for $\AW_2$]\label{thm:adaptedBaireBrenier}
    The set of $(\mu,\nu)$ for which $\AW_2(\mu, \nu)$ admits a unique optimal coupling $\pi^*$ and $\pi^*$ is given by a bi-triangular map $T: H^N\to H^N$ is comeager in  $\Pc_2(H^N)\times \Pc_2(H^N)$. 
\end{theorem}
As mentioned above, $\AW_2$ is considered as a natural metric for the adapted weak topology on $\Pc_2(H^N)$ but (like other compatible metrics) it is not  complete.
It is shown in \cite{BaBePa21} that the completion of $(\Pc_2(H^N),\AW_2)$ consists precisely in the space of stochastic processes with \emph{filtration}, where two stochastic processes are identified, in signs $X \simad Y$ if they have the same probabilistic properties. This can be made precise, e.g.\ by asserting that $X$ and $Y$ have Markovian lifts with the same laws, see \cite{BePfSc24}.

To present the completion in terms of $L_2([0,1]^N; H)$, we recall that $(\F_t)_{t=1}^N$ denotes the coordinate filtration on $([0,1]^N,\lambda)$ and set 
\begin{align*}
    L_{2,ad}^N(\H)  := & \  \{ X=(X_t)_{t=1}^N \in L_2([0,1]^N; \H^N ) : X_t \text{ is $\F_t$-measurable}  \}, \\
    \AW_2(X,Y):= & \ \inf \{ \| X'-Y'\|_2 : X'\simad X,Y'\simad Y\}. 
\end{align*}

Elements of $\Pc_2(\H^N)$ correspond to (equivalence classes of) naturally filtered processes $X\in L_{2,ad}^N(\H)$. $X$ is naturally filtered if
    $\E[f(X)|\F_t] = \E[f(X)|X_1, \ldots, X_t]$  
for all $t\leq N$ and $f: \H^N\to \R$ continuous bounded. Intuitively this means that the coordinate filtration $(F_t)_t$ contains no extra information on the process beyond what can be inferred from the past of the process.
Using this embedding, the completion\footnote{The space of  $\text{filtered processes}$ equipped with $ \AW_2$ admits a number of convenient properties. It is a Polish geodesic space,  
the set of martingales is closed and geodesically convex, there is a Prohorov-type compactness criterion, 
 finite state Markov chains are dense and Doob decomposition, optimal stopping, Snell-envelope, pricing, hedging, utility maximization, etc.\  are Lipschitz continuous w.r.t.\ $\AW_2$, cf.\ \cite{BaBePa21} and the references therein.} of  $(\Pc_2(\H^N),\AW_2)$ is given by
$$(L_{2,ad}^N(\H) / \simad, \AW_2)=:(\text{filtered processes},\AW_2).$$

Our main result for the adapted Wasserstein distance is the following characterization of optimal couplings in terms of convex functions on the set of adapted processes. Notice that the case $N=1, \H= \R^d$ corresponds once more to the classical Brenier theorem.  
\begin{theorem}[Brenier Theorem for $\AW_2$]\label{thm:AW_Brenier_Intro}
    Let $X,Y\in L_{2, ad}^N(\H)$. Then $(X,Y)$ is an optimal coupling (i.e.\ $\|X-Y\|_2=\AW_2(X,Y)$) if and only if 
    \begin{align}\label{eq:AdOptIntro} 
    Y(u_1, \ldots) \in \partial \overline{\phi}( X(u_1, \ldots)),\quad u_1\in [0,1]
    \end{align}
    for some convex adapted-law invariant $\overline{\phi}: \H\times L_{2, ad}^{N-1}(H)\to (-\infty, \infty]$.

    Moreover, a coupling $\pi \in \cplbc(\mu, \nu)$, $\mu, \nu\in \Pc_2(H^N)$ is optimal if and only if $\pi \sim \law(X,Y)$ for some naturally filtered processes $X,Y$ satisfying \eqref{eq:AdOptIntro}. 
\end{theorem}
Importantly $(\Pc_2(\H^N),\AW_2)$ is a dense $G_\delta$ subset of its completion  $(\text{filtered processes},\AW_2)$, see \cite{Ed19, BaBeEdPi17}.  This is crucial to deduce \Cref{thm:adaptedBaireBrenier} from \Cref{thm:AW_Brenier_Intro}  and ultimately goes back to the fact that Monge couplings are a $G_\delta$ subset of the set of couplings with fixed first marginal, see \cite{MeZh84} and \cite[IV, 43]{DeMeB}.  

\subsection{Related literature}\label{sec:intro_literature}

Starting from the classical Monge and Kantorovich formulations, the modern theory of optimal transport  originated around Brenier’s discovery that, for quadratic costs on Euclidean space, optimal plans are gradients of convex potentials \cite{Br91}, see also \cite{KnSm84, RaRu90}. 
In particular it paved the way for McCann’s displacement convexity \cite{Mc94} and the interpretation of many PDEs as gradient flows on the Wasserstein space. These developments were systematized in the metric-measure framework and gradient-flow theory of Ambrosio--Gigli--Savaré \cite{AmGiSa08,AmGi13} and in Villani’s monographs \cite{Vi03,Vi09}, with more recent expositions by Santambrogio \cite{Sa15} and Figalli--Glaudo \cite{FiGl21}. Beyond the Euclidean Brenier setting, there is a rich line of extensions and refinements --- covering structure/uniqueness of optimal maps, displacement convexity and convex-order tools, and regularity/duality on general spaces -- see, among others  \cite{GaMc96} (general convex costs), \cite{Mc95} (Brenier/polar factorization on manifolds), \cite{GaSw98, AgCa11}  (multi-marginal transport and barycenters),
\cite{FeUs02} (Monge transport on abstract Wiener spaces via Cameron-Martin convexity, existence/uniqueness of maps) and \cite{FeUs04a} (infinite-dimensional Monge-Ampère on Wiener space, regularity/structure of the transport)\cite{AmKiPr04} (Monge maps for strictly convex costs), \cite{AmGiSa08} (Brenier theorem for Gaussian on Hilbert spaces), \cite{ChDe10, ChPa2011} (existence of Monge optimizers under low regularity), \cite{Gi11} (convex potential beyond smooth Euclidean settings), \cite{BeCoHu14} (Skorokhod embedding), \cite{BeJu16, HeTo13, GhKiLi16b} (martingale transport), \cite{CaGaSa09} (connection to Knothe--Rosenblatt coupling), \cite{GoJu18} (weak optimal transport). 

In optimal transport and mean field games, the ``Lions’ lift'' refers to the idea of studying maps defined on the Wasserstein space of probability measures by lifting them to functions on a Hilbert space of square-integrable random variables.
This was introduced by Lions in his Collège de France lectures \cite{Li13} and became widely known through Cardaliaguet’s lecture notes \cite{Ca13} and monographs of Carmona--Delarue’s monographs \cite{CaDe18, CaDe18b} which employ the Lions derivative as an essential tool in the theory of mean field games.

Emami and Pass \cite{EmPa25} were the first to establish existence and uniqueness of Monge solutions for measures on $\Pc_2(M)$ and  $\W_2^2$-costs. Emami and Pass consider a smooth Riemannian manifold $M$ and use
 structural assumptions on a reference measure on $\mathcal{P}_2(M)$. A central ingredient in their analysis is Dello Schiavo’s Rademacher theorem on Wasserstein spaces \cite{De20}, that allows a classical optimal-map strategy to be carried out on the space of measures. Notably it is challenging to construct measures satisfying Dello Schiavo's hypothesis together with the absolute continuity hypothesis necessary for the main Monge result of \cite{EmPa25} and examples are only known in the case where $M$ is the one dimensional torus.

Independently and in parallel with an earlier version of this manuscript, Pinzi--Savaré \cite{PiSa25}, see also \cite{PiSa25b,Pi25}, analyzed the case $N=1$ of measures on $\mathcal{P}_2(H)$ for separable Hilbert spaces $H$, developing a Brenier-type theory via totally convex functionals and their Lagrangian/Lions lift, and identifying natural classes of full-support, transport-regular laws for which the Monge formulation is uniquely solved. Our initial preprint \cite{BePaSc25}, posted the same day their preprint appeared, likewise treated $N=1$ but focused on $\mathbb{R}^d$; the restriction to $\mathbb{R}^d$ was  due to the then-open problem of constructing a transport-regular reference measure on infinite-dimensional Hilbert spaces, while all other arguments were dimension-agnostic. The present version advances in two directions: it covers arbitrary $N$ and it works on general separable Hilbert spaces. The extension relies on an adapted Lions lift tailored to the $N$-step structure. In particular, we construct a transport-regular measure $\Lambda$ on $\mathcal{P}_2(H)$ and, by induction, on $\mathcal{P}_2^{N}(H)$; here it is  essential to include the case of infinite dimensional $H$ even to establish the $N$-level result in $\mathbb{R}^d$ for $N>1$, since the induction from lower levels requires the Hilbert-space setting. We note that the construction of $\Lambda$ as the occupation measure of a Wiener process (resp.\ based on the Wiener sheet for $N>1$) is original to this paper to the best of our knowledge.

A main motivation for this article was to establish a Brenier theorem for the adapted Wasserstein distance, which requires to understand the $\W_2$-optimizers on $\mathcal{P}_2^{\,N}(H)$ for general $N$. As noted above, 
 the adapted variant of Wasserstein distance is useful when measuring the distance between stochastic processes since it  accounts for the inherent information structure of the processes. It has  applications from stochastic optimization, stochastic control, and mathematical finance to the theory of geometric inequalities and machine learning; see \cite{PfPi14, La18, XuAc21, BaBePa21, AcKraPa24, CoLi24, JiOb24, BaWi23, AcKrPa25} and the references therein.

\subsection{Organization of the paper}

\Cref{sec:LionsRecap} recalls the Lions lift and basic continuity properties. We identify $c$-conjugation for $c=\mc$ with classical convex conjugation on the lift (\Cref{thm:MCiConvConj}) and relate total/Lions convexity, $\mc$-convexity, and subdifferentials.

\Cref{sec:DensenessOfTransportReg} treats the case of optimal transport between measures on $\Pc_2(H)$ ($N=1$). We characterize $\W_2^2$-optimizers on $\Pc_2(H)$ via convex, law-invariant potentials on the lift (\Cref{thm:char.optimal_cpls}) and construct transport-regular reference laws: on $\R^d$ using Brownian sheets (\Cref{lem:BMDifferentiable}, \Cref{thm:2nd.TR.dense}) and on separable Hilbert spaces using $Q$-Wiener processes (\Cref{thm:Q-Wiener.TR}).

\Cref{sec:PolishnessOfRegularMeasures} develops a Baire-category viewpoint: we quantify non-regularity via $\tau^c$, prove upper semicontinuity, and deduce that $c$-transport-regular laws form a $G_\delta$ subset of $\Pc_p(\X)$ (\Cref{thm:TransportRegularPolish}).

\Cref{sec:OT_On_P_N} introduces the iterated max-covariance $\mc$ on $\Pc_2^{N}(H)$ and the characterization of $\mc$-optimal couplings in terms of $\mc$-convex potentials. We also derive a DPP which will be necessary to establish the characterization of $\mc$-convex potentials in terms of convex conjugation in \Cref{sec:MC_cx_ad}.

In \Cref{sec:AdLift} we provide a Hilbertspace  representation of laws on $\Pc_2^{N}(H)$, an adapted transfer principle, and a Skorokhod-type representation (\Cref{prop:ReprLawadAsRV}, \Cref{cor:transfer_ad}, \Cref{prop:iteratedSkorohod}).

\Cref{sec:MC_cx_ad} identifies $\mc$-conjugation with the convex conjugate of the adapted lift (\Cref{thm:MC_trafo_ast_ad}), establishes the $\mc$-order which is used to establish the existence of the Lions derivative in the adapted setting, and links $\mc$-subdifferentials with subdifferentials on the lift.

\Cref{sec:TR} constructs full-support, transport-regular measures on $\Pc_2^{N}(H)$; this yields the $N$-level Brenier theorems stated in the introduction (Theorems~\ref{thm:1-coupling_Brenier}, \ref{Thm:MainMongeIntro} and \ref{thm:N-coupling_Brenier} and Corollary~\ref{cor:MainMongeIntro}).

\Cref{sec:AW} applies the theory to the adapted Wasserstein distance $\AW_2$: via the adapted law isometry we obtain a Brenier theorem for $\AW_2$ and a Baire-Brenier uniqueness/Monge statement (\Cref{thm:AW_Brenier_Intro}, \Cref{thm:adaptedBaireBrenier}).

Technical measurability, selection, and auxiliary results are collected in the appendix.

\tableofcontents

\section{Lions lift and convexity of law invariant functions} \label{sec:LionsRecap}

\subsection{Recap on Lions lift and Lions derivative}

The aim of this section is to recall the definition of the Lions lift of a function $\phi : \Pc_2(\H) \to (-\infty, \infty]$ and to provide properties of the lift that we will need subsequently.
\begin{definition}

Given a probability space $(\Omega,\F,\P)$,
the Lions lift of a function $\phi : \Pc_2(\H) \to (-\infty, +\infty]$ is defined as 
\[
\overline{\phi} : L_2(\Omega,\F,\P;\H) \to (-\infty, +\infty] , \quad \overline{\phi}(X) := \phi(\law(X)).
\]
\end{definition}
We say that a function $\overline{\phi} : L_2(\Omega,\F,\P) \to (-\infty , +\infty]$ is law-invariant if $\overline{\phi}(X) = \overline{\phi}(X')$ whenever $X \sim X'$. Clearly, law-invariant functions are precisely those which arise as Lions lift of a function  $\phi : \Pc_2(\H) \to (-\infty, +\infty]$. Recall that $(\Omega,\F,\P)$ is a standard probability space if it is isomorphic to the unit interval with Lebesgue measure.

\begin{tcolorbox}
In the following we will always assume that $(\Omega,\F,\P)$ is a standard probability space and given a function $\phi : \Pc_2(\H) \to (-\infty, +\infty]$, the function $\overline{\phi} : L_2(\Omega,\F,\P;\H) \to (-\infty, +\infty]$ always denotes its Lions lift. We use $X,X', Y,Z$ to denote elements of $L_2(\Omega, \F, \P;\H)$.
\end{tcolorbox}

Given a function $f $ with values in $(-\infty, \infty]$ we write, $\dom (f) := f^{-1}((-\infty, \infty)) $ for its domain, $\cont (f)\subseteq \dom (f) $ for its continuity points and call $f$ proper if $\dom (f) \neq \emptyset$. Clearly we have
\[
\dom(\overline{\phi}) = \{ X \in L_2(\Omega,\F,\P;\H) : \law(X) \in \dom(\phi) \}. 
\]

\begin{lemma}\label{lem:contequiv}
Let $\phi : \Pc_2(\H) \to (-\infty, +\infty]$. Then $\phi$ is (lower semi) continuous wrt $\W_2$ if and only if $\overline{\phi}$ is (lower semi) wrt $\| \cdot \|_2$. 
\end{lemma}
\begin{proof}
This follows from the continuity of the law map and the Skorokhod representation theorem.     
\end{proof}

\begin{lemma}\label{lem:simple_transfer}
    Let $X,X'$ be random variables on $(\Omega,\F,\P)$  such that $X \sim X'$.
    Then, for every $\epsilon > 0$, there exists a measure-preserving bijection $T : \Omega \to \Omega$ such that $\mathbb P(|X' - X \circ T| \le \epsilon) \ge 1 - \epsilon$.
\end{lemma}

See  \cite[Lemma 5.23]{CaDe18} for a  stronger result which pertains to $L^\infty$-convergence instead of convergence in probability.

\begin{proof}[Proof of \Cref{lem:simple_transfer}]
    Since $(\Omega,\F,\P)$ is a standard probability space, we may assume w.l.o.g.\ that  $\Omega=(0,1)$, $\F$ are the Borel sets  and $\P$ is the Lebesgue measure $\lambda$. Set $U(x)=x$ for $x\in (0,1)$.
    Then $ \mu := \law(U,X)$ and $\nu := \law(U,X')$ are atomless distributions and both are concentrated on the graph of a function. By assumption, $\mu, \nu$ have the same marginals, hence,
    \[
        \inf_{\pi \in \cpl(\mu,\nu)} \int |x-x'| \wedge 1 \, d\pi((x,u),(x',u')) = 0.
    \]
    By \cite{Ga99}, the transport plans supported on Borel bijections are dense in $\cpl(\mu, \nu)$. Thus there exists a sequence of Borel bijections $(R_n,S_n) :  (0,1)\times \R \to (0,1) \times \R$, $n\geq 1$ such that $(R_n, S_n)_\# \mu = \nu$ and $S_n(u,x)\to x$ in $\hat \mu$-probability. The second assertions means that $S_n(\omega, X(\omega)) \to  X(\omega)$ in $\lambda$-probability. On the other hand $(R_n, S_n)_\#  \mu =  \nu$ yields that $S_n(\omega, X(\omega)) = X'(R_n(\omega, X(\omega)))$
    $\lambda$-a.s.\ and $\omega \mapsto R_n(\omega,X(\omega))$ is a measure preserving bijection. 
\end{proof}

\begin{lemma}\label{lem:epsTransfer}
Let $X,Y$ and $X'$ be random variables on $(\Omega,\F,\P)$ such that $X \sim X'$. For $\epsilon>0$, there are random variables $X'', Y''$ on $(\Omega,\F,\P)$ such that $(X,Y) \sim (X'',Y'')$ and $\P(|X'-X''| \ge \epsilon) < \epsilon$.
\end{lemma}
\begin{proof}
By \Cref{lem:simple_transfer}, there is measure preserving bijection $T$ such that $\P(|X \circ T - X'| \ge \epsilon) \le \epsilon$. We set $X'':= X \circ T$ and $Y'':= Y \circ T$. 
\end{proof}

\subsection{$\mc$-convexity}
In this section we consider functionals on  $(\Pc_2(\H),\W_2)$. Specifically, we  show that the notion of \emph{$L$-convexity} (see \cite{CaDe18}) can be linked to a notion of convex conjugation with respect to the \emph{max-covariance functional} on $(\Pc_2(\H),\W_2)$. Moreover, we establish a connection between the $\mc$-transform and the convex conjugate on $L_2(\Omega,\F,\P)$ and investigate the connection between subdifferentials of the lifts and $\mc$-subdifferentials.

\begin{definition}
$\phi : \Pc_2(\H) \to (-\infty, +\infty]$ is called L-convex / totally convex, if its Lions lift $\overline{\phi}$ is convex.
\end{definition}
Equivalently,
 $\phi : \Pc_2(\H) \to (-\infty, +\infty]$ is totally convex  if for all $\mu_0, \mu_1 \in \Pc_2(\H)$, $\pi \in \cpl(\mu_0,\mu_1)$, and $t \in [0,1]$, we have
 \[
  \phi( ((x,y) \mapsto (1-t)x+ty)_\# \pi ) \le (1-t) \phi(\mu_0) +t \phi(\mu_1).
 \]

\begin{proposition}\label{prop:ConvexEquiv} 
Let $\phi : \Pc_2(\H) \to (-\infty, +\infty]$ be proper and lsc. Then the following are equivalent:
\begin{enumerate}
    \item $\phi$ is totally convex,
    \item $\overline{\phi}$ is convex,
    \item $\phi$ is MC-convex,
    \item $\phi$ has totally convex domain, is displacement convex and increasing in convex order.
\end{enumerate}
\end{proposition}

Here, we call a set $A \subset \Pc_2(\H)$ totally convex / $\mc$-convex if for every $\mu, \nu \in A$, every $\pi \in \cpl(\mu,\nu)$ and every $t \in [0,1]$, we have $( (x,y) \mapsto ((1-t)x+ty) )_\ast \pi \in A$. Clearly, this is equivalent to the fact the convex indicator function $\chi_A$ (i.e.\ $\chi_A(\mu) = 0$ if $\mu \in A$ and $\chi_A(\mu) = + \infty$ otherwise) is totally convex / MC-convex. 

\medskip

The following lemma is a straightforward consequence of the denseness of Monge couplings in the set of Kantorovich couplings in the case of continuous starting distributions. 
\begin{lemma}\label{lem:EnoughRandomization}
    Let $c:  \H \times   \H \to \R $  be continuous, $|c|\leq a\oplus b, a\in L^1 (\mu), b\in L^1(\nu)$ and assume that $X\sim \mu$. Then we have
    \begin{align}\label{eq:lem:EnoughRandomization}
        \sup_{\pi\in \cpl(\mu, \nu) } \int c\, d\pi = \sup_{Y\sim \nu} \E[c(X,Y)].
    \end{align}
\end{lemma}
\begin{proof}
As for every $Y \sim \nu$, we have $\law(X,Y) \in \cpl(\mu,\nu)$, we clearly have $\sup_{\pi\in \cpl(\mu, \nu) } \int c\, d\pi \ge \sup_{Y\sim \nu} \E[c(X,Y)].$ To see the converse inequality, fix $\pi \in \cpl(\mu,\nu)$.
Since $(\Omega,\F,\P)$ is standard and atomless, it supports random variables $X',Y'$ with $(X',Y') \sim \pi$. In particular, we have that $X' \sim \mu$.
Therefore, we can apply Lemma \ref{lem:epsTransfer} to find, for each $\epsilon > 0$, random variables $X_\epsilon,Y_\epsilon$ with $\mathbb P(|X-X_\epsilon| \ge \epsilon) < \epsilon$ and $(X_\epsilon,Y_\epsilon) \sim \pi$. 
By dominated convergence we conclude that
\[
    \lim_{\epsilon\searrow 0} \E[c(X,Y_\epsilon)] = \E[c(X',Y')] = \int c \, d\pi. \qedhere
\]

\end{proof}

The following observation plays a fundamental role for this paper.
\begin{theorem}\label{thm:MCiConvConj}
Let $\phi : \Pc_2(\H) \to (-\infty, +\infty]$ be proper. Then
\[
{\overline{\phi\,}}^\ast = \overline{\phi^{\mc}}.
\]
In particular, the convex conjugate of a law-invariant function is law-invariant. 
\end{theorem}
\begin{proof}
For $Y \in L_2(\Omega,\F,\P;H)$ we have by Lemma~\ref{lem:EnoughRandomization}
\begin{align*}
   \overline{\phi\,}^\ast(Y) &= \sup_{X \in L_2 } \E[ X \cdot Y ] - \overline{\phi}(X) = \sup_{ \mu \in \Pc_2(\H)}  \sup_{X \sim \mu} \E[X \cdot Y]   -  \phi(\mu) \\
   &= \sup_{\mu \in \Pc_2(\H) } \mc(\mu,\law(Y)) - \phi(\mu)  = \phi^{\mc}(\law(Y)) = \overline{\phi^{\mc}}(Y). \qedhere
\end{align*}
\end{proof}

For the sake of completeness we prove that $\mc(\cdot,\nu)$ is totally convex. Note that the argument given here is as in \cite[Theorem~7.3.2]{AmGiSa08}.
\begin{lemma}\label{lem:MCtotcvx}
For every $\nu \in \Pc_2(\H)$, the functional $\mc(\cdot,\nu)$ is totally convex.  
\end{lemma}
\begin{proof}
Let $\mu^0,\mu^1 \in \Pc_2(\H)$,  $\pi \in \cpl(\mu^0,\mu^1)$ and $t \in [0,1]$ be given. We write   $\mu^t = ( (x,y) \mapsto (1-t)x+ty )_\# \pi$. By conditional independent gluing we find  $\Pi \in \cpl(\pi, \nu)$ such that $(  (x,y,z) \mapsto ((1-t)x+ty,z) )_\# \Pi \in \cpl(\mu^t,\nu)$ is optimal. We have 
\begin{align*}
\mc(\mu^t,\nu) &= \int   ((1-t)x^0+tx^1) \cdot y \, \Pi(dx^0,dx^1,dy) \\&= (1-t) \int  x^0 \cdot y \,  \Pi(dx^0,dy) + t \int  x^1 \cdot y \,  \Pi(dx^1,dy) 
\le  (1-t)\mc(\mu^0,\nu) +t \mc(\mu^1,\nu). \qedhere
\end{align*}
\end{proof}

\begin{proof}[Proof of \Cref{prop:ConvexEquiv}]
The equivalence of (1) and (2) is trivial.

To see that (2) implies (3) note that if $\overline{\phi}$ is convex, then by the Fenchel--Moreau theorem (see e.g.\ \cite[Theorem 13.32]{BaCo11}) and \Cref{thm:MCiConvConj}, 
\[
\overline{\phi} = \overline{\phi}^{\ast\ast} = \overline{\phi^{\mc}}^\ast = \overline{\phi^{\mc\mc}}.
\]
Hence, $\phi = \phi^{\mc\mc}$, so $\phi$ is MC-convex.

(3) implies (2) because MC is totally convex (see \Cref{lem:MCtotcvx}) and this property is preserved by suprema.

(3) implies (4) is true because MC is displacement convex and increasing in convex order (this was established by Carlier \cite{Ca08} in the case $H=\R^N$, see \Cref{prop:MC_order:_char} for the general case) and these properties are preserved under suprema.

We show the implication  (4) implies (1) under the additional assumption that $\phi$ is finitely valued. The general case (which is not needed for the other results of this paper) is deferred to \cite{BePaScSo25}.

Assume first that the dimension of $\H$ is strictly larger than one. Note that as $\phi$ is lower semi-continuous and increasing in convex order, we have for $\mu \in \Pc_2(\H)$ that
\[
    \liminf_{n \to \infty} \phi(\mu_n) \ge \phi(\mu) \ge \lim_{n \to \infty} \phi(\mu_n),
\]
for every sequence $(\mu_n)_{n \in \N}$ with $\mu_n \le_{cx} \mu$ and $\lim_{n \to \infty} \mu_n = \mu$ in $\Pc_2(\H)$.
Hence, for each such sequence we have $\lim_{n \to \infty} \phi(\mu_n) = \phi(\mu)$.
Therefore, the claim follows from \cite[Proof of Theorem~9.1]{CaSaSo25}.

If $\H = \R$, we fix $\mu,\nu \in \Pc_2(\R)$, $\pi \in \cpl(\mu,\nu)$ and $t \in (0,1)$.
Denote by $Q_\mu$, $Q_\nu$, and $Q$ the quantile functions of $\mu$, $\nu$, and $\mu_t := ((x,y) \mapsto x(1-t) + yt )_\# \pi$, respectively.
Let $U,V_1,V_2\sim {\rm Unif}([0,1])$ with $(Q_\mu(V_1), Q_\nu(V_2)) \sim \pi$ and 
\[
    Q(U) = (1-t) Q_\mu(V_1) + t Q_\nu(V_2).
\]
Clearly, we have for $u \in (0,1)$ that $\law(Q(U) | U \le u)$ dominates $\law((1-t)Q_\mu(U) + tQ_\nu(U)|U \le u)$ in stochastic order and therefore
\[
    \mathbb E[Q(U) | U \le u] \le
    \mathbb E[(1-t)Q_\mu(U) + t Q_\nu(U) | U \le u],
\]
from where we deduce that $\mu_t \le_{cx} \nu_t := \law( (1-t) Q_\mu(U) + t Q_\nu(U))$.
Since $\phi$ is geodesically convex and increasing in convex order, we get
\[
    \phi(\mu_t) \le \phi(\nu_t) \le (1-t)\phi(\mu) + t \phi(\nu).
\]
We conclude that $\phi$ is totally convex.
\end{proof}

In convex analysis on Hilbert spaces, the smallest convex functions (wrt pointwise order) are linear functions and the biggest are convex indicator functions of singletons. The convex conjugation is order reversing and it maps linear functions to indicators of singletons. The analogue of this in the present setting is the following:

\begin{example}\label{ex:MCMC}
Let $\nu \in \Pc_2(\H)$ and consider the functional $\phi(\mu) = \mc(\mu,\nu)$. We then have $\phi^\mc = \chi_{ \{ \rho \in \Pc_2(\H) : \rho \le_c \nu \}  }$, where $\le_c$ denotes the convex order of probability measures.
\end{example}
\begin{proof}
We have 
\[
\phi^\mc(\rho) = \sup_{ \mu \in \Pc_2(\H)} \mc(\mu,\rho) - \mc(\mu,\nu) .
\]
By \cite{Ca08, WiZh22} we have that $\rho \le_c \nu$ if and only if $\mc(\mu,\rho) \le \mc(\mu,\nu)$ for every $\mu \in \Pc_2(\H)$. Hence, if $\rho \le_c \nu$, setting $\mu=\nu$ in the supremum, we find $\phi^\mc(\rho)=0$. Otherwise there exists $\mu \in \Pc_2(\H)$ such that $\mc(\mu,\rho) - \mc(\mu,\nu) >0$.  As $\mc({s_\lambda}_\#\mu,\rho) = \lambda \mc(\mu,\rho)$, where $s_\lambda(x) = \lambda x$ for $\lambda >0$, this yields that $\phi^{\mc}(\rho) = +\infty$ whenever $\rho$ in not smaller in convex order than $\nu$.
\end{proof}

\begin{example}
Let $\nu \in \Pc_2(\H)$. Then the closed $\mc$-convex hull of $\{\nu\}$ is given by $\{ \rho \in \Pc_2(\H) : \rho \le_c \nu \}$. In particular, singletons are $\mc$-convex sets if and only if they consist of a Dirac measure.  
\end{example}
\begin{proof}
Clearly, $\chi_{\{\nu\}}^\mc = \mc(\cdot, \nu)$. Hence $  \chi_{\{\nu\}}^{\mc\mc} = \mc(\cdot,\nu)^\mc =  \chi_{ \{ \rho \in \Pc_2(\H) : \rho \le_c \nu \}  } $ by \Cref{ex:MCMC}.
\end{proof}

It is well known (see \cite[Example 7.3.3]{AmGiSa08}) that the Wasserstein-2 distance is not displacement convex. However, the distance to an MC-convex set is MC-convex (and hence displacement) convex. Moreover, this convexity property characterizes $\mc$-convex sets. 

\begin{proposition}
Let $A \subset \Pc_2(\H)$ be closed. Then $A$ is $\mc$-convex if and only if $ \mu \mapsto \textup{dist}_{\W_2}(\mu,A) = \inf_{\nu \in A} \W_2(\mu,\nu)$ is $\mc$-convex.    
\end{proposition}
\begin{proof}
Writing $B = \{ X : \law(X) \in A\}$ we have by \Cref{lem:EnoughRandomization}
\[
\overline{\textup{dist}_{\W_2}(\cdot,A)}(X) = \inf_{\nu \in A} \W_2(\law(X), \nu) = \inf_{Y \in B} \| X- Y \|_2 = \textup{dist}_{\|\cdot\|_2}(X,B).
\]
In Hilbert spaces a closed set is convex if and only if its distance function is convex.
\end{proof}

\subsection{Subdifferentials}

Next, we spell out the definition of the $\mc$-subdifferential, which is the definition of the $c$-superdifferential from optimal transport applied to $c=\mc$.
\begin{definition}\label{def:MCsubdiff}
Let $\phi : \Pc_2(\H) \to (-\infty, +\infty]$. Then 
\[
\partial_{\mc} \phi(\mu) = \{ \nu \in \Pc_2(\H) : \phi(\rho) \ge \phi(\mu) + \mc(\rho,\nu) - \mc(\mu,\nu) \text{ for every } \rho \in \Pc_2(\H)\}.
\]
\end{definition}

While the subdifferential of a convex function on a vector space is in every point a convex set, this is not the case for $\mc$-subdifferentials of $\mc$-convex sets. In fact, they are typically not $\mc$-convex because they are often singletons (see differentiability results below) and singletons are not $\mc$-convex sets (unless they consist of a Dirac).

Next, we investigate the connection between the $\mc$-differential and subdifferential of the Lions lift. We start with a characterization of when random variables $(X,Z)$ are in the subdifferential of the Lions lift in terms of the $\mc$-subdifferential.

\begin{proposition}\label{prop:Subdiff_MCvsL}
Let $\phi: \Pc_2(\H) \to (-\infty, +\infty]$ be L-convex and $X,Z \in L_2(\Omega,\F,\P;\H)$ with $X\sim \mu$, $Z\sim \nu$. Then we have
\begin{align*}
    Z  \in  \partial \overline{\phi}(X) \iff \nu \in \partial_{\mc}\phi(\mu) \text{ and } \law(X,Z) \in \cplopt(\mu,\nu).
\end{align*}
\end{proposition}
\begin{proof}
Let   $Z  \in  \partial \overline{\phi}(X)$ and write $\mu = \law(X)$, $\nu= \law(Z)$. By definition we have 
\begin{align}\label{eq:prf:subdif_coincide}
\overline{\phi}(Y) \ge \overline{\phi}(X) + \E[(Y-X)\cdot Z].    
\end{align}
For every $Y \sim \mu$, we have $\overline{\phi}(X) = \overline{\phi}(Y)$ and   \eqref{eq:prf:subdif_coincide} yields $\E[X \cdot Z] \ge \E [ Y \cdot Z]$. Hence, $\E[X \cdot Z] = \mc(\mu,\nu)$ by \Cref{lem:EnoughRandomization}. Therefore $\law (X,Y)\in \cplopt(\mu, \nu)$. 

Next, fix $\rho \in \Pc_2(\H)$ and $\varepsilon >0$. By \Cref{lem:EnoughRandomization} there is $Y \sim \rho$ such that $\E[Y \cdot Z] \ge \mc(\rho,\nu) - \varepsilon$. Applying \eqref{eq:prf:subdif_coincide} then yields $\phi(\rho) \ge \phi(\mu) + \mc(\rho,\nu) - \mc(\mu,\nu) - \varepsilon$. Thus $\nu\in \partial_{\mc} \phi(\mu)$.

\medskip

Conversely, assume that $\nu \in \partial_\mc\phi(\mu)$ and that $X \sim \mu$, $Z \sim \nu$ are satisfying $\E[X \cdot Z] = \mc(\mu,\nu)$. Then for every $Y$, we have for $\rho = \law(Y)$
\[
\overline{\phi}(Y) = \phi(\rho) \ge \phi(\mu) + \mc(\rho,\nu) - \mc(\mu,\nu) \ge \overline{\phi}(X) + \E[ Y \cdot Z] - \E[X \cdot Z].
\]
Hence, $ Z  \in  \partial \overline{\phi}(X)$.
\end{proof}

\begin{lemma}\label{lem:subdiff_cond}
Let $\overline{\phi}$ be L-convex. If $Z \in \partial \overline{\phi}(X)$, then $\E[Z|X] \in \partial \overline{\phi}(X)$.
\end{lemma}
\begin{proof}
Using that $\overline{\phi}$ is monotone in the convex order (see \Cref{prop:ConvexEquiv}) we find for every $Y$
\[
\overline{\phi}(Y) \ge \overline{\phi}( \E[Y|X] ) \ge \overline{\phi}(X) + \E[ (\E[Y|X]-X) \cdot Z ] = \overline{\phi}(X) + \E[ (Y-X) \cdot \E[ Z|X] ]. \qedhere
\]
\end{proof}

Note that \Cref{prop:Subdiff_MCvsL} implies that the subdifferential of law-invariant convex functions consists of optimal couplings, i.e.\ if $(X,Z) \in \partial\overline{\phi}$ then $\law(X,Z)$ is optimal between its marginals. In a certain sense the subdifferential of the Lions lift contains more information than the $\mc$-subdifferential: It consists of optimal $(X,Z)$ whereas the $\mc$-subdifferential carries only information on the marginals $\law(X)$ and $\law(Z)$. 

\begin{definition}\label{def:IndRand}
    We say that a random variable $X\in L_2(\Omega, \F, \P;\H)$ allows for independent randomization if there exists a uniformly distributed $U\in L_2(\Omega, \F, \P;\R)$ which is independent of $X$.
\end{definition}
Apparently not every random variable allows for independent randomization, e.g.\ if $X$  is one-to-one, then it can not allow for independent randomization. 
The following lemma says $X$ does allow for independent randomization if it is not at all one-to-one.
\begin{lemma}\label{lem:IndRand=Not1-1}
   Let  $X\in L_2(\Omega, \F, \P;\H)$ set $\rho:=(\omega \mapsto (\omega, X(\omega))_{\#}(\P)$ and write $(\rho^x)_x$ for the disintegration of $\rho$ w.r.t.\ $(\proj_{\H})_{\#}(\rho)= \law_{\P}(X)$. Then the following are equivalent: 
   \begin{enumerate}
       \item $X$ allows for independent randomization.
       \item The measure $\rho^x$ is continuous for $\law_{\P}(X)$-a.e.\ $x$. 
   \end{enumerate}
\end{lemma}
\begin{proof}
    Straightforward.
\end{proof}

\begin{corollary}\label{cor:MCdiffToLdiff}
Let $\phi: \Pc_2(\H) \to (-\infty, +\infty]$ be L-convex and $\overline{\phi} : L_2(\Omega,\F,\P;\H) \to (-\infty, +\infty]$ be its Lions lift. Suppose that $\rho \in \partial_\mc\phi (\mu)$ and let $X \sim \mu$ allow for independent randomization. Then there is $Z \in \partial\overline{\phi}(X)$ satisfying $Z \sim \rho$. 
\end{corollary}
\begin{proof}
Pick some $\pi \in \cplopt(\mu,\rho)$. As $X$ allows for independent randomization, there is a random variable $Z$ such that $(X,Z) \sim \pi$. By \Cref{prop:Subdiff_MCvsL} we have $Z \in \partial\overline{\phi}(X)$.
\end{proof}

Note that in \Cref{cor:MCdiffToLdiff} it is crucial that $X$ allows for independent randomization. This fact turns out to be an obstacle for the law-invariance of the subdifferential. However, it turns out that the subdifferential is law-invariant as long as we restrict to random variables $X$ that allow for independent randomization. Moreover, we can characterize those couplings that can be transferred to any $X \sim \mu$ as the optimal Monge couplings between $\mu$ and elements of $\partial_\mc\phi(\mu)$.

\begin{proposition}\label{cor:LawInvSubDiffRandom}
Let $\phi: \Pc_2(\H) \to (-\infty, +\infty]$ be L-convex and $\mu\in \Pc_2(\H)$.  

For all $X\sim \mu$ we have $$ \{ \law(Y) : Y\in \partial \overline\phi(X) \} \subseteq \phi^\MC (\mu).$$
Assume further that one of the following holds:
\begin{enumerate}
    \item $X$ admits independent randomization.
    \item $\mu$ is transport-regular.
\end{enumerate}
Then we have
$$ \{ \law(Y) : Y\in \partial \overline\phi(X) \} = \phi^\MC (\mu).$$
\end{proposition}
\begin{proof}
The first inclusion  follows immediately from \Cref{prop:Subdiff_MCvsL}.

By \Cref{cor:MCdiffToLdiff}, we have equality provided that $X$ admits independent randomization. 

If $\mu$ is transport-regular, then for every $\nu\in\phi^\MC (\mu)$, $\cplopt(\mu,\nu)=\{\law(X,T(X))\} $ for some map $T$. Hence we obtain again equality.   
\end{proof}

\begin{definition}\label{def:MCdiff}
    Let $\phi:\Pc_2(\H)\to (-\infty, \infty]$ be $\MC$-convex and $\mu\in \dom(\phi)$. We say that $\phi $ is $\MC$-differentiable at $\mu$ if $|\partial \phi(\mu)|=1$.
\end{definition}

\begin{lemma}\label{lem:mc_diff_char}
    Let $\phi:\Pc_2(\H)\to (-\infty, \infty]$ be $\MC$-convex and $\mu\in \dom(\phi)$. Then the following are equivalent: 
    \begin{enumerate}
        \item $\phi$ is $\MC$-differentiable in $\mu$.
        \item For all $X$ with $X\sim \mu$ we have $|\partial \overline \phi (X)|=1$.
        \item There exists $X\sim \mu$  which admits independent randomization and satisfies $|\partial \overline \phi (X)|=1$.
    \end{enumerate}
\end{lemma}

\begin{proof}
(1) implies (2): Suppose that there are $Z_1 \neq Z_2 \in \partial\overline{\phi}(X)$ for some $X \sim \mu$. By \Cref{prop:Subdiff_MCvsL}, we have $\law(Z_1), \law(Z_2) \in \partial_\mc\phi(\mu)$. Hence, if $\law(Z_1) \neq \law(Z_2)$, we conclude that $\partial_\mc\phi(\mu)$ is not a singleton. Otherwise, observe that $Z:= \frac12 (Z_1+Z_2) \in \partial\overline{\phi}(X)$ and hence $\law(Z) \in \partial_\mc\phi(\mu)$. If $f : \H \to \R$ is any strictly convex function, we have $\E[f(Z)] < \frac12 (\E[f(Z_1)]+\E[f(Z_2)]) = \E[f(Z_1)]$. Hence, $\law(Z) \neq \law(Z_1)$ showing that $\partial_\mc\phi(\mu)$ is no singleton.

(2) implies (3) is trivial.

(3) implies (1): Suppose that there are $ \nu_1 \neq \nu_2 \in \partial_{\mc}\phi(\mu)$. Let $X \sim \mu$ allow for independent randomization. Then there are $Z_1, Z_2 \in \partial\overline{\phi}(X)$ satisfying $Z_1 \sim \nu_1$ and $Z_2 \sim \nu_2$. In particular, $Z_1 \neq Z_2$.
\end{proof}

\begin{proposition}\label{prop:LdiffXi}
Let $\phi:\Pc_2(\H)\to (-\infty, \infty]$ be $\MC$-differentiable in $\mu$. Then there exists   a measurable function $\xi : \H \to \H$ such that $  \partial\overline{\phi}(X) = \{\xi(X)\}$ for all $X\sim \mu$.
\end{proposition}
\begin{proof}
Let $X \sim \mu$ and $Z \in \partial \overline{\phi}(X)$. By \Cref{lem:subdiff_cond} we have $\E[X|Z] \in \partial \overline{\phi}(X)$. As $\partial \overline{\phi}(X)$ is a singleton, we have $Z=\E[Z|X]$, hence there exists a measurable $\xi$ such that $Z=\xi(X)$. Next, we need to argue that the function $\xi$ does not depend on the choice of the random variable $X \sim \mu$. To that end, let $X_1,X_2 \sim \mu$ with $\partial\overline{\phi}(X^i) = \{ \xi^i(X^i)\}$ for $i \in \{1,2\}$. Pick $X \sim \mu$ that allows for independent randomization. We then have $\xi^i(X) \in \partial\overline{\phi}(X)$ for $i \in \{1,2\}$ by \Cref{cor:LawInvSubDiffRandom}. As $\partial\overline{\phi}(X)$ is a singleton, we have $\xi^1=\xi^2$ $\mu$-a.s. 
\end{proof}

\begin{corollary}\label{cor:impliesheridMonge}
Let $\phi: \Pc_2(\H) \to \R$ be $\mc$-convex and $\mu \in \Pc_2(\H)$. If $\phi$ is $\mc$-differentiable in $\mu$ with  derivative $\nu$, then $(\mu,\nu)$ is a Monge pair.    
\end{corollary}
\begin{proof}
    This is a consequence of \Cref{lem:mc_diff_char}, \Cref{cor:LawInvSubDiffRandom} and \Cref{prop:LdiffXi}.
\end{proof}

\begin{remark}
We compare three different possible notions of differentiability for a  $\mc$-convex functions $\phi: \Pc_2(\H) \to (-\infty,+\infty] $ at $\mu \in \dom(\phi)$: 
\begin{enumerate}
    \item $\overline{\phi}$ is Frechet differentiable at some $X \sim \mu$. This is the notion of L-differentiability as define in \cite[Defintion~5.22]{CaDe18}. It is shown in \cite[Proposition~5.24] {CaDe18} this is equivalent to $\overline{\phi}$ being Frechet differentiable at any $X \sim \mu$.
    \item $\phi$ is $\mc$ differentiable according to \Cref{def:MCdiff}, i.e.\ $\partial_\mc\phi(\mu)$ is a singleton. This is the notion that we use in the present article and it is equivalent to $\partial \overline{\phi}(X)$ being a singleton at any $X \sim \mu$ (see \Cref{lem:mc_diff_char}). Provided that $\phi$ is continuous in $\mu$, this equivalent to $\overline{\phi}$ being Gateaux differentiable at any $X \sim \mu$.\footnote{This is because $\overline{\phi}$ is the continuous in any $X \sim \mu$ by \Cref{lem:contequiv}. Moreover, a convex function is Gateaux differentiable in a point if it is continuous and its subdifferential is a singleton, see e.g.\ \cite[Proposition~17.26]{BaCo17}.} 
    \item $\partial \overline{\phi}(X)$ is a singleton at some $X \sim \mu$.
\end{enumerate}
It is clear that (1) is stronger than (2) which is again stronger than (3). In fact, the converse implications are in general not true (even if $\H$ is finite dimensional). 

\end{remark}

\begin{example}\label{ex:nonFrechet} 
For every $P \in \Pc_2(\Pc_2(\R^d))$ there exists an $\mc$-convex function $\phi : \Pc_2(\R^d) \to \R$ such that for $P$-a.e.\ $\mu$ the Lions $\overline{\phi}$ is not Frechet differentiable at any $X \sim \mu$. This function $\phi$ can be constructed using a modification of the construction in \cite[Example~4.6.10]{BoVa10}: As $P$ is a Borel measure, there exists an increasing sequence $(K'_n)_n$ of compact subsets of $\Pc_2(\R^d)$ such that $P(K'_n) \to 1$. We write $K_n$ for the $\mc$-convex hull of $K_n'$ and note that $K_n$ is still compact. We set $\phi_n(\mu) = \textrm{dist}_{\W_2}(\mu,K_n)$ and $\phi := \sum_n 2^{-n} \phi_n$.

We write $C_n := \{X : \law(X) \in K_n \}$ and note that as $C_n$ is law-invariant by \Cref{lem:EnoughRandomization} 
\[
\overline{\phi_n}(Y) =  \textrm{dist}_{\W_2}(\law(Y),K_n) = \inf_{X \in C_n } \| X- Y \|_2 = \textrm{dist}_{\|\cdot\|_2}(Y,C_n).
\]
This shows in particular that $\phi_n$ and hence $\phi$ is L-convex. Moreover, it is shown in \cite[Exercise~4.2.6]{BoVa10} that if $C_n$ is a closed convex set with empty interior, then $\textrm{dist}_{\|\cdot\|_2}(\cdot,C_n)$ is in every $X \in C_n$ not Frechet differentiable. Hence, $\overline{\phi} = \sum_n 2^{-n}  \overline{\phi_n}$ is not Frechet differentiable in every point of $\bigcup_n C_n$.

Below we will show that for certain  $P \in \Pc_2(\Pc_2(\R^d))$ every  $\mc$-convex $\phi : \Pc_2(\R^d) \to (-\infty, \infty]$ is $P$-a.s.\ $\mc$-differentiable. In particular  that $\phi$ is $\mc$-differentiable in $\mu$ does not imply that $\overline\phi$ is Frechet differentiable in some $X\sim \mu$.
\end{example}

\begin{example}
There exist a $\mc$-convex $\phi:\Pc_2(\Pc_2(\R^d)) \to \R^d$ and $X, X', X\sim X'$ such that $|\partial \overline\phi(X)| = 1, |\partial \overline\phi(X')| > 1 $. Specifically let $d=2$, assume wlog that $(\Omega, \F, \P)$ is the unit square $[0,1]^2$ equipped with two dimensional  Lebesgue measure $\lambda^{(2)}$. We set $\mu:=\lambda\otimes \delta_0$, $\nu:= \lambda\otimes \tfrac{\delta_{-1} + \delta_{-1}}{2}$ (where $\lambda$ denotes one dimensional Lebesgue measure), $\phi(.):= \MC(.,\nu)$, take $X$ to be an isomorphism of $([0,1]^2, \mathfrak B, \lambda^{(2)})$ and $(\R^2, \mathfrak B, \mu )$ and set $X'(u_1, u_2):= (u_1,0)$. 

Then $X\sim \mu\sim X'$ and it is straightforward to see that $|\partial \overline\phi(X)| = 1$, while $|\partial \overline\phi(X')| > 1$.  In particular $\overline\phi$ is Gateaux differentiable in $X$ but not in $X'$. 
\end{example}

\section{Brenier's theorem for measures on $\Pc_2(\H)$}\label{sec:DensenessOfTransportReg}

\subsection{Characterization of $\mc$-optimal couplings}

\begin{theorem} \label{thm:char.optimal_cpls}
    Let $P,Q \in \Pc_2(\Pc_2(\H))$ and $\Pi \in \cpl(P,Q)$. Then, the following are equivalent:
    \begin{enumerate}
        \item $\Pi$ is optimal;
        \item there exists $\phi : \Pc_2(\H) \to (-\infty, + \infty]$ L-convex with $\Pi(\partial_\mc \phi) = 1$;
        \item there exist $\overline{\phi} : L_2(\Omega,\F,\P) \to (-\infty, + \infty]$ lower semi-continuous, law-invariant, convex and random variables $X,Y$ with values in $L_2(\Omega,\F,\P)$ such that
        \[
            (\law(X),\law(Y)) \sim \Pi \quad \text{and} \quad
            (X,Y) \in \partial \overline{\phi} \text{ almost surely.}
        \]
    \end{enumerate}
\end{theorem}

\begin{proof}
    From the fundamental theorem of optimal transport we know that $\Pi$ is optimal if and only if there exists $\phi : \Pc_2(\H) \to (-\infty,+\infty]$ with
    \[
        \phi^{\mc\mc} = \phi \quad \text{and} \quad \Pi(\partial_\mc \phi) = 1.
    \]
    Since $\phi$ being L-convex is equivalent to $\phi^{\mc\mc} = \phi$, we have the equivalence of (1) and (2).

    We have already seen that $\phi$ being $\mc$-convex means precisely that its lift $\overline{\phi}$ is lower semi-continuous, law-invariant and convex, and the converse is also true.
    Therefore, (2) and (3) are equivalent.
\end{proof}

\begin{proposition} \label{prop:TR=>HMonge}
    Let $P \in \Pc_2(\Pc_2(\H))$. The following are equivalent:
    \begin{enumerate}
        \item $P$ is transport-regular;
        \item For every $\mc$-convex $\phi$, we have $P(\{ \mu : |\partial_\mc \phi(\mu)| > 1 \}) = 0$.
    \end{enumerate}
    In particular, for every $Q \in \Pc_2(\Pc_2(\H))$ the optimal coupling $\Pi \in \cpl(P,Q)$ is induced by a Monge map $T:\Pc_2(\H) \to \Pc_2(\H)$ and, in addition, for $P$ almost all $\mu$, $\cplopt(\mu, T(\mu))$  consists of a single element which is again of Monge type.
\end{proposition}

\begin{proof}
    The characterization of transport-regularity of $P$ simply follows from applying \Cref{lem:apx.char.TR} to the current setting. Secifically, we let $(\mathcal X,d) = (\Pc_2(\Pc_2(\H)),\W_2)$ and $c = -\mc$ and note that $\Pc_2(\Pc_2(\H))$ has by \cite[Lemma 2.3]{BaBePa18} the property that families in $\Pc_2(\Pc_2(\H))$ with bounded second moment are tight.

    To see the remaining claim, we recall that by \Cref{cor:impliesheridMonge}
    \[
        \{ \nu \} = \partial_\mc \phi(\mu) \implies
        \exists! \pi \in \cplopt(\mu,\nu) \text{ and }\pi\text{ is Monge}.
    \]
    We conclude that $\Pi$ has the claimed properties.
\end{proof}

\subsection{Regularization on $\R^d$ by Brownian sheet}

The goal of this section, is to construct transport regular measures based on Brownian sheet.
To this end, we consider the law of a $(d,d)$-Brownian sheet, denote by $\mathbb W$, viewed as a probability measure on $L_2([0,1]^d;\R^d)$.
That is, $\mathbb W$ is the centered, non-degenerate Gaussian measure on $L_2([0,1]^d;\R^d)$ induced by the $(d,d)$-Brownian sheet ($d$-parameters and $d$-dimensions).

\begin{lemma} \label{lem:BMDifferentiable}
    Let $\mathbb W$ be the law of a $(d,d)$-Brownian sheet.
    Let $\phi : \Pc_2(\R^d) \to \R $ be $\mc$-convex and let $Y \in L_2([0,1]^d;\R^d)$ be finitely valued.
    Define $P\in \Pc_2(\Pc_2(\R^d))$ by $$P(A) = \WW(\{X\in L_2([0,1]^d;\R^d): \law_\lambda (X + Y) \in A \})$$ for all measurable $A\subseteq \Pc_2(\R^d)$, i.e.\ $P = (X \mapsto \law_\lambda(X+Y))_{\#} (\WW)$.
    Then $P$-a.s.\ $\# \partial_\mc\phi \leq 1$.

    In particular, $P$ is transport-regular.
\end{lemma}

\begin{proof}
    We need to show that $P$-almost surely the subdifferential of $\phi$ is a singleton.

    To this end, we first fix a bounded set $A \subseteq \Pc_2(\R^d)$ and define the $\mc$-convex potential
    \[
        \phi_A(\mu) := \sup_{\nu \in A} \mc(\mu,\nu) - \phi^\mc(\nu) = (\phi^\mc + \chi_A)^{\mc}(\mu).
    \]
    We have that $\phi_A \le \phi$, and $(\Pc_2(\R^d) \times A) \cap \partial_\mc \phi \subseteq \partial_\mc \phi_A$.
    Furthermore, since $A$ is bounded, $\phi_A$ is Lipschitz continuous because
    \[
        \phi_A(\mu) - \phi_A(\mu') \le \sup_{\nu \in A} \mc(\mu,\nu) - \mc(\mu',\nu) \le \W_2(\mu,\mu') \sup_{\nu \in A} \, \Big( \int |y|^2 \, d\nu \Big)^\frac12,
    \]
    for all $\mu,\mu' \in \Pc_2(\R^d)$.
    Clearly, this also provides Lipschitz continuity of $\overline{\phi}_A$.
    Hence, using that under $\WW$, $X + Y$ is  distributed according to a non-degenerate Gaussian on $L_2([0,1]^d;\R^d)$, we find that $X \mapsto \overline{\phi}_A(X + Y)$ is $\WW$-almost surely Gateaux differentiable, see \cite{Ph78}.
    Furthermore, by \cite[Theorem 3.1 and preceding discussion]{AyWuYi08} the Brownian sheet admits almost surely a local time.
    Consequently, as $Y$ takes finitely many values, $\law_\lambda(X + Y)$ is almost surely absolutely continuous w.r.t.\ the $d$-dimensional Lebesgue measure.
    In particular, $P$ is concentrated on transport regular measures.
    Hence, we can apply \Cref{cor:LawInvSubDiffRandom} and obtain that for $\WW$-a.e.\ $X$ 
    \[
        \{ \law(Z) : Z \in \partial \overline{\phi}_A(X + Y) \} = \partial_\MC \phi_A(\law_\lambda(X + Y)).
    \]
    
    As Gateaux differentiability of $\overline{\phi}_A(X)$ yields that $\partial \phi_A(\law(X))$ is a singleton, we get that $\partial_\mc \phi_A(\mu)$ is a singleton for $P$-almost every $\mu$.
    
    Next, we consider the set
    \[
        N := \{ \mu \in \Pc_2(\R^d) : \partial_\mc \phi(\mu) \text{ contains more than 1 element} \}.
    \]
    Since $\Pc_2(\R^d)$ can be covered by a countable sequence of bounded sets, it suffices to show that
    \[
        N_B := \{ \mu \in \Pc_2(\R^d) : \partial_\mc \phi(\mu) \cap B  \text{ contains more than 1 element} \},
    \]
    where $B \subseteq \Pc_2(\R^d)$ is bounded, is a $P$-null set.
    As observed in the first part of the proof, we have the inclusion
    \[
        (\Pc_2(\R^d) \times B) \cap \partial_\mc \phi \subseteq \partial_\mc \phi_B.
    \]
    Since $\partial_\mc \phi_B(\mu)$ is a singleton for $P$-almost every $\mu$, we conclude that $P(N_B) = 0$.

    Finally, we conclude that $P$ is transport-regular by \Cref{prop:TR=>HMonge}.
\end{proof}

\begin{theorem} \label{thm:2nd.TR.dense}
    The set of transport-regular measures in $\Pc_2(\Pc_2(\R^d))$ is dense.
\end{theorem}

\begin{proof}
    Fix a measure $Q \in \Pc_2(\Pc_2(\R^d))$.
    Furthermore, by an approximation argument we can assume that $Q$ is supported on finitely supported measures in $\Pc_2(\R^d)$.
    Consider a probability space $(\tilde \Omega,\tilde \F,\tilde \P)$ supporting independent random variables $X,\Gamma$ taking values in $L_2([0,1]^d;\R^d)$ where $X$ is finitely valued and $\law_\lambda(X) \sim Q$ as well as $\Gamma \sim \WW$.
    
    For $\epsilon > 0$, consider the random variable $X_\epsilon := X + \epsilon \Gamma$ and denote the induced law on $\Pc_2(\Pc_2(\R^d))$ by $Q_\epsilon := \law_{\tilde \P}(\law_\lambda(X_\epsilon))$.
    Clearly, we have for $\epsilon \searrow 0$
    \[
        X_\epsilon \to X \quad \text{in }L_2([0,1]^d,\R^d),
    \]
    pointwise on $\tilde \Omega$, which yields that
    \[
        \law_\lambda(X_\epsilon) \to \law_\lambda(X) \quad \text{in }\Pc_2(\R^d),
    \]
    pointwise on $\tilde \Omega$.
    Since by independence of $X$ and $\Gamma$, $\E_{\tilde \P}[\|X_\epsilon\|^2] = \E_{\tilde \P}[\|X\|^2] + \epsilon^2$, we conclude that $Q_\epsilon \to Q$ in $\Pc_2(\Pc_2(\R^d))$.
    
    It remains to show that $Q_\epsilon$ is transport-regular.
    To this end, note that $Q_\epsilon$ is the mixture of the transport-regular measures $\law_{\tilde \P}( \law_\lambda(X_\epsilon) | X)$.
    Clearly, being transport-regular is closed under taking mixtures, which concludes the proof.
\end{proof}

\subsection{Regularization on $\H$ by Brownian motion/standard Wiener process}\label{sec:Wiener_process_reg}

Throughout this section, let $B=(B_t)_{t\in [0,1]}$ be a one-dimensional Brownian motion and let $W=(W_t)_{t\in [0,1]}$ be a $Q$-Wiener process, where $Q : H \to H$ is a symmetric, positive semi-definite linear operator with finite trace. 
Recall that a $Q$-Wiener process $(W_t)_{t\in[0,1]}$ on $H$ is an $H$-valued stochastic process such that each $W_t$ is a centered Gaussian random variable with covariance operator $tQ$, the trajectories are continuous, $W_0=0$, and for $t>s$, the increment $W_t - W_s$ is independent of the past and Gaussian with covariance $(t-s)Q$.

The main result of this section establishes the transport regularity of the law of the occupation measure associated with a $Q$-Wiener process, where the covariance operator $Q$ is non-degenerate. 
In particular, it applies to the law of the occupation measure of Brownian motion when $\H = \mathbb{R}^d$.
\begin{theorem} \label{thm:Q-Wiener.TR}
    Let $Q \in L(H)$ be a positive, symmetric linear operator with finite trace.
    Let $W$ be a $Q$-Wiener process. 
    Then $\law(W_\# \lambda)$ is a transport-regular measure on $\mathcal P_2(H)$.
\end{theorem}

Following the classical approach, it suffices to show that continuous and convex functions on $H$ are almost surely Gateaux differentiable.
By Zajíček \cite{Za}, the set of non-differentiability points of a continuous convex function on a separable Hilbert space is contained in a countable union of (c$-$c)-hypersurfaces.
Here, a (c$-$c)-hypersurface $S$ is a subset of $H$ such that there are $v \in H$ and convex Lipschitz $f,g : \{ v \}^\perp \to {\rm span}(v)$ with
\[
    S = \{ x + f(x) - g(x) : x \in H \text{ with } \langle x, v \rangle = 0 \}
\]
Consequently, it suffices to prove that, almost surely, the occupation measure of a $Q$-Wiener process does not charge any DC hypersurface in $H$.

In fact, we will even show that, almost surely, Lipschitz hypersurfaces are not charged.
To prove this, the crucial observation is that the occupation measure $b_\# \lambda$ associated with a path $b : [0,1] \to H$ charges no Lipschitz hypersurface if for all unit vectors $v \in H$ we have
\[
    \forall A \subseteq [0,1], \; \lambda(A) > 0 : \sup_{s,t \in A} \frac{|\langle b(s) - b(t), v \rangle|}{\|b^\perp(s) - b^\perp(t)\|} = \infty,
\]
where $b = \langle b, v \rangle v + b^\perp$.

\begin{lemma} \label{lem:non.zero}
    Let $B$ be a 1-dimension Brownian motion and let $W$ be an independent $Q$-Wiener process.
    Assume that
    \begin{equation}
    \label{eq:non.zero}
        \mathbb P\Big( \exists A \subseteq [0,1], \, \lambda(A) > 0 : \sup_{s,t \in A} \frac{|B_s - B_t|}{\|W_s - W_t\|} < \infty \Big) > 0.
    \end{equation}
    Then there exists $L > 0$ such that for all $\epsilon \in (0,1)$
    \[
        \mathbb P\Big( \exists A \subseteq [0,1], \, \lambda(A) \ge 1 - \epsilon : \sup_{s,t \in A} \frac{|B_s - B_t|}{\|W_s - W_t\|} \le L  \Big) > 0.
    \]
\end{lemma}

\begin{proof}
    Suppose that \eqref{eq:non.zero} holds. 
    Then there exist $\delta \in (0,1)$ and $L > 0$ such that the event $\mathcal E(\delta,L,0,1)$, defined for $0 \le a \le b \le 1$ by
    \[
        \mathcal E(\delta,L,a,b) := \Big\{ \exists A \subseteq [a,b], \, \lambda(A) \ge (b-a) \delta : \sup_{s,t \in A} \frac{|B_s - B_t|}{\|W_s - W_t\|} \le L \Big\},
    \]
    has positive probability.
    
    Fix $\epsilon \in (0,1)$. 
    Let $A \subseteq [0,1]$ be a measurable set with $\lambda(A) > 0$. 
    By the Lebesgue density theorem, there exists $t \in [0,1]$ and $r > 0$ such that for all $\epsilon' \in (0,r)$,
    \[
        \frac{\lambda(A \cap (t-\epsilon',t+\epsilon'))}{2\epsilon'} \ge 1 - \epsilon/2.
    \]
    Let $(q_k)_{k \in \mathbb N}$ be a sequence of rational numbers converging to $t$. 
    Then for sufficiently large $n$, we have
    \[
        \lim_{k \to \infty}
        \frac{\lambda(A \cap (q_k - 1/n, \, q_k + 1/n))}{2/n} 
        = \frac{\lambda(A \cap (t - 1/n, \, t + 1/n))}{2/n} 
        \ge 1 - \epsilon/2.
    \]
    Hence, there exists $(q,n) \in \mathbb Q \times \mathbb N$ such that
    \[
        \frac{\lambda(A \cap [q, q+1/n])}{1/n} \ge 1 - \epsilon.
    \]
    Since $A \subseteq [0,1]$ with $\lambda(A) > 0$ was arbitrary, we deduce
    \[
        \mathcal E(\delta,L,0,1) \subseteq \bigcup_{(q,n) \in \mathbb Q \times \mathbb N} \mathcal E(1 - \epsilon ,L,q,q+1/n).
    \]
    Therefore, there exists some $(q,n) \in \mathbb Q \times \mathbb N$ such that $\mathbb P(\mathcal E(1 - \epsilon,L,q,q+1/n)) > 0$.
    
    Since $(B,W)$ is a Wiener process, we may assume w.l.o.g.\ that $q = 0$. 
    Then
    \begin{align*}
        \mathcal E(1-\epsilon,L,0,1/n) 
        &= \Big\{ \exists A \subseteq [0,1/n], \, \lambda(A) \ge \frac{1 - \epsilon}{n} : 
        \sup_{s,t\in A} \frac{|B_{s/n} - B_{t/n}|}{\|W_{s/n} - W_{t/n}\|} \le L \Big\}.
    \end{align*}
    Note that by the scaling property
    \[
        \sqrt{n}\,\big((B_{t/n},W_{t/n}))_{t \in [0,1]} \sim (B,W).
    \]
    Hence,
    \[
        \mathbb P(\mathcal E(1 - \epsilon, L, 0, 1/n)) = \mathbb P(\mathcal E(1 - \epsilon, L, 0, 1)),
    \]
    which proves the claim.
\end{proof}

\begin{theorem} \label{thm:nolip}
    Let $B$ be a 1-dimension Brownian motion and let $W$ be an independent $Q$-Wiener process.
    Then, we have
    \begin{equation}
        \label{eq:thm.nolip}
        \mathbb P\Big( \exists A \subseteq [0,1], \, \lambda(A) > 0 : \sup_{s,t \in A} \frac{|B_s - B_t|}{\|W_s - W_t\|} < \infty \Big) = 0.
    \end{equation}
\end{theorem}

\begin{proof}
    Suppose, for contradiction, that the probability in \eqref{eq:thm.nolip} is positive. 
    Then, by \Cref{lem:non.zero}, there exists $L > 0$ such that for all $\epsilon \in (0,1)$,
    \[
        \mathbb P \Big( \exists A \subseteq [0,1], \, \lambda(A) \ge 1 - \epsilon : \sup_{s,t \in A} \tfrac{|B_s - B_t|}{\|W_s - W_t\|} \le L \Big) > 0.
    \]
    Consider the event
    \[
        E := \Big\{ \sup_{s,t \in [0,1]} \|W_s - W_t\| \le 1, \ \inf_{s \in [0,1/3], \, t \in [2/3,1]} |B_s - B_t| \ge 2L \Big\},
    \]
    and set $p := \mathbb P(E) > 0$. 
    Fix $\epsilon > 0$ with $4 \epsilon < p$.
    Let
    \[
        \mathcal A := \{ A \subseteq [0,1] : \lambda(A) \ge 1 - \epsilon \}.
    \]
    For $n \in \mathbb N$ and $k = 0,\dots,2^n$, define $h_n := 2^{-n}$, $t_{n,k} := k h_n$, and
    \begin{multline*}
        E(n,k) := \Big\{ \sup_{s,t \in [t_{n,k},t_{n,k+1}]} \|W_s - W_t\| \le \sqrt{h_n}, \\
        \inf_{s \in [t_{n,k}, t_{n,k} + h_n/3], \, t \in [t_{n,k+1} - h_n/3, \, t_{n,k+1}]} |B_s - B_t| \ge L \sqrt{h_n} \Big\}.
    \end{multline*}
    By scaling invariance of Brownian motion and the Wiener process,
    \[
        \mathbb P(E(n,k)) = \mathbb P(E(n,0)) = \mathbb P(E) = p > 0.
    \]
    For each $n \in \mathbb N$, define
    \[
        S_n := \Big\{ \omega \in \Omega : 2^{-n} \# \{ k : \omega \in E(n,k) \} \ge 4 \epsilon \Big\},
    \qquad 
        S := \bigcup_{n \in \mathbb N} S_n.
    \]
    Thus, for $\omega \in S$, there exist $n \in \mathbb N$ and pairwise distinct indices $k_1,\dots,k_J$ with 
    \[
        J \ge \lceil \epsilon 2^n \rceil
    \]
    such that $\omega \in \bigcap_{j=1}^J E(n,k_j)$.

    Now fix $A \in \mathcal A$. 
    Since $\lambda(A) \ge 1 - \epsilon$, a simple covering argument shows that there exists at least one $j$ with
    \[
        [t_{n,k_j}, \, t_{n,k_j} + h_n/3] \cap A \neq \emptyset 
        \quad \text{and} \quad
        [t_{n,k_j+1} - h_n/3, \, t_{n,k_j+1}] \cap A \neq \emptyset.
    \]
    Hence, for every $\omega \in S$ and every $A \in \mathcal A$,
    \[
        \sup_{s,t \in A} \frac{|B_s(\omega) - B_t(\omega)|}{\|W_s(\omega) - W_t(\omega)\|} \ge 2L.
    \]
    Finally, since $4\epsilon < p$, the law of large numbers yields
    \[
        \lim_{n \to \infty} \mathbb P(S_n) = 1.
    \]
    Consequently,
    \[
        \Big\{ \exists A \in \mathcal A : \sup_{s,t \in A} \tfrac{|B_s - B_t|}{\|W_s - W_t\|} \le L \Big\} \subseteq S^C,
    \]
    and therefore
    \[
        \mathbb P\Big( \exists A \in \mathcal A : \sup_{s,t \in A} \tfrac{|B_s - B_t|}{\|W_s - W_t\|} \le L \Big) = 0,
    \]
    contradicting our assumption. 
    This proves the theorem.
\end{proof}

When $v$ is some vector in $H$ and $f : \{ v \}^\perp \to {\rm span}(v)$, we write
\[
        \Gamma(f) := \{ x + f(x) : x \in \{ v \}^\perp \}.
\]

\begin{lemma} \label{lem.Lipschitz.graph}
    Let $v, w \in H$ with $\|v\| = \|w\| = 1$, and let $f : \{ v \}^\perp \to {\rm span}(v)$ be $L$-Lipschitz.
    If 
    \[
        \langle v,w \rangle > \frac12 \Big( 1 - \frac{L}{\sqrt{1 + L_2}}\Big)^2,
    \] 
    then there exists a Lipschitz function $g : \{ w \}^\perp \to \mathbb R$ such that $\Gamma(g) = \Gamma(f)$.
\end{lemma}

\begin{proof}
    We show that the set
    \begin{equation}
        \label{eq:lem.Lipschitz.graph}
        \big\{ \big({\rm pr}_{\{w\}^\perp}(x + f(x)), \, x + f(x) \big) : x \in \{v \}^\perp \big\}
    \end{equation}
    defines the graph of a Lipschitz map.
    
    First, observe that since $\|w\| = 1$, for all $x,y \in \{ v\}^\perp$, $x \neq y$,
    \begin{align*}
        \frac{|\langle y + f(y) - x - f(x), \, w \rangle|}{\|y + f(y) - x - f(x)\| \, \|w\|} 
        &\le \frac{\|y - x\| \cdot \|v - w \| + \|f(y) - f(x)\|}{\sqrt{\|y-x\|^2 + \|f(y) - f(x)\|^2}} \\
        &\le \|v - w\| + \frac{L}{\sqrt{1 + L_2}} =: \gamma.
    \end{align*}
    Under our assumption we have
    \[
        \|v-w\| = \sqrt{2 - 2 \langle v,w \rangle} \le \sqrt{1 - 2 \frac{L}{\sqrt{1 + L_2}} + \frac{L_2}{1 + L_2}} = 1 - \frac{L}{\sqrt{1 + L_2}}.
    \]
    It follows that $\gamma \in [0,1)$.
    Rearranging terms and recalling that $\|w\|=1$, we obtain
    \[
        |\langle y + f(y) - x - f(x), \, w \rangle| \le \gamma \, \|y + f(y) - x - f(x)\|.
    \]
    Consequently,
    \begin{align*}
       & \| {\rm pr}_{\{ w \}^\perp }(y + f(y)) - {\rm pr}_{\{ w \}^\perp }(x + f(x)) \| \\
        = & \| (y + f(y)) - (x + f(x)) - \langle (y + f(y)) - (x + f(x)), w \rangle w \| \\
        \ge & \| y + f(y) - x - f(x) \| - |\langle y + f(y) - x - f(x), w \rangle| \\
        \ge & (1 - \gamma) \, \| y + f(y) - x - f(x) \|.
    \end{align*}
   Next, we verify that the projection of the set in \eqref{eq:lem.Lipschitz.graph} onto the first coordinate is exactly $\{ w \}^\perp$.
    Fix $a \in \{ w \}^\perp$ and define
    \[
        \phi(t) := 
        \langle v, f\big( {\rm pr}_{\{ v \}^\perp }(a + tw)\big) - a - tw \rangle, \qquad t \in \mathbb R.
    \]
    Then
    \[
        \phi(t) - \phi(0) 
        \le |t| L \, \| {\rm pr}_{\{v\}^\perp }(w) \| - t \langle v,w \rangle 
        = t \big( {\rm sign}(t) \, L \sqrt{1 - \langle v,w \rangle^2} - \langle v,w \rangle \big).
    \]
    Because $\langle v,w \rangle^2 > \tfrac{L_2}{ 1 + L_2 }$, we get 
    \[
        \frac{\langle v,w \rangle^2}{1 - \langle v,w \rangle^2} > L_2.
    \]
    This ensures that $\phi : \mathbb R \to \mathbb R$ is surjective as $\phi$ is also continuous.
    Therefore, there exists $t_0 \in \mathbb R$ with $\phi(t_0) = 0$.
    Setting $x_0 := {\rm pr}_{ \{ v \}^\perp }(a + t_0 w)$, we obtain
    \[
        a + t_0 w = x_0 + \langle a + t_0 w, v\rangle v = x_0 + f(x_0),
    \]
    which implies $a = {\rm pr}_{ \{ w \}^\perp }(x_0 + f(x_0))$, as desired.
    
    We conclude that \eqref{eq:lem.Lipschitz.graph} is indeed the graph of a $\tfrac{1}{1 - \gamma}$-Lipschitz function.
\end{proof}

\begin{proof}[Proof of \Cref{thm:Q-Wiener.TR}]
    Fix $v \in H$, $v \neq 0$. 
    We decompose $W$ into a one-dimensional process $B^v$ along ${\rm span}(v)$ and an orthogonal component $B^{v,\perp}$ in $\{ v \}^\perp$: 
    \[
        B_t^v := \langle W_t, v \rangle v, 
        \qquad 
        B_t^{v,\perp} := W_t - B_t^v,
        \quad t \in [0,1].
    \]
    For $w, w' \in \{ v \}^\perp$, we compute
    \begin{gather*}
        {\rm Cov}(B_t^v, B_s^v) 
        = \mathbb E[ \langle W_t, v \rangle \langle W_s, v \rangle ] 
        = \min(t,s)\,\langle Qv, v \rangle,
        \\
        \mathbb E[ \langle B_t^{v,\perp}, w \rangle \, \langle B_s^{v,\perp}, w' \rangle ] 
        = \min(t,s)\,\langle Qw, w' \rangle.
    \end{gather*}
    Since $Q$ is non-degenerate and non-negative, we have $\langle Qv, v \rangle > 0$, and $B^{v,\perp}$ is a centered $Q|_{\{v\}^\perp}$-Wiener process. 
    Thus, by \Cref{thm:nolip},
    \begin{equation} \label{eq:Q.Wiener.sets.not.charged}
        \mathbb P \Big( \exists A \subseteq [0,1], \, \lambda(A) > 0 : 
        \sup_{s,t \in A} \frac{\|B_s^v - B_t^v\|}{\|B_s^{v,\perp} - B_t^{v,\perp}\|} < \infty \Big) = 0.
    \end{equation}
 Let ${\rm Lip}(v)$ denote the set of Lipschitz functions $f : \{ v \}^\perp \to {\rm span}(v)$.
    As a direct consequence of \eqref{eq:Q.Wiener.sets.not.charged}, 
    \[
        \mathbb P \Big( \forall f \in {\rm Lip}(v): \lambda\big( \{ t \in [0,1] : W_t \in \Gamma(f) \} \big) = 0 \Big) = 1.
    \]
    Denote by $\Omega^v$ the corresponding full-probability event.

    Since $H$ is separable, there exists a countable dense family of unit vectors $(v_n)_{n \in \mathbb N}$ in the unit ball of $H$. 
    Define
    \[
        \tilde \Omega := \bigcap_{n \in \mathbb N} \Omega^{v_n}.
    \]
    By \Cref{lem.Lipschitz.graph}, for every $v \in H$ and $f \in {\rm Lip}(v)$ there exist $n \in \mathbb N$ and $g \in {\rm Lip}(v_n)$ such that $\Gamma(f) = \Gamma(g)$. 
    Consequently, for every $\omega \in \tilde \Omega$,
    \[
        \forall v \in H, \, \forall f \in {\rm Lip}(v): 
        \quad \lambda\big( \{ t \in [0,1] : W_t(\omega) \in \Gamma(f) \} \big) = 0.
    \]
    Hence, on $\tilde \Omega$, the occupation measure $(t \mapsto W_t(\omega))_\# \lambda$ does not charge Lipschitz hypersurfaces of $H$. 
    By Zajíček \cite[Theorem 2]{Za}, this implies
    \[
        \mathbb P\Big( \forall \phi : H \to \mathbb R \ \text{convex}, \ 
        \lambda\big( \{ t \in [0,1] : |\partial \phi(W_t)| > 1 \} \big) = 0 \Big) = 1.
    \]
    Therefore, $(t \mapsto W_t)_\# \lambda$ is almost surely a transport regular measure on $H$, which completes the proof.
\end{proof}

\section{On transport-regularity and Baire category} \label{sec:PolishnessOfRegularMeasures}

In this section, we let $(\mathcal X,d)$ be a Polish metric space
and let $c : \mathcal X \times \mathcal X \to \R$ be a continuous cost function satisfying the growth condition $|c(x,y)| \le C ( 1 + d(x,x_0)^p + d(y,y_0)^p)$ for some $C>0$, $p \ge 1$, and $x_0,y_0 \in \mathcal X$.
We write $\cplopt^c(\mu,\nu)$ for the set of couplings that are optimal for the transport problem between $\mu$ and $\nu$ with cost $c$, that is,
\[
    V_c(\mu,\nu) := \inf_{\pi \in \cpl(\mu,\nu)} \int c(x,y) \, \pi(dx,dy).
\]
A measure $\mu$ is the called $c$-transport-regular if for every $\nu \in \Pc_p(\mathcal X)$ there is a unique optimal transport plan from $\mu$ to $\nu$ and this plan is of Monge type.
Recall that the $c$-transform of a Borel measurable function $\phi : \mathcal X \to \mathbb R \cup \{ - \infty \}$, $\phi \neq -\infty$, is given by
\[
    \phi^c(y) = \inf_{x \in \mathcal X} c(x,y) - \phi(x),
\]
which is thus an upper semicontinuous function, and $\phi$ is called $c$-concave if
\[
    \phi(x) = \inf_{y \in \mathcal X} c(x,y) - \phi^c(y).
\]
In this case, the $c$-superdifferential of $\phi$ is given by
\[
    \partial^c \phi = \{ (x,y) \in \mathcal X \times \mathcal X : \phi(x) + \phi^c(y) = c(x,y) \},
\]
where we also write $\partial^c\phi(x) = \{ y \in \mathcal X : (x,y) \in \partial^c \phi \}$ for $x \in \mathcal X$.

\begin{lemma} \label{lem:apx.char.TR}
    Let $\mu \in \Pc_p(\mathcal X)$. Then $\mu$ is $c$-transport-regular if and only if for all $c$-concave $\phi : \mathcal X \to (-\infty,+\infty]$
            \begin{equation}
        \label{eq:subdiff.big}
        \mu(A) = 0 \quad\text{where }A:=\{x \in \mathcal X : \partial_c \phi(x) \text{ contains more than one point} \}.
    \end{equation}
\end{lemma}

\begin{proof}
    Assume that $\eqref{eq:subdiff.big}$ holds for all $c$-concave $\phi$.
    By the fundamental theorem of optimal transport there is for each $\nu \in \Pc_p(\mathcal X)$ an optimal $c$-concave potential $\phi$ such that
    \[
        \pi(\partial^c \phi) = 1, \quad \pi \in \cplopt^c(\mu,\nu). 
    \]
    Hence, $\partial^c \phi(x)$ contains at most one point for $\mu$-a.e.\ $x$, from where it readily follows that $\pi$  
    is of Monge type.
    Consequently, since $\cplopt^c(\mu,\nu)$ is a non-empty convex set, it is a singleton.

    Now, assume that $\mu$ is transport-regular.
    As the set 
    \[
        \{ (x,y_1,y_2) \in \mathcal X^3 : y_1,y_2 \in \partial^c\phi(x),\, y_1 \neq y_2\}
    \]
    is Borel, there exists by the Jankov--von Neumann uniformization theorem an analytically measurable selection $h = (h_1,h_2) : A \to \mathcal X \times \mathcal X$, and observe that $h_1 \neq h_2$ by construction.
    With these maps we define subprobability measures $\nu_i := (h_i)_\# \mu( \cdot \cap A)$ for $i = 1,2$.
    By inner regularity, there exists a compact set $K \subseteq \mathcal X$ such that $\nu_i(K) > \mu(A) / 2$.
    Define the potential $\phi_K := (\phi^c - \chi_K)^c$, which satisfies
    \[
        \phi_K \ge \phi \text{ and }
        \phi_K(x) + \phi^c_K(y) = c(x,y) \text{ for all }(x,y) \in \partial^c\phi \cap (\mathcal X \times K),
    \]
    and, as $K$ is compact, $\partial^c \phi_K(x) \neq \emptyset$ for all $x \in \mathcal X$.
    Let $g$ be an analytically measurable selector of $\partial^c \phi_K$, and set
    \[
        g_i(x) := 
        \begin{cases}
            h_i(x) & x \in A \cap h_1^{-1}(K) \cap h_2^{-1}(K), \\
            g(x) & \text{otherwise}.
        \end{cases}
    \]
    Then we have that
    \[
        \pi := \frac12 \big( (x \mapsto (x,g_1(x)))_\# \mu + (x \mapsto (x,g_2(x)))_\# \mu \big)
    \]
    has first marginal $\mu$, its second marginal $\nu \in \Pc_p(\mathcal X)$ is concentrated on $K$, and $\pi(\partial^c \phi_K) = 1$.
    Since $\nu_i(K) > \mu(A)/2$, we have $\mu(A \cap h_1^{-1}(K) \cap h_2^{-1}(K)) \ge \mu(A) - \nu_1(K^C) - \nu_2(K^C) > 0$.
    Hence, we have constructed a coupling $\pi \in \cplopt^c(\mu,\nu)$ that is not Monge, which contradicts $c$-transport regularity of $\mu$.
\end{proof}

We introduce functionals which quantify how far a pair of measures $(\mu,\nu)$ is from admitting a unique optimal transport map between them. Based on this, we will be able to quantify how far a measure is from being $c$-transport-regular. 
We introduce the quantities
\[
    \tau^c(\mu,\nu) := \sup_{ \pi \in \cplopt^c(\mu,\nu)} \int \var{\pi_x} \, \mu(dx), \quad \text{and} \quad \tau_R^c(\mu) := \sup_{\nu : \W_p(\nu,\delta_{x_0}) \le R}
    \tau(\mu,\nu),
\]
where the variance for probabilities in $\nu \in \Pc_p(\mathcal X)$ is defined as
\[
    \widehat{{\rm Var}}(\nu) := \sum_{n \in \N} 2^{-n} {\rm Var}((f_n)_\# \nu),
\]
where $(f_n)_{n \in \N}$ is a family of continuous, point-separating functions on $\mathcal X$ bounded by 1, and
\[
    {\rm Var}(\rho) = \int |x-x'|^2 \, \rho\otimes\rho(dx,dx'),
\]
for $\rho \in \mathcal P_2(\R)$.

\begin{lemma} \label{lem:apx.tauRegEquiv}
    For $\mu,\nu \in \Pc_p(\mathcal X)$ we have:
    \begin{enumerate}
        \item There is a unique $\pi \in \cplopt^c(\mu,\nu)$ and $\pi$ is Monge iff $\tau^c(\mu,\nu) = 0$;
        \item $\mu$ is $c$-transport-regular iff $\tau^c_R(\mu) = 0$ for all $R > 0$.
    \end{enumerate}
\end{lemma}

\begin{proof}
    To see (1), first note that $\widehat{\var{\nu}} > 0$ if and only if $\nu$ is not a Dirac measure. Indeed, if $\nu$ is not a Dirac measure then, as $(f_n)_{n \in \N}$ is point-separating, there exists $n \in \N$, such that $(f_n)_\# \nu$ is not a Dirac measure, which yields that
    \[
        \widehat{\var{\nu}} \ge 2^{-n} \var{(f_n)_\# \nu} > 0.
    \]
    Hence, $\tau^c(\mu,\nu) = 0$ iff for every $\pi \in \cplopt^c(P,Q)$ we have that $\pi_x$ is $\mu$-almost surely a Dirac.
    Since $\cplopt^c(\mu,\nu)$ is a convex set, this yields that $\cplopt^c(\mu,\nu)$ contains a single transport plan and this plan is Monge.

    To see (2), note that by the first part we have for every $\nu \in \Pc_p(\mathcal X)$ a unique optimal transport plan $\pi \in \cplopt^c(\mu,\nu)$, which means that $\mu$ is $c$-transport-regular.
\end{proof}

\begin{lemma} \label{lem:apx.tau_usc}
    The map $\Pc_p(\mathcal X) \times \Pc_p(\mathcal X) \to \R : (\mu,\nu) \mapsto \tau^c(\mu,\nu)$ is $\W_p$-upper semi-continuous.
\end{lemma}

\begin{proof}
We first show that
\[
\widehat{\rm Var} : \Pc_p(\mathcal X) \to \R,
\]
is bounded, continuous and concave.
Clearly, ${\rm Var}$ is continuous and concave on $\Pc_2(\R)$ as ${\rm Var}(\rho) = \int x^2 \, \rho(dx) - (\int x \, \rho(dx))^2$.
Since $(f_n)_{n \in \N}$ is a point-separating family of continuous functions on $\mathcal X$ bounded by $1$, we have that $\widehat{\rm Var}$ is the sum of concave and continuous functions on $\Pc_p(\mathcal X)$ bounded by 1, from where we infer the claimed properties of $\widehat{\rm Var}$.

Next, we invoke \cite[Proposition 2.8]{BaBePa18} to find that the map
\[
    \Pc_p(\mathcal X \times \mathcal X) \to \R : \pi \mapsto \int \widehat{\rm Var}(\pi_x) \, \pi(dx),
\]
is upper semi-continuous, concave and bounded.

Finally, consider sequences $(\mu^n)_n$ and $(\nu^n)_n$ and measures $\mu^\infty, \nu^\infty \in \Pc_p(\mathcal X)$ with $\W_p(\mu^n,\mu) \to 0$ and $\W_p(\nu^n,\nu) \to 0$. Let $\pi^n \in \cplopt^c(\mu^n,\nu^n)$ be such that 
\[
\int \var{\pi^n_x} \, \pi^n(dx) \ge \tau(\mu^n,\nu^n) - 1/n.
\]
As $\{ \mu_n : n \in \N \cup \{\infty\}\}$ and $\{ \nu_n : n \in \N \cup \{\infty\}\}$ are $\W_p$-compact, we have that
\[
A: = \bigcup_{n \in \N \cup \{+\infty \} } \cpl(\mu^n,\nu^n)
\]
is $\W_p$-compact as immediate consequence of Prokhorov's theorem and the characterization of $\W_p$-convergence, see \cite[Definition 6.8]{Vi09}. Therefore, there is $\pi \in A$ such that, after possibly passing to a subsequence, $W_p(\pi^n, \pi) \to 0$. By continuity of the map which assigns a couplings its marginals, we have $\pi \in \cpl(\mu,\nu)$. As the coupling $\pi^n$ is optimal between its marginals, the same holds true for the limiting $\pi$, i.e., $\pi \in \cplopt^c(\mu,\nu)$. This is because optimality is characterized in terms of $c$-cyclical monotonicity of the support and the latter is preserved under weak limits (see e.g.\ \cite[Theorem 5.20]{Vi09}). 

All in all, we have 
\[
\tau^c(\mu,\nu) \ge \int \widehat{\rm Var}(\pi_x) \, \pi(dx) \ge \limsup_n \int \widehat{\rm Var}(\pi^n_x) \,  \pi^n(dx) = \limsup_n \tau^c(\mu^n,\nu^n). \qedhere 
\]
\end{proof}

\begin{lemma} \label{lem:apx.tauRusc}
    Assume that $\mathcal X$ has the property that $\mathcal W_p$-bounded sets in $\Pc_p(\mathcal X)$ are tight.
    Then, for $R \ge 0$, the map $\Pc_p(\mathcal X) \to \R : \mu \mapsto \tau_R^c(\mu)$ is $\W_p$-upper semi-continuous.
\end{lemma}

\begin{proof}
    The proof is line by line as for \Cref{lem:apx.tau_usc} subject to replacing $\bigcup_{n \in \N \cup \{\infty\}} \cpl(\mu^n,\nu^n)$ with 
    \[
        A := \bigcup_{n \in \N \cup \{ \infty \} } \bigcup_{\nu : \W_p(\nu,\delta_{x_0}) \le R} \cpl(\mu_n,\nu).
    \]
    While $A$ is not $\W_p$-compact, it is tight by our assumption, as the marginals form a $\mathcal W_p$-bounded set.
    Furthermore, $A$ is weakly closed since any limit point $\pi$ of a sequence $(\pi^n)_n$ in $A$ and $p$-th moment of the second marginals bounded by $R$, has second marginal with $p$-th moment bounded by $R$.
    This follows directly from weak lower semi-continuity of the $p$-th moment w.r.t.\ weak convergence.
    
    We can proceed the proof as in \Cref{lem:apx.tau_usc} and obtain that $\pi$ is in $\cplopt^c(\mu,\cdot)$ by \cite[Theorem 5.20]{Vi09}, and thereby analogously conclude $\W_p$-upper semi-continuity of $\tau^c_R$.
\end{proof}

\begin{theorem}\label{thm:TransportRegularPolish}
    Assume that $\mathcal X$ has the property that $\mathcal W_p$-bounded sets in $\Pc_p(\mathcal X)$ are tight.
    Then the set of $c$-transport-regular measures is $G_\delta$ in $\Pc_p(\mathcal X)$.
\end{theorem}

\begin{proof}
    By \Cref{lem:apx.tauRegEquiv}, the set of $c$-transport-regular measures is precisely given by
    \[
        \bigcap_{R \in \N} \bigcap_{k \in \N} \big\{ \mu \in \Pc_p(\mathcal X) : \tau_R^c(\mu) < 1/k \big\}.
    \]
    As $\tau_R^c$ is upper semicontinuous by \Cref{lem:apx.tauRusc}, the sets $\{ \mu \in \Pc_p(\mathcal X) : \tau_R^c(\mu) < 1/k\}$ are open.
    Therefore, the set of $c$-transport-regular measures is a $G_\delta$-subset of $\Pc_p(\mathcal X)$.
\end{proof}

\section{Optimal transport on $\Pc_2^N(\H)$} \label{sec:OT_On_P_N}
The aim of the section is apply the fundamental theorem of optimal transport to transport on $\Pc_2^N(\H)$ to establish a duality in terms of $\mc$-convex functions. Properties of $\mc$-convex functions, in particular, the connection to convex analysis via an appropriate version of the Lions lift will be the subject of the subsequent section.

Starting with the max-covariance of two probabilities $\mu,\nu \in \Pc_2(\H)$, which is given by
\[
    \mc(\mu,\nu) = \sup_{\pi \in \cpl(\mu,\nu)}
    \int \langle x,y \rangle \, d\pi(x,y),
\]
we now define its iterated counterpart for probabilities in $\Pc_2^N(\H)$.

\begin{definition}\label{def:iteratedMC}
    Let $N >1$ and $P,Q \in \Pc_2^N(\H)$. We inductively define 
    \[
        \mc(P,Q) := \sup_{\Pi \in \cpl(P,Q)} \int \mc(p,q) \, \Pi(dp,dq),
    \]
    where $\mc(p,q)$ is the max-covariance of $p,q \in \Pc_2^{N-1}(\H)$.
\end{definition}
Note that we always use the same symbol $\mc$ for all max-covariance functionals $\mc : \Pc_2^N(\H) \times \Pc_2^N(\H) \to \R$, $N\geq 1$  and trust that this causes no confusion as we always specify the level of iterations we work with. In particular, for fixed $N$, we denote with capital letters probabilities $P,Q \in \Pc_2^N(\H)$ and with lower case letters probabilities $p,q \in \Pc_2^{N-1}(\H)$ on the underlying space.   

We briefly state the connection between the iterated $\mc$-functional and the iterated Wasserstein distance. Note that the iterated Wasserstein distance on $\Pc_2^N(\H)$, $N>1$, is inductively defined as 
\[
\W_2(P,Q) := \inf_{\Pi \in \cpl(P,Q)} \int \W_2(p,q) \, \Pi(dp,dq).
\]
Given a Polish space $\mathcal{X}$, we write $I$ for the intensity operator $I : \Pc^2_2(\mathcal{X}) = \Pc_2(\Pc_2(\mathcal X)) \to \Pc_2(\mathcal{X})$ satisfying the characteristic property, for all measurable bounded $f : \mathcal X \to \R$, 
\[
\int_\mathcal{X} f \, dIP  := \int_{\Pc_2(\mathcal{X})} \int_\mathcal{X} f(x) \, p(dx) \,  P(dp). 
\]
As this operator is well-defined on any Polish space $\mathcal{X}$, it is in particular possible if $\mathcal{X}$ is itself a space of measures. 
This enables us to define the iterated intensity operator
\[
    I^{N-1} : \Pc_2^N(\H) \to \Pc_2(\H).
\]
The next lemma is then a straightforward consequence of the definition of $\W_2, \mc$, and $I^{N-1}$.
\begin{lemma}\label{lem:W2_MC_connection_iterated}
Let $P,Q \in \Pc_2(\H)$. We then have  
\begin{enumerate}
    \item $\W_2^2(P,Q) = \mc(P,P)+ \mc(Q,Q) - 2\mc(P,Q)$,
    \item $\mc(P,P) = \int |x|^2 \, (I^{N-1}P)(dx)$,
    \item $\Pi \in \cpl(P,Q)$ is $\W_2$-optimal iff it is $\mc$-optimal.
\end{enumerate}
\end{lemma}
This simple observation motivates the study of the max-covariance functional on iterated spaces of probability measures.
In order to obtain a characterization of its optimizers, we follow in the footsteps of classical transport theory and specialize the corresponding definition of $c$-transform and $c$-subdifferential to the case of $c = \mc$.

\begin{definition}\label{def:MC_subdiff_ad}
    Let $\phi : \Pc_2^N(\H) \to (-\infty,+\infty]$ be proper. We define its $\mc$-transform as 
    \[
        \phi^{\mc}(Q) = \sup_{ P \in \Pc_2^N(\H)}  \mc(P,Q) -  \phi(P). 
    \]
    A function $\psi : \Pc_2^N(\H) \to (-\infty,+\infty]$ is called $\mc$-convex if there exists a proper function $\phi : \Pc_2^N(\H) \to (-\infty,+\infty]$ such that $\psi = \phi^\mc$.
    Moreover, the $\mc$-subdifferential of $\phi$ at $P \in \Pc_2^N(H)$ is defined as  
    \[
        \partial_{\mc} \phi(P) = \{ Q\in \Pc_2^N(\H) : \phi(R) \ge \phi(P) + \mc(R,Q) - \mc(P,Q) \text{ for every } R \in \Pc_2^N(\H)\},
    \]
    and we write
    \[
        \partial_\mc \phi := \{ (P,Q) : Q \in  \partial_{\mc} \phi(P), P \in \Pc_2^N(\H) \}. 
    \]
\end{definition}

We remark that for a $\mc$-convex function $\phi$, the $\mc$-subdifferential $\partial_\mc \phi$ consists precisely of those $(P,Q) \in \Pc_2^N(H) \times \Pc_2^N(H)$ with $\phi(P) + \phi^c(Q) = \mc(P,Q)$, see also \cite[Definition 5.2]{Vi09}.
Applying the fundamental theorem of optimal transport
to this setting yields
\begin{theorem}\label{thm:ftot.mc}
Let $P,Q \in \Pc_2^N(\H)$. Then we have the  duality relation 
\[
\mc(P,Q) = \inf_{ \substack{ \phi : \Pc_2^{N-1}(\H) \to (-\infty,+\infty] \\ \text{$\mc$-convex} } } \int \phi \,dP + \int \phi^{\mc} \, dQ,
\]
and there exist both a primal optimizer $\Pi \in \cpl(P,Q)$ and a dual optimizer $\phi : \Pc_2^{N-1}(\H) \to (-\infty,+\infty]$.  

Moreover, the following complementary slackness condition holds: Candidates $\Pi$ and $\phi$ are optimal for primal and dual problem, respectively, if and only \[ \Pi(\partial_\mc \phi) =1, \]
that is, $\Pi$ is concentrated on the $\mc$-subdifferential of $\phi$.
\end{theorem}

\begin{remark}
    Any $\mc$-convex function 
    $\phi : \Pc_2^{N-1}(\H) \to (-\infty,+\infty]$ 
    is integrable with respect to every measure 
    $P \in \Pc_2^N(\H)$.  
    In particular, the integrals 
    $\int \phi \, dP$ and $\int \phi^{\mc} \, dQ$ 
    in the dual formulation are always well defined 
    (possibly taking the value $+\infty$).
    
    Indeed, since $\phi$ is $\mc$-convex, its $\mc$-transform is proper.
    Hence, there is $q \in \Pc_2^{N-1}(H)$ such that $\phi^\mc(q) \in \R$.
    Recalling that $\W_2^2(p,q) = \mc(p,p) - 2 \mc(p,q) + \mc(q,q)$ by \Cref{lem:W2_MC_connection_iterated} and using the inequality $\phi(p) \ge \mc(p,q) - \phi^\mc(q)$, we conclude that $2 \phi$ is dominated from below by the $P$-integrable function $p \mapsto \mc(p,p) + \mc(q,q) - \W_2^2(p,q) - 2 \phi^\mc(q)$.
\end{remark}

Above we discussed how to solve the optimal transport problem on the outer most layer.
However, this procedure can be iterated according to the dynamic programming principle.
More specifically, if $P,Q \in \Pc_2^N(\H)$ and $\Pi \in \cplopt(P,Q)$, then one can again solve the optimal transport problems between $\Pi$-almost every pair $(p,q) \in \Pc_2^{N-1}(H) \times \Pc_2^{N-1}(H)$ and gets a selection of optimizers $\pi_{p,q} \in \cplopt(p,q)$. By gluing $\Pi$ and the kernel $(\pi_{p,q})_{p,q}$, one arrives at a measure that contains the information of the optimal transport on the first two layers. Continuing this procedure leads to the following concept of $N$-couplings. 

To give the definition of $N$-couplings, we need to introduce a notation for iterated pushforwards. Let $\Xs$ and $\Ys$ be Polish spaces and $f : \Xs \to \Ys$ be Borel. We write $\Pc[f]$ for the pushforward map 
\[
\Pc[f] : \Pc(\Xs) \to \Pc(\Ys) : \mu \mapsto f_\#\mu
\]
and inductively
\[
\Pc^N[f] : \Pc^N(\Xs) \to \Pc^N(\Ys) : P \mapsto (\Pc^{N-1}[f])_\#P.
\]

\begin{definition}\label{def:CplN}
    For $P,Q \in \Pc_2^N(\H)$, we set 
    \[
    \cpl^N(P,Q) = \{ \Pi \in \Pc^N_2(\H \times \H) :  \Pc^N[\pr_1](\Pi) = P, \Pc^N[\pr_2](\Pi) = Q  \},
    \]
    and call elements of $\cpl^N(P,Q)$ $N$-couplings of $P$ and $Q$.
\end{definition}

The following results make precise that elements of $\cpl^N(P,Q)$ encode the full information of the transport across all layers

\begin{proposition}\label{prop:DPP_CplN_MC}
    For $P, Q \in \Pc_2^N(\H)$ we have 
    \begin{equation}
        \label{eq:DPP_CplN_MC}
        \mc(P,Q) = \max_{ \Pi \in \cpl^N(P,Q) } \int x \cdot y \, d I^{N-1} \Pi (x,y).
    \end{equation}
\end{proposition}
\begin{proof}
    For $N=1$, the assertion follows directly from the definition of $\mc$.
    
    Next, assume that the statement holds for $N-1$. 
    By definition we have
    \[
        \mc(P,Q) = \max_{\pi \in \cpl(P,Q)} \int \mc(p,q) \, \Pi(dp,dq).
    \]
    Fix an optimizer $\pi \in \cpl(P,Q) \subset \Pc_2( \Pc_2^{N-1}(\H) \times \Pc_2^{N-1}(\H))$. 
    By the induction hypothesis,
    \begin{equation}
        \label{eq:DPP_CplN_MC.1}
        \mc(p,q) = \inf_{\Pi_{p,q} \in \cpl^{N-1}(p,q)} \int x \cdot y \, dI^{N-2}\Pi_{p,q}(x,y),
    \end{equation}
    for $p,q \in \Pc_2^{N-1}(\H)$.
    By a standard measurable selection argument (e.g.\ the Jankov--von Neumann uniformization theorem \cite[Theorem 18.1]{Ke95}), there is a universally measurable optimal selector 
    \[
    \Phi : \Pc_2^{N-1}(\H) \times \Pc_2^{N-1}(\H) \to \Pc_2^{N-1}(\H \times \H) :\quad (p,q) \mapsto \Pi_{p,q}
    \]
    for the right-hand side of \eqref{eq:DPP_CplN_MC.1}.
    Define $\Pi := \Phi_\# \pi \in \Pc_2^N(\H \times \H)$. 
    
    The marginal constraint $\Pi_{p,q} \in \cpl^{N-1}(p,q)$ amounts to $\Pc^{N-1}[\pr_i] \circ \Phi = \pr_i$ for $i \in \{1,2\}$.
    Therefore, we obtain that
    \begin{align*}
        \Pc^N[\pr_1](\Pi) = \Pc^N[\pr_1]( \Pc[\Phi](\pi)) = \Pc[  \Pc^{N-1}[\pr_1] \circ \Phi ](\pi) = \Pc[ \pr_1 ](\pi) = P,
    \end{align*}
    and with the same argument $\Pc^N[\pr_2](\Pi) = Q$, that is, $\Pi \in \cpl^N(P,Q)$.
    Note that $\int \Pi_{p,q} \, \pi(dp,dq) = I ( \Phi_\# \pi ) = I \Pi$.
    Hence, by optimality of $\pi$,
    \begin{align*}
        \mc(P,Q) &= \int \mc(p,q) \, \pi(dp,dq)
        \\
        &= \iint x \cdot y \, I^{N-2}\Pi_{p,q}(dx,dy) \, \pi(dp,dq) \\
        &= \int x \cdot y \, I^{N-1}\Pi(dx,dy).
    \end{align*}
    This completes the induction step and proves \eqref{eq:DPP_CplN_MC}.
\end{proof}

\section{Probabilistic representation of iterated probability measures}
\label{sec:AdLift}

The adapted Lions lift is designed to provide a Hilbert space representation of functionals on $\mathcal{P}_2^{N}(H)$ and will play a central role in analyzing the structure of $\W_2^2$-optimal couplings of measures on $\mathcal{P}_2^{N}(H)$. In this section, we introduce the necessary notation and establish basic properties of the representation of iterated probability measures in terms of filtrations. These results form the foundation for the study of the adapted Lions lift in the subsequent \Cref{sec:MC_cx_ad}.

Recall from \Cref{sec:BrenierAndLiftsIntro} that $(\F_t)_{t=1}^N$
denotes the coordinate  filtration\footnote{Note that instead of $([0,1]^N, \lambda, (\F_t)_{t=1}^N)$ we could also work with an abstract filtered space
$(\Omega,\G,\P,(\G_t)_{t=1}^N)$ where  $(\Omega,\G,\P)$ is standard Borel and for each $t<N$ there exists a continuous $\G_{t+1}$-measurable random variable which is independent of $\G_t$.} on  $([0,1]^N,  \lambda)$ and   that  
\begin{align}\label{eq:adlaw6}
 \lawad(X)=\law(\law( \ldots \law(X| \F_{N-1}) \ldots | \F_1)) \in \Pc_2^N(H),\quad  \quad   X\simad Y \text{ if }  \lawad(X) = \lawad(Y)\end{align} 
 for $X,Y \in L_2([0,1]^N; H).$
In this section we show that every $P\in \Pc_2^N(\H)$ can be represented in this way, establish an adapted transfer priniciple and derive a Skorokhod representation theorem that links convergence on $L_2([0,1]^N; H)$ and $\Pc_2^N(H^N)$.
 
 To this end, it will be necessary to disentangle towers of measures / iterated conditional laws  which will require considerable notation which we now start to set up. In this section $\Xs$ will always denote a Polish metric space. We will use the notation \eqref{eq:adlaw6} also for  $\Xs$ instead of the separable Hilbert space $H$ and we employ the shorthand  
 $$L_2^N(\Xs):= L_2([0,1]^N; \Xs).$$
For $t \in \{1,\dots,N\}$ we write $ \law^{\F_t}(X) = \law(X|\F_t)$. Writing multiple $\sigma$-algebras in the superscript indicates iterated conditional expectations, e.g.\ $  \law^{\F_{2:1}}(X) =  \law^{\F_2,\F_1}(X) = \law(\law(X|\F_2)|\F_1)$. Moreover, we use that abbreviations
\begin{align}\label{eq:IpDef}
\begin{split}
    \ip_t(X) &= \law^{\F_{(N-1):t}}(X) \qquad t=1,\dots, N-1\\
    \ip_N(X) &= X.
\end{split}
    \end{align}
Note that $\ip_t(X)$ takes values in the space
 $
    A_t := \Pc_2^{N-t}(\H),
$ 
where we use the convention $\Pc_2^0(\H)=\H$. 
We  write 
\begin{align*}
    \ip_{1:t}(X) = (\ip_1(X),\dots, \ip_t(X)), \quad  \ip(X) = \ip_{1:N}(X),
    \quad  A_{1:t} = A_1  \times \dots \times A_t
\end{align*}
and note that $\lawad(X) = \law( \ip_1(X)  )$.

We also recall from the introduction that $T:[0,1]^N\to [0,1]^N$ being bi-adapted means that $T$ is a bijection for which $T, T^{-1}$ are $(\F_t$-$\F_t)$-measurable for every $t\leq N$.
\begin{lemma}\label{lem:iso_pres_ip}
Let  $T: [0,1]^N\to [0,1]^N $ be bi-adapted with $T_{\#}\lambda= \lambda$. Then for $s\le t, X\in L_2^N(H)$
$$\law^{\F_t,\dots, \F_s}(X \circ T^{-1}) = \law^{\F_t, \dots, \F_s}(X) \circ T^{-1} .$$ 
In particular,  $\ip_t(X \circ T^{-1}) = \ip_t(X) \circ T^{-1}$ and $\lawad(X \circ T^{-1}) = \lawad(X)$.
\end{lemma}
\begin{proof}
We first show that for every random variable $X \in L_2^N( \Xs)$ and every $t$, we have $\law(X \circ T^{-1} |\F_t) = \law(X |\F_t) \circ T^{-1}$.  To that end, let $f : \Xs \to \R$ be bounded and Borel, and write 
$$
Y := \E[f(X \circ T^{-1}) | \F_t] \circ T
$$
and note that $Y$ is $\F_t$-measurable. Moreover, we find for every bounded $\F_t$-measurable random variable $Z$ using that $Z \circ T^{-1}$ is $\F_t$-measurable
\begin{align*}
    \E[Y Z] = \E[ \E[f(X \circ T^{-1}) | \F_t]  Z \circ T^{-1}  ] = \E[ f(X \circ T^{-1}) Z \circ T^{-1} ] = \E[ f(X) Z ].
\end{align*}
Hence, $Y = \E[f(X)|\F_t]$. 
As this is true for every bounded Borel $f : \Xs \to \R$, we conclude $\law(X \circ T^{-1} | \F_t) = \law(X |\F_t) \circ T^{-1}$.

Since this claim is proven for random variables with values in an  Polish space, we can iteratively apply it and obtain
\[
\law^{\F_t,\dots, \F_s}(X \circ T^{-1}) =   \law^{\F_t,\dots, \F_{s+1}}( \law^{\F_s}(X) \circ T^{-1}) =     \cdots = \law^{\F_t, \dots, \F_s}(X) \circ T^{-1}. \qedhere
\]
\end{proof}

\begin{proposition}\label{prop:ReprLawadAsRV}
    Let $P \in \Pc_2^N(\mathcal{X})$. Then there is $X \in L_2^N(\Xs)$ such that $\lawad(X)=P$.  
\end{proposition}
\begin{proof}
We assume w.l.o.g. that $\mathcal{X}$ is uncountable and  we denote with $(U_t)_{t=1}^N$ the coordinate process on $[0,1]^N$. We show by induction on $N$ that there is a Borel map 
\begin{align*}
    \Phi_N : \Pc_2^N(\mathcal{X}) \times  [0,1]^N \to \mathcal{X}
\end{align*}
such that $\lawad(\Phi_N(P,U_{1:N}))=P$ for every $P \in \Pc_2^N(\R^d)$.

To show the claim for $N=1$, let $\Psi: \R \to \mathcal{X}$ be a Borel isomorphism. If $\mu \in\Pc_2(\R)$, we write $Q_\mu$ for its quantile function. We set 
\[
\Phi_1(P,U_1) = \Psi(Q_{\Psi^{-1}_\#P}(U_1))
\]
and observe that 
\[
\law( \Phi_1(P,U_1) ) = \Psi_\# \law( (Q_{\Psi^{-1}_\#P}(U_1) ) = \Psi_\#\Psi^{-1}_\#P=P.
\]
Now suppose that the claim is true for $N-1$. By the claim for $N=1$ applied with the metric space $\Pc_2^{N-1}(\mathcal{X})$ there is 
\[
\Phi_1 : \Pc_2^N(\mathcal{X}) \times [0,1] \to \Pc_2^{N-1}(\mathcal{X})
\]
such that $\law( \Phi_1(P,U_1)) = P$. 
By the claim for $N-1$ there is 
\[
\Phi_{N-1} : \Pc_2^{N-1}(\mathcal{X}) \times [0,1]^{N-1} \to \mathcal{X}
\]
such that $\law(\law^{\F_{N-1:2}}(\Phi_{N-1}(p,U_{2:N}))) = p$ for every $p \in \Pc_2^{N-1}(\mathcal{X})$. 
We then set
\[
\Phi_N(P,U_{1:N}) := \Phi_{N-1}(\Phi_1(P,U_1),U_{2:N}).
\]
As $\F_1=\sigma(U_1)$ is independent of $U_{2:N}$ and $\Phi_1(P,U_1)$ is $\F_1$-measurable, 
$
\law^{\F_{N-1:1}}( \Phi_N(P,U_{1:N}) ) = \Phi_1(P,U_1).
$
Hence, $\lawad(\Phi_N(P,U_{1:N})) = \law(\Phi_1(P,U_1)) =P$.
\end{proof}

Note that $\cpl(\lawad(X),\lawad(Y))$ does not capture the entire information of the joint adapted distribution of $X$ and $Y$. For this reason we introduce the following notion.

\begin{lemma} \label{lem:cplN}
For $P, Q \in \Pc_2^N(\H)$ we have 
\begin{align*}
    \cpl^N(P,Q) &=  \big\{   \lawad( X,Y )   : X \simad  P, Y \simad Q  \big \}. 
\end{align*}
\end{lemma}
\begin{proof}
If $\Pi \in \cpl^N(P,Q)$, then by \Cref{prop:ReprLawadAsRV} there is an $\H \times \H$-valued random variable $(X,Y)$ with $\lawad(X,Y)=\Pi$. We then have $\lawad(X) = \lawad(\pr_1(X,Y))= \Pc^N[\pr_1](\Pi)=P$ and with the same argument $\lawad(X) = Q$. Conversely, if $X \simad P$ and $Y \simad Q$, then $\Pc^N[\pr_1](\lawad(X,Y))=\lawad(X)=P$ and $\Pc^N[\pr_2](\lawad(X,Y))=\lawad(Y)=Q$.
\end{proof}

\begin{corollary}\label{cor:MC_max_over_RV}
For $P, Q \in \Pc_2^N(\H)$ we have 
    \[
    \mc(P,Q) = \max_{ X \simad P, Y \simad Q} \E[ X \cdot Y].
    \]  
\end{corollary}
\begin{proof}
By \Cref{prop:DPP_CplN_MC}, there is $\Pi \in \cpl^N(P,Q) \subset \Pc_2^N(\H \times \H)$  optimal for $\mc(P,Q)$. By \Cref{prop:ReprLawadAsRV}, there exists an $\H\times\H$-valued random variable $(X,Y)$ such that $\lawad(X,Y)=\Pi$. In particular, we have $\lawad(X) = \lawad( \pr_1(X,Y)) = \Pc^N[\pr_1](\lawad(X,Y)) = \Pc^N[\pr_1](\Pi) =P  $ and with the same argument $\lawad(Y)=Q$.
\end{proof}

\begin{lemma} \label{lem:DPP}
    Let $P,P' \in \mathcal P^N(\mathcal X)$ and $X,X'$ be random variables with $X\simad P$, $X' \simad P'$.
    Let $c^{(0)} := c : \mathcal X \times \mathcal X \to \R$ be continuous and $|c(x,x')| \le f(x) + g(x')$ for $f(X),g(X') \in L^1(\mathbb P)$.
    Define recursively for $t = 1,\dots,N$
    \[
    \begin{aligned}
    c^{(t)} \colon &\Pc^{t}(\mathcal{X}) \times \Pc^{t}(\mathcal{X}) \to \R,\\
    &(Q,Q') \mapsto \inf_{\pi \in \cpl(Q,Q')} \int c^{(t-1)}(q,q')\, \mathrm{d}\pi(q,q').
    \end{aligned}
    \]    
    Then, for $X \simad P$ 
    \[
        c^{(N)}(P,P') =
        \inf_{ 
        \substack{ T : [0,1]^N \to [0,1]^N \\  \text{bi-adapted, } T_{\#}(\lambda)=\lambda }}  
        \E[c(X \circ T,X')].
    \]
\end{lemma}
The proof of \Cref{lem:DPP} is postponed to \Cref{sec:app}.

\begin{corollary}\label{cor:transfer_ad}   
    Let $X,X' \in L_2^N(\mathcal X)$ with $\lawad(X) = \lawad(X')$.
    Then, for every $\epsilon > 0$, there exists a measure-preserving bi-adapted bijection $T : [0,1]^N \to [0,1]^N$ such that $\lambda(|X' - X \circ T| \ge \epsilon) < \epsilon$.
\end{corollary}

\begin{proof}
    Using that $\lawad(X) = \lawad(X')$, the claim follows directly from \Cref{lem:DPP} applied with $\mathcal X := (0,1) \times \H$ and $c(x,x'):= \min(d(x,x'),1)$. 
\end{proof}

\begin{proposition}\label{prop:iteratedSkorohod} For $P,Q \in \Pc_2^N(\H)$ we have 
 \begin{align*}\label{eq:W2_min_RV}
     \W_2(P,Q) = \min_{ X \simad P, \, Y \simad Q}    \|X-Y\|_2.
 \end{align*}
Let $P,P_1,P_2,\ldots \in \Pc_2^N(\H)$ and $X \simad P$. Then the following are equivalent: 
\begin{enumerate}
    \item $\W_2(P_n,P) \to 0$,
    \item there are $X_n \simad P_n$ such that $\|X_n -X \|_2 \to 0$.
\end{enumerate}
In particular, the map $\lawad :  (L_2^N(\H),\|\cdot\|_2) \to (\Pc_2^N(\H),\W_2) $ is continuous.
\end{proposition}

\begin{proof}
First note that \Cref{cor:MC_max_over_RV} and \Cref{lem:W2_MC_connection_iterated} imply that for $P,Q \in \Pc_2^N(\H)$ we have 
 \begin{align*}
     \W_2(P,Q) = \min_{ X \simad P,\, Y \simad Q}    \|X-Y\|_2.
 \end{align*}
To see that (1) implies (2) first observe that  
applying this fact to $P$ and $P_n$ yields $Y_n,Z_n \in L_2^N(H)$ with $Y_n \simad P$ and $Z_n \simad P_n$ such that $\W_2(P_n,P) = \| Y_n -Z_n \|_2$.
By \Cref{cor:transfer_ad} there are isomorphisms $T_n$ such that $\| Y_n \circ T_n - X \|_{L_2} \to 0$. 
We set $X_n := Z_n \circ T_n$ and not that by \Cref{lem:iso_pres_ip}, we have $\lawad(X_n) = \lawad(Z_n)=P_n$.  Moreover, we have
\[
\| X_n - X \|_2 =  \| Z_n \circ T_n - Y_n \circ T_n +  Y_n \circ T_n - X\|_2  \le \| Z_n - Y_n \|_2 + \| Y_n \circ T_n - X\|_2 \to 0.
\]    

Next, we show that (2) implies (1). To that end, we show by backward induction on $t$ that $\ip_t(X_n)\to \ip_t(X)$ in $L_2^N(\Pc_2^{N-t}(\H))$. For $t=N$ this trivial. Suppose it is true for $t+1$. Then,
\begin{align*}
    \W_2^2(\ip_t(X_n),\ip_t(X)) &= 
    \W_2^2(\law^{\F_t}(\ip_{t+1}(X_n)),\law^{\F_{t}}(\ip_{t + 1}(X))) 
    \\
    &\le \E[\W_2^2(\ip_{t+1}(X_n),\ip_{t+1}(X))|\F_t] \quad \text{a.s.},
\end{align*}
where the inequality follows from the observation that $\law^{\F_t}(\ip_{t+1}(X_n), \ip_{t+1}(X))$ is a.s.\ a coupling in $ \cpl(\ip_t(X_n),\ip_t(X))$.
Hence, 
\[
\E[ \W_2^2(\ip_t(X_n),\ip_t(X)) ] \le \E[\W_2^2(\ip_{t+1}(X_n),\ip_{t+1}(X))] \to 0.
\]
This finishes the induction and in particular we have $\ip_1(X_n) \to \ip_1(X)$ in $L_2^N(\Pc_2^{N-1}(\H))$. Hence, $\lawad(X_n) = \law(\ip_1(X_n)) \to \law(\ip_1(X)) = \lawad(X)$.
\end{proof}

\section{The adapted Lions lift}\label{sec:MC_cx_ad}

In this section we develop the basic theory of adapted Lions lift and its applications to $\MC$-convex functions. 
\begin{tcolorbox}
In the following we write $L_2^N(H):= L_2([0,1]^N; H)$  and 
for  $\phi : \Pc_2^N(\H) \to (-\infty, +\infty]$, the function $\overline{\phi} : L_2^N(H) \to (-\infty, +\infty]$  always denotes the adapted Lions lift $\overline\phi(X)= \phi(\lawad(X)) $  of $\phi$. 
\end{tcolorbox}
We will first establish the crucial  result \Cref{thm:MC_trafo_ast_ad} on the identification of $\MC$-conjugacy and convex conjugacy on $L_2^N(H)$. In  particular this allows us to prove \Cref{thm:coupling_Brenier} on the characterization of optimizers in terms of adapted-law invariant convex functions. We also introduce the notion of $\mc$-differentiability of $\mc$-functions which yields a characterization of transport regularity.

\subsection{Connection between $\mc$-transform and convex conjugate of the lift}

Note that for $\phi : \Pc_2^N(\H) \to (-\infty, +\infty]$, we  have 
$ 
\dom(\overline{\phi}) = \{ X \in L_2^N(H) : \lawad(X) \in \dom(\phi) \} 
$ and we  will see in \Cref{lem:continuity_lift} below that the corresponding statement holds also for the continuity points. 
\begin{theorem}\label{thm:MC_trafo_ast_ad}
Let $\phi : \Pc_2^N(\H) \to (-\infty, +\infty]$ be proper. Then 
\begin{align}\label{eq:crucial}
  \overline{\phi}^\ast = \overline{\phi^{\mc}}.  
\end{align}
In particular, the convex conjugate of an adapted-law invariant function is adapted-law invariant.

Moreover, the following are equivalent:
\begin{enumerate}
    \item $\overline{\phi}$ is lsc convex;
    \item $\phi$ is $\mc$-convex.
\end{enumerate}
\end{theorem}
\begin{proof}
For $Y \in L_2^N(\H)$ we write $Q =\lawad(Y)$. By \Cref{lem:DPP} we find
\begin{align*}
   \overline{\phi\,}^\ast(Y) &= \sup_{X \in L_2^N(H) } \E[ X \cdot Y ] - \overline{\phi}(X) = \sup_{ P \in \Pc_2^N(\H)}  \sup_{X \simad P} \E[X \cdot Y]   -  \phi(P) \\
   &= \sup_{P \in \Pc_2^N(\H) } \mc(P,Q) - \phi(P)  = \phi^{\mc}(Q) = \overline{\phi^{\mc}}(Y).
\end{align*}

Next, we prove the equivalence of (1) and (2).

Suppose that $\overline{\phi}$ is lsc convex. By first applying the Fenchel--Moreau theorem and then \eqref{eq:crucial} twice we find
\[
\overline{\phi} = \overline{\phi}^{\ast\ast} = \overline{\phi^\mc}^\ast = \overline{\phi^{\mc\mc}}.
\]
Hence, $\phi = \phi^{\mc\mc}$, which shows that $\phi$ is $\mc$-convex. 

Conversely assume that $\phi$ is $\mc$-convex. Then $\phi = \psi^{\mc}$ for some proper function $\psi : \Pc_2^N(\H) \to (-\infty,+\infty]$. By \eqref{eq:crucial}, we have $\overline{\phi} = \overline{\psi^\mc} = \overline{\psi}^\ast$, which shows that $\overline{\phi}$ is lsc convex. 
\end{proof}

\subsection{$\MC$-order}
In this section we introduce an order on $\Pc_2^N(H)$  which will subsequently be  used to establish the existence of the Lions derivative in the adapted setting. 

\begin{definition}\label{def:MCorder}
Let $P,Q \in \Pc_2^N(\H)$. We say that $P$ is smaller than $Q$ in the $\mc$-order, written as $P \lemc Q$ if $\int \phi \,  d P \le \int \phi \, dQ$ for every $\mc$-convex function $\phi : \Pc_2^{N-1}(\H) \to (-\infty,+\infty]$. 
\end{definition}
We recall the convention $\Pc_2^0(\H)=\H$ and note that the $\mc$-order extends the classical convex order on $\Pc_2(H)$.

 The main result of this subsection is the following characterization of the $\mc$-order, which is an analogue of the results of Carlier \cite{Ca08} (see also Wiesel--Zhang \cite{WiZh22}) for the classical convex order.

\begin{proposition}\label{prop:MC_order:_char}
Let $P,Q \in \Pc_2^N(\H)$. Then the following are equivalent:
\begin{enumerate}
    \item $P \lemc Q$,
    \item $\mc(P,R) \le \mc(Q,R)$ for every $R \in \Pc_2^N(\H)$. 
\end{enumerate}
\end{proposition}
We record a crucial consequence of \Cref{prop:MC_order:_char}.
\begin{corollary}\label{cor:MCincrMC}
For every $R \in \Pc_2^N(\H)$ the function $\mc(\cdot,R)$ is increasing in the $\lemc$-order.    
\end{corollary}

\begin{proof}[Proof of  \Cref{prop:MC_order:_char}]
We first show that (1) implies (2). To that end, let $P \lemc Q \in \Pc_2^N(\H)$ and fix $R \in \Pc_2^N(\H)$. Let $\phi$ be a dual optimizer for $\mc(Q,R)$. We then have 
\[
\mc(P,R) \le \int \phi \,dP + \int \phi^{\mc} \, dR \le \int \phi \,dQ + \int \phi^{\mc} \, dR = \mc(Q,R).
\]

To establish the reverse implication suppose that (1) is not satisfied and assume first that $N\geq 2$. Then there exists a $\mc$-convex $\phi : \Pc_2^{N-1}(\H) \to (-\infty, +\infty]$ such that $\int \phi \, dP > \int \phi \, dQ$.
By approximating $\phi$ with $\mc$-convex functions of the form $(\phi^\mc + \chi_A)^\mc$ for bounded $A \subseteq \Pc^{N-1}_2(\H)$, we can additionally assume that $\phi$ is Lipschitz and, in particular, finitely valued.
Let $f^\epsilon : \Pc^{N-1}_2(\H) \to \Pc^{N-1}_2(\H)$ be an $\epsilon$-selection of $\phi^{\mc\mc}$, i.e.,
\[
    \phi(p) + \phi^\mc(f^\epsilon(p)) \le \mc(p,f^\epsilon(p)) + \epsilon, \quad p \in \Pc^{N-1}_2(\H).
\]
Let $2 \epsilon = \int \phi \, dP - \int \phi \, dQ$.
Define the measure
\[
    R_\epsilon := (f^\epsilon)_\# P.
\]
Then, we have
\begin{align*}
    \MC(P,R_\epsilon) \ge P(\phi) + R_\epsilon(\phi^\mc) - \epsilon > Q(\phi) + R_\epsilon(\phi^\mc) \ge \MC(Q,R_\epsilon),
\end{align*}
which yields a contradiction.

In the case $N=1$, the same argument applies, replacing $\Pc^{N-1}_2(\H)$ by $H$ and $\mc$ by the inner product.
\end{proof}

\begin{proposition}
Let $\phi : \Pc^N_2(\H) \to (-\infty, \infty]$ be $\mc$-convex. Then $\phi$ is increasing in the $\lemc$-order.    
\end{proposition}
\begin{proof}
    This follows because $\mc(\cdot, R)$ is increasing in the $\lemc$-order (see \Cref{cor:MCincrMC}) and every $\mc$-convex function is by definition supremum of such functionals (plus constants). 
\end{proof}

\subsection{$N$-Monge Couplings}\label{sec:NMonge}
In this section we use $\Xs, \Ys$ to denote Polish metric spaces.

Recall from \eqref{eq:IpDef} that $\ip_t= \law^{\F_{(N-1):t}}(X) \in     A_t = \Pc_2^{N-t}(\H)$. 
\begin{definition}
Let $P,Q \in \Pc_2^N(\X)$. We say that $\Pi \in \cpl^N(P,Q)$ is $N$-Monge if there exists a map $\xi : A_{1:N} \to \X$ such that, whenever $(X,Y) \simad \Pi$, we have $Y = \xi(\ip(X))$. 
\end{definition}
Crucially, being $N$-Monge is equivalent to the fact that $\ip(Y) = T(\ip(X))$ as it is possible to `unfold' $\xi: A_{1:N} \to \X$ to a map $T: A_{1:N} \to A_{1:N}$ that satisfies $\ip(Y) = T(\ip(X))$.

\begin{proposition} \label{prop:N-Monge.ip}
Let $P \in \Pc_2^N(\mathcal X)$ and let $\xi : \prod_{t=1}^N\Pc^{N-t}(\mathcal X) \to \mathcal Y$ be measurable. 
Write $Y = \xi(\ip(X))$ and $Q := \lawad(Y)$ and inductively define mappings $T_N := \xi$ and, for $t=N-1,\dots,1$,
\begin{align*}
   T_t : A_{1:t} \to \Pc_2^{N-t}(\X) ,\quad 
   T_t(p^1, \dots, p^t) := T_{t+1}(p^1, \dots, p^t, \cdot)_\# p^t.
\end{align*}
Then, for every $t = 1,\dots, N$, we have $\ip_t( Y)  = T_t(\ip_{1:t}(X))$ almost surely.
\end{proposition}

\begin{proof}
We show the claim by backward induction. 
For $t=N$, we trivially have 
$$\ip_N(Y) = Y = \xi(\ip(X)) = T_N(\ip_{1:N}(X)).$$ 
Next, suppose the claim is true for $t+1$.
We have almost surely that
\begin{align*}
    \ip_t(Y) &= \law(\ip_{t+1}(Y) | \F_t) = \law(T_{t+1}(\ip_{1:t+1}(X)) | \F_t)
    \\
    &= (T_{t+1}(\ip_{1:t}(X),\cdot))_\# \law(\ip_{t+1}(X) | \F_t)
    \\
    &= (T_{t+1}(\ip_{1:t}(X),\cdot))_\# \ip_t(X) = T_t(\ip_{1:t}(X)),
\end{align*}
where the first, fourth and fifth equality hold by definition, the second by inductive assumption and the third by $\F_t$-measurability of $\ip_{1:t}(X)$.
This yields the claim.
\end{proof}

\begin{corollary}
    \label{cor:N-Monge.uniqueness}
    Let $P,Q \in \Pc_2^N(\H)$. Let $\Pi^1,\Pi^2 \in \cpl^N(P,Q)$ be $N$-Monge. Then, 
    \[ (\Pi^1 + \Pi^2)/2 \text{ is $N$-Monge} \iff \Pi_1 = \Pi_2. \]
\end{corollary} 

\begin{proof}
    Note that the $\Longleftarrow$-implication is trivial.
    
    To see the $\Longrightarrow$-implication, apply \Cref{prop:N-Monge.ip} with $\mathcal X = H$, $\mathcal Y = H \times H$ and $\xi(\ip(X)) = (X,\xi^i(\ip(X)))$, where $\xi^i$ is the corresponding map associated with the $N$-Monge coupling $\Pi^i$, $i = 1,2$.
    Using the assertion of \Cref{prop:N-Monge.ip} with $t = 1$, we obtain Borel maps $T^i$ with $T^i_\# P = \Pi^i$.
    As $(\Pi^1 + \Pi^2)/2$ is assumed to be $N$-Monge, we deduce (by analogous reasoning as above) that $(\Pi^1 + \Pi^2)/2$ is given by pushforward of some measurable map.
    This enforces that $T^1 = T^2$ $P$-almost surely, and hence $\Pi^1 = \Pi^2$.
\end{proof}

\begin{definition} 
We call $(\mu,\nu) \in \Pc_2(\H) \times \Pc_2(\H)$ a strict Monge pair if the optimal coupling between $\mu$ and $\nu$ is unique and of Monge type.

For $N \ge 2$, we call
$(P,Q)  \in \Pc_2^N(\H) \times \Pc_2^N(\H) $ a strict Monge pair if
\begin{itemize}
    \item the optimal coupling between $P$ and $Q$ is unique and of Monge type, i.e.\ $\cplopt(P,Q) = \{ (\id,T)_\# P \}$ for some $T : \Pc^{N-1}_2(\H) \to \Pc^{N-1}_2(\H)$;
    \item For $P$-almost every $p$, the pair $(p,T(p))$ is a strict Monge pair.
\end{itemize}
\end{definition}

\begin{proposition} \label{prop:N-Monge.strict-Monge}
Let $P,Q \in \Pc_2^N(\H)$. Then, $(P,Q)$ is a strict Monge pair if and only if $\cplopt^N(P,Q)= \{\Pi\}$ and $\Pi$ is $N$-Monge.  
\end{proposition}

\begin{proof}
    First, suppose that $(P,Q)$ is a strict Monge pair. We need to show that there exists a unique optimizer $\Pi \in \cpl^N(P,Q)$ and that it is $N$-Monge. To this end, we show that there is $\xi : A_{1:N} \to \H $ such that for every optimizer $\Pi \in \cpl^N(P,Q)$ and random variables $X,Y$ with $(X,Y)\simad \Pi$ it holds $Y = \xi(\ip(X))$ almost surely. 
    
    We show the claim by induction on $N$. For $N=1$ it is clear. Suppose that the claim holds for $N-1$.  That is,  if $(p,q) \in \Pc_2^{N-1}(H) \times \Pc_2^{N-1}(H)$ is a strict Monge pair then there is $\xi^{p,q}: A_{2:N} \to H$ such that for all $X',Y' \in L_2^{N-1}(H)$ with $\ip_1(X',Y') \in \cplopt^{N-1}(p,q)$ we have $Y' = \xi^{p,q}(\ip(X'))$ almost surely.
    
    Fix an optimizer $\Pi \in \cplopt^N(P,Q)$ and random variables  $(X,Y)\simad \Pi$. We then have a.s.
    \[
        \Pc^{N-1}[\pr_1](\ip_1(X,Y)) = \ip_1(X) \sim P \text{ and } 
        \Pc^{N-1}[\pr_2](\ip_1(X,Y)) = \ip_1(Y) \sim Q.
    \]
    Therefore, we get
    \begin{align*}
        \MC(P,Q) &= \int x \cdot y \, dI^{N-1}\Pi(x,y) = \E\Big[ \int x \cdot y \, I^{N-2}\ip_1(X,Y)(dx,dy) \Big]
        \\
        &\le \E\big[ \mc(\ip_1(X),\ip_1(Y)) \big] \le \MC(P,Q).
    \end{align*}
    Thus, all inequalities are equalities. In particular, $\law(\ip_1(X),\ip_1(Y)) \in \cplopt(P,Q)$ and almost surely $\ip_1(X,Y) \in \cplopt^{N-1}(\ip_1(X),\ip_1(Y))$.
    As $(P,Q)$ is a strict Monge pair, this yields $\ip_1(Y) = T(\ip_1(X))$ for some map $T: \Pc_2^{N-1}(H) \to \Pc_2^{N-1}(H)$ and there is $\lambda$-full set $\Omega' \in \F_1$ such that for every $\omega \in \Omega'$ we have that 
    \begin{gather*}
        (\ip_1(X)(\omega),\ip_1(Y)(\omega)) \in \Pc_2^{N-1}(\H ) \times \Pc_2^{N-1}(\H ) \text{ is a strict Monge pair},
        \\
        \ip_1(X,Y)(\omega) \in \cplopt^{N-1}(\ip_1(X)(\omega),T(\ip_1(X)(\omega))).
    \end{gather*}
    By the induction hypothesis, we obtain that for all $\omega \in \Omega'$
    \begin{align*}
        Y(\omega_1,\cdot) = \xi^{\ip_1(X)(\omega), \ip_1(Y)(\omega)}(\ip_{2:N}(X)(\omega_1,\cdot)).
    \end{align*}
    Using this we can define (on a $\law(\ip(X))$-full set) the desired map  $\xi : A_{1:N} \to \H $ via
    \[
        \xi(\ip(X)) := \xi^{\ip_1(X),T(\ip_1(X))}(\ip_{2:N}(X)),
    \]
    which satisfies $\xi(\ip(X)) = Y$ almost surely.

    Conversely, we need to show that if $\cplopt^N(P,Q)= \{\Pi\}$ and $\Pi$ is $N$-Monge that then $(P,Q)$ is a strict Monge pair.
    Again, the claim is trivially satisfied when $N = 1$. Assume that the claim holds for $N-1$. That is, if $(p,q) \in \Pc_2^{N-1} \times \Pc_2^{N-1}$ satisfies $\cplopt^{N-1}(p,q) = \{ \Pi_{p,q} \}$ and $\Pi_{p,q}$ is $(N-1)$-Monge, then $(p,q)$ is strict Monge.
    If there are $\pi_1, \pi_2 \in \cplopt(P,Q)$ with $\pi_1 \neq \pi_2$ and $\tilde \Pi_{p,q} \in \cplopt^{N-1}(p,q)$ is a measurable selection, then $\Pi_i := ((p,q) \mapsto \tilde \Pi_{p,q})_\# \pi_i \in \cplopt^N(P,Q)$. Thus, $\Pi = \Pi_1 = \Pi_2$ which can only be the case if $\pi_1 = \pi_2$.
    Since $\Pi$ is $N$-Monge there is a measurable map $\xi : A_{1:N} \to \mathcal X$ such that $(X,\xi(\ip(X))) \simad \Pi$ for $X \simad P$.
    As in the first part, we have that $\law(\ip_1(X), \ip_1(\xi(\ip(X)))) \in \cplopt(P,Q)$, from where it follows that $\pi_1 = \law(\ip_1(X), \ip_1(\xi(\ip(X))))$ is also of Monge type.
    By the inductive assumption, we have that for almost surely $(\ip_1(X),\ip_1(Y))$ is a strict Monge pair.
    This completes the induction step and concludes the proof.
\end{proof}

\subsection{$\mc$-subdifferentials}
Before considering subdifferentials, we briefly compare continuity on $\Pc_2^N(H)$ and $L_2^N(H).$ \begin{lemma}\label{lem:continuity_lift} 
Let $\phi : \Pc_2^N(\H) \to (-\infty, +\infty]$ and $P \in \Pc_2^N(\H)$. Then the following are equivalent:
\begin{enumerate}
    \item $\phi$ is continuous at $P$.
    \item For all $X \simad P$,  $\overline{\phi}$ is continuous at $X$.
    \item There exists $X \simad P$ such that $\overline{\phi}$ is continuous at $X$.  
\end{enumerate}
In particular, 
\[
\cont(\overline{\phi}) = \{ X \in L_2^N(\H) : \lawad(X) \in \cont(\phi) \},
\]
and hence 
$\cont(\overline{\phi})$ is adapted-law invariant.  
\end{lemma}

With the same arguments as in the proof of \Cref{lem:continuity_lift} it follows that a function $\phi : \Pc_2^N(\H) \to (-\infty, +\infty]$ is lsc if and only if its adapted Lions lift $\overline{\phi}$ is lsc.
\begin{proof}[Proof of \Cref{lem:continuity_lift}]
To see that (1) implies (2) let $X \simad P$ and $\|X_n - X \|_2 \to 0$. Then, $\W_2(\lawad(X_n), \lawad(X)) \to 0$ and hence $\overline{\phi}(X_n) = \phi(\lawad(X_n) \to \phi(\lawad(X)) = \overline{\phi}(X)$. 

The implication from (2) to (3) is trivial.

Suppose that (3) holds true. Let $P_n \to P$ in $\W_2$. By \Cref{prop:iteratedSkorohod} there are $X_n \simad P_n$ such that $\| X_n - X\|_2 \to 0$. Hence, $\phi(P_n) = \overline{\phi}(X_n) \to \overline{\phi}(X) = \phi(P)$. 
\end{proof}

\medskip

We recall the notion of $\mc$-subdifferntials from \Cref{def:MC_subdiff_ad}. For a function  $\phi : \Pc_2^N(\H) \to (-\infty, \infty]$ and $P \in \Pc_2^N(\H)$ its $\mc$-subdifferential at $P$ is defined as 
\[
\partial_{\mc} \phi(P) = \{ Q\in \Pc_2^N(\H) : \phi(R) \ge \phi(P) + \mc(R,Q) - \mc(P,Q) \text{ for every } R \in \Pc_2^N(\H)\}.
\]

\begin{proposition}\label{prop:subdiff_char_ad}
Let $\phi : \Pc_2^N(\H) \to (-\infty, +\infty]$ be proper and let $X,Y \in L_2^N(\H)$. If $X \simad P$ and $Z \simad Q$ we have
\[
Z \in \partial \overline{\phi}(X) \iff Q \in \partial_{\mc} \phi(P) \text{ and } \E[X \cdot Z] = \mc(P,Q) 
\]
\end{proposition}
\begin{proof}
The proof is line by line as the proof of \Cref{prop:Subdiff_MCvsL}.    
\end{proof}

A crucial observation about subdifferentials of law invariant convex functions is that if $Z \in \partial\overline{\phi}(X)$, then $\E[Z | X ] \in \partial\overline{\phi}(X)$. This allowed us to conclude that if $\phi$ is $\mc$-differentiable at $\law(X)$ then $\partial\overline{\phi}(X) = \{ \xi(X) \}$. Our next goal is to establish an analogue of this for adapted-law invariant functionals. Specifically, we show that if $Z \in \partial\overline{\phi}(X)$ then $\E[Z | \ip(X) ] \in \partial\overline{\phi}(X)$ and therefore conclude that $\mc$-differentiability of $\phi$ at $\lawad(X)$ implies that $\partial\overline{\phi}(X) = \{ \xi(\ip(X)) \}$. To this end, we need some auxilary results.

\begin{lemma}\label{lem:MCmonotone} 
    Let $X,Y$ be random variables on $([0,1]^N,\lambda)$ such that $\mathbb E[ Y |  \ip(X) ] = 0$.
    Then, we have $$\lawad(X + Y) \gemc \lawad(X).$$
\end{lemma}

\begin{proof}
    By \Cref{lem:ip-meas-rv} we can assume w.l.o.g.\ that \(
        X \text{ has a }\sigma(Z_t, U_{t+1:N})\text{-measurable version},\)
    where $Z_t := \law^{\mathcal F_{N-1:t}}(X)$, for all $t = 1,\dots,N-1$.
    Consider the family of random variables
    \[
        Y_0 := Y, \quad Y_t := \E[ Y | Z_{1:t}, U_{t+1:N}], \quad Y_N := \E[ Y | X, Z_{1:N-1} ] = 0.
    \]
    We claim that $\lawad(X + Y_t) \gemc \lawad(X + Y_{t + 1})$ for $t = 0, \dots, N-1$, which we proceed to show by induction.

    Let $\phi$ be an $\MC$-convex function, then we have as $\overline{\phi}$ is convex and lsc
    \[
        \E[\overline{\phi}(X + Y_0) | Z_1, U_{2:N}] = \overline{\phi}(\E[ X + Y | Z_1, U_{2:N}] ) = \overline{\phi}(X + \E[Y | Z_1,U_{2:N}] = \overline{\phi}(X + Y_1).
    \]
    By taking the expectation on both sides, the resulting inequality shows the claim for $t = 0$.

    Next, assume that the claim holds true for $t < N$. Again, let $\phi$ be $\MC$-convex, then
    \begin{align*}
        \E[\overline{\phi}(X + Y_t) | Z_{1:t}, U_{t + 1: N}] &\ge
        \overline{\phi}(\E[X + Y_t | Z_{1:t}, U_{t + 1 : N}]) 
        \\
        &= \overline{\phi}(X + \E[Y_t | Z_{1:t}, U_{t + 1 : N}]) =
        \overline{\phi}(X + Y_{t + 1}).
    \end{align*}
    As above we conclude that $\lawad(X + Y_t) \gemc \lawad(X + Y_{t+1})$ and the assertion follows by transitivity of $\lemc$ and noting that $X + Y_1 = X + Y$ as well as $X = X + Y_N$.
\end{proof}

\begin{proposition}\label{prop:sub_grad_proj_ad}
Let $\overline{\phi} : L_2^N(\H) \to (-\infty, +\infty]$ be adapted-law invariant and convex. If $Z \in \partial \overline{\phi}(X)$, then 
\[
\E[Z |  \ip(X) ] \in \partial \overline{\phi}(X). 
\]
\end{proposition}
\begin{proof} Let $Z \in \partial \overline{\phi}(X)$. For every $Y \in L_2^N(\H)$, we have 
$
\E[ Y -    \E[ Y | \ip(X)] |  \ip(X) ] =0.
$
Hence, by \Cref{lem:MCmonotone} we find
    \begin{align*}
        \overline{\phi}(Y) &\ge \overline{\phi}(  \E[ Y | \ip(X) ])  \\
        & \ge \overline{\phi}(X) + \E [    (\E[ Y |  \ip(X) ] - X)  \cdot Z   ]  \\
        &= \overline{\phi}(X) + \E [   \E[ Y - X | \ip(X)]   \cdot Z   ] \\
        &= \overline{\phi}(X) + \E [ (Y-X) \cdot \E[ Z | \ip(X)]].
    \end{align*}
    Hence, $\E[ Z |\ip(X)] \in \partial\overline{\phi}(X)$. 
\end{proof}

\begin{definition}\label{def:MCdiff_ad}
    Let $\phi:\Pc_2^N(\H)\to (-\infty, \infty]$ be $\MC$-convex and $P\in \dom(\phi)$. We say that $\phi $ is $\MC$-differentiable at $P$ if $|\partial \phi(P)|=1$.
\end{definition}

\begin{lemma}\label{lem:mc_diff_char_ad}
    Let $\phi:\Pc_2^N(\H)\to (-\infty, \infty]$ be $\MC$-convex and $P\in \dom(\phi)$. Then the following are equivalent: 
    \begin{enumerate}
        \item $\phi$ is $\MC$-differentiable in $P$.
        \item For all $X$ with $X\simad P$ we have $\# \partial \overline \phi (X) =1$.
    \end{enumerate}
\end{lemma}
\begin{proof}
(1) implies (2): Suppose that there are $Z_1 \neq Z_2 \in \partial\overline{\phi}(X)$ for some $X \simad P$. By \Cref{prop:subdiff_char_ad}, we have $\lawad(Z_1),\lawad(Z_2)  \in \partial_\mc\phi(P)$. Hence, if not $(Z_1) \simad Z_2$, we conclude that $\partial_\mc\phi(P)$ is no singleton. Otherwise, observe that $Z:= \frac12 (Z_1+Z_2) \in \partial\overline{\phi}(X)$ and hence $\law(Z) \in \partial_\mc\phi(P)$. As $\| \cdot \|_2^2$ is strictly convex, $\|Z\|_2^2 < \frac12 (\|Z_1\|_2^2+\|Z_2\|_2^2) = \|Z_1\|_2^2$. Hence, $\law(Z) \neq \law(Z_1)$ and in particular the adapted laws are different. This shows that $\partial_\mc\phi(P)$ is no singleton.

(2) implies (1): Let $Q_1,Q_2 \in \partial_\mc \phi(P)$.
By \Cref{prop:ReprLawadAsRV} and \Cref{cor:MC_max_over_RV} there are random variables $X_i,Y_i \in L_2^N(H)$ with $X_i \simad P$ and $Y_i \simad Q_i$ such that $\mc(P,Q_i) = \E[X_i \cdot Y_i]$, for $i = 1,2$.
By \Cref{prop:subdiff_char_ad} we have that $\{Y_i\} = \partial\overline{\phi}(X_i)$ and due to \Cref{prop:sub_grad_proj_ad} there exist measurable functions such that $\xi_i(\ip(X_i)) = Y_i$.
Since $$\mc(P,Q_1) = \E[X_1 \cdot Y_1] =\E[ X_1 \cdot \xi_1(\ip(X_1)) ]= \E[ X_2 \cdot \xi_1(\ip(X_2)) ]$$  Hence, $\xi_1(\ip(X_2)) \in \partial \overline \phi(X_2)$. By assumption $\#\partial \overline \phi(X_2) = 1$ and it follows that $\xi_1(\ip(X_2)) = Y_2$. Thus $Q_1 = \lawad(Y_1) = \lawad(Y_2) = Q_2$. Hence $\phi$ is $\mc$-differentiable at $P$.
\end{proof}

\begin{proposition}\label{prop:LdiffXi_ad}
Let $\phi:\Pc_2^N(\H)\to (-\infty, \infty]$ be $\MC$-differentiable in $P$. Then there exists  a measurable function  $\xi : A_{1:N} \to \H$ such that for all $X\simad P$ 
\[
\partial\overline{\phi}(X) = \{\xi( \ip(X) )\}.
\]
\end{proposition}
\begin{proof}
Let $X \simad P$. Then \Cref{prop:sub_grad_proj_ad} implies that $\partial\overline{\phi}(X) = \{\xi( \ip(X) )\}$. It remains to show that $\xi$ does not depend on the choice of $X \simad P$. To that end, let $X'  \simad P$ and $\xi'$ such that  $\partial\overline{\phi}(X') = \{\xi'( \ip(X') )\}$. By \Cref{prop:subdiff_char_ad}, we have $\lawad(\xi'(\ip(X'))) \in \partial_\mc \phi(P)$ and $\E[ X' \cdot \xi'(\ip(X')) ] = \mc(P,Q)$. As $\E[ X \cdot \xi'(\ip(X))]= \E[ X' \cdot \xi'(\ip(X')) ] = \mc(P,Q) $ and $\lawad(\xi'(\ip(X))) = \lawad(\xi'(\ip(X'))) \in \partial_\mc\phi(P)$,  \Cref{prop:subdiff_char_ad} implies that $\xi'(\ip(X)) \in \partial\overline{\phi}(X)$. As $\partial\overline{\phi}(X)$ is a singleton, we have $\xi = \xi'$ almost surely.
\end{proof}

\subsection{Characterization of strict Monge pairs by differentiability of $\mc$}
In the following we generalize the result of Alfonsi--Jourdain \cite{AlJo20} that characterizes the existence of a unique optimal transport that is induced by a map in terms of Lions differentiability of the Wasserstein distance to the case of iterated probability measures.

\begin{proposition}\label{prop:AJ_adapted}
    For $P,Q \in \Pc_2^N(\H)$ the following are equivalent:
    \begin{enumerate}
        \item $\mc(\cdot,Q)$ is $\mc$-differentiable at $P$.
        \item There is $\xi : A_{1:N} \to \H $ such that for some $X \simad P$
        \[
            \big\{ Y \simad Q : \E[X \cdot Y] = \MC(P,Q) \big\} = \{ \xi(\ip(X)) \} = \partial \overline{\mc(\cdot,Q)}(X).
        \]
        \item For some $X \simad P$, the function $\overline{\mc(\cdot,Q)}$ is Frechet differentiable at $X$.
        \item The pair $(P,Q)$ is a strict Monge pair.
    \end{enumerate}
    Moreover, whenever one (and hence all) of the above conditions holds, statements (2) and (3) are valid for every $X \simad P$, and $\partial_\mc \mc(\cdot,Q)(P) = \{ Q \}$.
\end{proposition}

\begin{proof}
We start with proving the equivalence of (1), (2), (3), and (4).

(4)$\implies$(2):
Assume that (4) holds and let $X \simad P$. By \Cref{prop:N-Monge.strict-Monge} there is $\xi : A_{1:N} \to \H $ such that $\cplopt^N(P,Q) = \{ \lawad(X,\xi(\ip(X))) \}$.
In particular, if $Y \simad Q$ with $\E[X \cdot Y]$ 

$\E[X \cdot \xi(\ip(X))] = \mc(P,Q)$.  Next, we show that $\overline{\mc(\cdot,Q)}$ is Frechet differentiable at $X$ with derivative $\xi(\ip(X))$. 
To this end, let $X_n \to X$ in $L_2$.
We then have 
\begin{equation} \label{eq:prop.AJ_adapted.1}
    \mc(\lawad(X_n), Q) \ge \E[X_n  \cdot \xi(\ip(X))] = \mc(P,Q) + \E[ (X_n-X) \cdot \xi(\ip(X)) ].
\end{equation}
Next, let $Y_n \simad Q$ with $\E[X_n \cdot Y_n ] \ge \mc(\lawad(X_n),Q)- \|X_n-X\|_2 /n$. We then find
\begin{align}
    \label{eq:prop.AJ_adapted.2}
    \begin{split}
    \mc(P,Q) &\ge \E[ X \cdot Y_n ] = \E[ X_n \cdot Y_n ] + \E[  (X-X_n) \cdot Y_n ] \\
    & \ge \mc(\lawad(X_n),Q) + \E[  (X-X_n) \cdot Y_n ] - \| X_n-X\|/n \\
    &= \mc(\lawad(X_n),Q) + \E[(X-X_n) \cdot \xi(\ip(X))] + \E[ (X-X_n) \cdot (Y_n - \xi(\ip(X)) ) ]  
    \end{split}
\end{align}
Rearranging the inequalities \eqref{eq:prop.AJ_adapted.1} and \eqref{eq:prop.AJ_adapted.2} yields
\begin{multline*}
    |  \mc(\lawad(X_n), Q)  - \mc(P,Q) - \E[ (X_n-X) \cdot \xi(\ip(X)) ] | 
    \\
    \le |\E[ (X-X_n) \cdot (Y_n - \xi(\ip(X) ) ) ]| \le \| X_n -X \|_2 \| Y_n - \xi(\ip(X)) \|_2. 
\end{multline*}
Since $\lim_{n \to \infty} \|Y_n - \xi(\ip(X))\|_2 = 0$ by \Cref{lem:analogue.to.AJ.2.5}, we find that $\overline{\mc(\cdot,Q)}$ is Frechet differentiable at $X$ with derivative $\xi(\ip(X))$.

(3)$\implies$(1): This implication is clear because Frechet differentiability of the Lions lift at some $X$ implies, by \Cref{thm:LconvexGdeltaDiff}, its Frechet differentiability at all $X' \simad P$. Hence, $|\{ \partial \overline{\mc(\cdot,Q)}(X')| = 1$ for all $X' \simad P$. Thus, it follows from \Cref{lem:mc_diff_char_ad} that $\mc(\cdot,Q)$ is $\mc$-differentiability at $P$.

(1)$\implies$(4): Assume that $\partial_\mc \mc(\cdot,Q)(P) = \{Q\}$. For any $\Pi \in \cplopt^N(P,Q)$ and $(X,Y) \simad \Pi$, we have that $\mc(P,Q) = \E[X \cdot Y]$ and, by \Cref{prop:subdiff_char_ad} we have also $Y \in 
\partial \overline{\mc(\cdot,Q)}(X)$.
Consequently, by \Cref{prop:sub_grad_proj_ad} we obtain that $\E[Y | \ip(X) ] \in \partial \overline{\mc(\cdot,Q)}(X)$ and, again by \Cref{prop:subdiff_char_ad}, $\lawad(\E[Y|\ip(X)]) \in \partial_\mc \mc(\cdot,Q)(P)$.
Hence, $Q = \lawad(\E[Y|\ip(X)])$ and $\E[|Y|^2] = \E[|\E[Y|\ip(X)]|^2]$, from where we deduce from the equality case of Jensen's inequality that $Y = \E[Y | \ip(X)]$ almost surely, i.e., $Y$ is $\ip(X)$-measurable.
We have shown that every $\Pi \in \cplopt^N(P,Q)$ is $N$-Monge, which, by \Cref{cor:N-Monge.uniqueness}, yields that $\cplopt^N(P,Q)$ contains a single element.
We conclude that $(P,Q)$ is a strict Monge pair by \Cref{prop:N-Monge.strict-Monge}.
\end{proof}

\begin{lemma} \label{lem:analogue.to.AJ.2.5}
    Let $P,Q,P_1,Q_1,P_2,Q_2,\dots \in \Pc^N_2(\H)$ and let $X,Y,X_1,Y_1,X_2,Y_2\dots$ be random variables with $X \simad P$, $Y \simad Q$, $X_n \simad P_n$, $Y_n \simad Q_n$ such that
    \[
        \E[X_n \cdot Y_n] \ge \mc(P_n,Q_n) -\epsilon_n, \quad Q_n \to Q \text{ in }\W_2, \text{ and } \epsilon_n \to 0.
    \]
    If $(P,Q)$ is a strict Monge pair and $X_n \to X$ in $L_2$, then
    \[
        Y_n \to \xi(\ip(X)) \quad \text{in }L_2,
    \]
    where $\xi$ is the unique map with $\law (\ip(X),\xi(\ip(X))) \in \cplopt^N(P,Q)$.
\end{lemma}

\begin{proof}
    Note that since $(P,Q)$ is a strict Monge pair, there exists a unique map $\xi : \Pc^{N-1}_2(\H) \times \dots \times \H \to \H$ with $Y := \xi(\ip(X)) \simad Q$ and $\lawad(X,Y) =: \hat \Pi$ is the unique element in $\cplopt^N(P,Q)$.
    By $\W_2$-relative compactness, up to extracting a subsequence, we can assume that
    \[
        \lim_{n \to \infty} \hat \Pi_n := \lawad(X_n,Y_n) = \tilde \Pi \text{ in }\W_2.
    \]
    Write $c(x,y) = x \cdot y$, by $\W_2$-continuity of $c^{(N)}$ we get
    \begin{align*}
        \int x \cdot y \, dI^{N-1}\hat \Pi(x,y) &= \mc(P,Q) =
        \lim_{n \to \infty}\mc(P_n,Q_n)
        \\
        &= \lim_{n \to \infty} \int x \cdot y \, d I^{N-1} \hat \Pi_n(x,y) = \int x \cdot y \, dI^{N-1}\tilde \Pi(x,y).
    \end{align*}
    Hence, $\tilde \Pi \in \cplopt^N(P,Q)$ and by uniqueness $\tilde \Pi = \hat \Pi$. We have shown that
    \[
        (X_n,Y_n) \to (X,Y) = (X,\xi(\ip(X))) \quad \text{in adapted distribution}.
    \]
    In particular, we have $(\ip(X_n),\ip(Y_n)) \to (\ip(X),\ip(Y))$ in distribution and as $\ip_N(Y_n) = Y_n$ and $\ip_N(Y) = Y$,
    $(\ip(X_n),Y_n) \to (\ip(X),Y) = (\ip(X), \xi(\ip(X)))$.
    Additionally, $(U,\ip(X_n)) \to (U,\ip(X))$ in probability and therefore in distribution. 
    We can invoke \cite[Lemma 6.4]{BePaScZh23} to obtain
    \begin{equation}
        \label{eq:joint.cv}
        (U,\ip(X_n),Y_n) \to 
        (U,\ip(X),Y) \text{ in distribution},
    \end{equation}
    and thus $(U,Y_n) \to (U,Y)$ in distribution.
    As $Y = \xi(\ip(X))$ and $Y_n$ are functions of $U$, \eqref{eq:joint.cv} entails by \cite[Lemma 3.14]{JoPa24} that $Y_n \to Y$ in probability.
    Note that $\E[|Y_n|^2] = \int |y|^2 \, dI^N(Q_n) \to \int |y|^2 \, dI^N(Q) = \E[|Y|^2]$ and therefore $Y_n \to Y$ in $L_2$, which was the claim.
\end{proof}

The Ekeland--Lebourg theorem (see e.g.\ \cite[Theorem 18.3]{BaCo17}) guarantees that a convex function on an infinite dimensonal Hilbert space that is continuous in at least on point is Frechet differentiable on a dense $G_\delta$ subset of the closure of its domain. For our purposes, we need a version of this result that is tailored to adapted-law invariant functions. 

\begin{theorem} \label{thm:LconvexGdeltaDiff}
Let $\overline{\phi} : L_2^N(\H) \to (-\infty, +\infty]$ be adapted-law invariant, convex and continuous at some $\mu \in \dom(\phi)$. Then the set of Frechet-differentiability points of $\overline{\phi}$ is  dense in the closure of $\dom(\overline{\phi})$ and it is the intersection of countable many open adapted-law invariant sets.
\end{theorem}
\begin{proof}[Sketch of the proof of \Cref{thm:LconvexGdeltaDiff}]

We write $A \subset L_2^N(H)$ for the set of Frechet differentiability points of $\overline{\phi}$. The Ekeland--Lebourg theorem \cite[Theorem 18.3]{BaCo17} asserts that $A$ is a dense $G_\delta$ subset of the closure of $\dom(\overline{\phi})$. It remains to observe that the dense open set $O_n$ such that $A = \bigcap_n O_n$ can chosen to be adapted-law invariant. 

By \cite[Proposition 18.1]{BaCo17} we have that $A = \bigcap_{n \in \N} O_n$ with
\[
O_n = \bigcup_{\eta >0} \Big\{  X \in \textup{cont}(\overline{\phi})  : \sup_{ \|Y\|_2=1}  \overline{\phi}(X+ \eta Y) +  \overline{\phi}(X- \eta Y) -2 \overline{\phi}(X) < \frac{\eta}{n} \Big\}.
\]
Note that $\textup{cont}(\phi)$ is adapted-law invariant by \Cref{lem:continuity_lift}. Hence, it is clear that $O_n$ is adapted-law invariant provided that $\overline{\phi}$ was. 
\end{proof}

\begin{remark}\label{rem:quotient_topology}
\Cref{lem:continuity_lift} implies that a map $f : \Pc_2^N(\H) \to \R$ is continuous if and only if $f \circ \lawad : L_2^N(\H) \to \R$ is continuous. This is precisely the universal property of quotient topologies, i.e.\  $\Pc_2^N(\H) =  L_2^N(\H)  / \simad $. Clearly, adapted-law invariant sets are precisely the saturated sets w.r.t.\ $\simad$ and there is a one-to-one correspondence between open sets in $\Pc_2^N(\H)$  and open adapted-law invariant sets in $L_2^N(\H)$. Hence, dense sets in  $\Pc_2^N(\H)$ correspond to adapted-law invariant dense sets in $L_2^N(\H)$ and $G_\delta$ sets in  $\Pc_2^N(\H)$ correspond to subsets of  $L_2^N(\H)$ that are a countable intersection of open adapted-law invariant sets. 
\end{remark}

\begin{corollary}\label{cor:MCdfb_Gdelta}
Let $\phi : \Pc_2^N(\H) \to (-\infty,+\infty]$ and suppose that $\cont(\phi)\neq\emptyset$. Then $\phi$ is $\mc$-differentiable on a dense $G_\delta$ subset of the closure of $\dom(\phi)$. 
\end{corollary}
\begin{proof}
We write $A \subset L_2^N(\H)$ for the set of Frechet differentiability points of $\overline{\phi}$. By \Cref{thm:LconvexGdeltaDiff}, the set $A = \bigcap_n O_n$ with $O_n$ adapted-law invariant, open and dense in the closure  of $\dom(\overline{\phi})$. By \Cref{rem:quotient_topology}, $U_n := \{ \lawad(X) : X \in O_n \}$ is open and dense in the closure of $\dom(\phi)$. As Frechet-differentiability implies that the subdifferential is a singleton, \Cref{lem:mc_diff_char_ad} yields that $\phi$ is $\mc$-differentiable on $\bigcap_n U_n$.  
\end{proof}

\begin{theorem}
Let $Q \in \Pc_2^N(\H)$. Then the set of $P \in \Pc_2^N(\H)$ such that $(P,Q)$ is a strict Monge pair is a dense $G_\delta$ set.    
\end{theorem}
\begin{proof}
This follows from \Cref{prop:AJ_adapted} and \Cref{cor:MCdfb_Gdelta}.     
\end{proof}

\section{Construction of transport-regular measures in $\Pc_2^N(\H)$}\label{sec:TR}

\subsection{Criterion for transport regularity}

The aim of this section is to provide a more accessible criterion for transport regularity of measures on iterated spaces of probability measures.
To this end, we leverage the adapted Lion's lift which permits us to formulate the following criterion:

\begin{theorem} \label{thm:char.TR}
    Let $P \in \Pc^N_2(\H)$ and $X \in L_2^N(\H)$ with $\lawad(X) = P$. 
    Assume that
    \begin{equation}
        \label{eq:thm.char.TR}
        \text{for almost every }u, \quad (\tilde u_t \mapsto X(u_{1:t-1},\tilde u_t,\cdot))_\# \lambda \in \Pc_{2}(L_2^{N-t}(H)) \text{ is transport regular},
    \end{equation}
    for all $t = 1,\dots,N$.
    Then, $P$ is transport regular.
\end{theorem}

An important role in the proof of \Cref{thm:char.TR} plays the following

\begin{proposition}\label{cor:LawInvSubDiffRandom.adapted}
Let $\phi: \Pc_2^N(\H) \to (-\infty, +\infty]$ be $\mc$-convex and let $P\in \Pc_2^N(\H)$. Then, for all $X \in L_2^N(\H)$ with $X\simad P$ we have 
\[
    \{ \lawad(Y) : Y\in \partial \overline\phi(X) \} \subseteq \partial_\mc \phi (P).
\]
Assume further that $P$ is transport regular, then we have
\[
    \{ \lawad(Y) : Y\in \partial \overline\phi(X) \} = \partial_\mc \phi (P).
\]
\end{proposition}

\begin{proof}
The first inclusion  follows immediately from \Cref{prop:subdiff_char_ad}.

For the second part, since $P$ is transport regular, we have by \Cref{prop:N-Monge.strict-Monge} that for every $Q\in\partial_\mc \phi(P)$, $\cplopt^N(P,Q)=\{\law(\ip(X),\xi(\ip(X)))\}$ for some measurable map $\xi$. This readily yields the desired equality.
\end{proof}

\begin{corollary} \label{cor:TR.crit}
    Let $P \in \Pc_2^N(\H)$ and $X \in L_2^N(\H)$ with $\lawad(X) = P$. Assume that
    \begin{enumerate}
        \item \label{it:cor.TR.crit.1} $(u_1 \mapsto X(u_1,\cdot))_\# \lambda$ is a transport regular measure in $\Pc_2(L_2^{N-1}(\H))$,
        \item \label{it:cor.TR.crit.2}for $\lambda$-almost every $u_1$, $\lawad(X(u_1,\cdot))$ is a transport regular measure in $\Pc_2^{N-1}(\H)$.
    \end{enumerate}
    Then, $P$ is transport regular.
\end{corollary}

\begin{proof} 
    By \Cref{lem:apx.char.TR}, we have to show that for every $\mc$-convex $\phi : \Pc_2^{N-1}(\H) \to (-\infty,+\infty]$, we have $\# \partial_\mc \phi(p) \le 1$ for $P$-almost every $p$.
    Using assumption (\ref{it:cor.TR.crit.1}), we get 
    \[
        \# \partial \overline{\phi}(X(u_1,\cdot)) \le 1 \quad \text{for $\lambda$-a.e.\ $u_1$.}
    \]
    Next, assumption (\ref{it:cor.TR.crit.2}) allows us to invoke \Cref{cor:LawInvSubDiffRandom.adapted} and we obtain
    \[
        \{ \lawad(Y) : Y \in \partial \overline{\phi}(X(u_1,\cdot)) \} = \partial_\mc \phi(\lawad(X(u_1,\cdot))) \quad \text{for $\lambda$-a.e.\ $u_1$.}
    \]
    Hence, $\# \partial_\mc \phi(\lawad(X(u_1,\cdot))) \le 1$ for $\lambda$-a.e.\ $u_1$ or equivalently $\# \partial_\mc \phi(p) \le 1$ for $P$-a.e.\ $p$.
    As $\phi$ is an arbitrary $\mc$-convex function, this shows transport regularity of $P$.
\end{proof}

\begin{proof}[Proof of \Cref{thm:char.TR}]
    We prove the claim by backward induction over $t$.
    More precisely, we show for $t = N,\dots,1$ that
    \[
        \text{for almost every $u$},\quad 
        \lawad\!\big(X(u_{1:t-1},\cdot)\big) \in \Pc_2^{N + 1 -t}(\H) 
        \text{ is transport regular.}
    \]
    \emph{Base case} ($t=N$):
    Using the assumption \eqref{eq:thm.char.TR} with $t = N$, we find that for almost every $u_{1:N-1}$ 
    \[
        \lawad(X(u_{1:N-1},\cdot)) 
        = (u_N \mapsto X(u_{1:N-1},u_N))_\# \lambda
        \in \Pc_2(\H)
        \text{ is transport regular}.
    \]
    \emph{Inductive step}:
    Assume the claim holds for $t+1$. 
    Using \eqref{eq:thm.char.TR} with $t$, we have
    \[ 
        \text{for a.e.\ }u, \quad
        P^{u_{1:t-1}} := (u_t \mapsto X(u_{1:t-1},u_t,\cdot) \big)  _\# \lambda 
        \in \mathcal P_2(L^{N-t}_2(\H))
        \text{ is transport regular}.
    \]
    Further, we derive from the induction hypothesis that, for a.e.\ $u$ and $P^{u_{1:t-1}}$-a.e.\ $Y$, $\lawad(Y) \in \Pc_2^{N-t}(\H)$ is transport regular.
    Thus, we can invoke \Cref{cor:TR.crit} and conclude that, for almost every $u$, $\lawad(X(u_{1:t-1},\cdot))$ is transport regular.

    Taking $t=1$, we yields that $\lawad(X)$ is a transport regular measure in $\Pc_2^N(\H)$.
\end{proof}

\subsection{Existence of transport-regular measures}

The criterion in \Cref{thm:char.TR} permits us to build transport-regular measures on $\Pc_2^N(H)$ using non-degenerate Gaussian random fields.
More precisely, we will consider the $H$-valued analogue to a Brownian sheet.

Recall that an $H$-valued $N$-parameter Brownian sheet $W$ with covariance operator $Q$ is an $H$-valued, continuous, centered Gaussian process indexed by $[0,1]^N$ such that
\begin{equation}
    \label{eq:}
    \mathbb E\big[ \langle W(u), x \rangle  \langle W(v), y \rangle \big] = \prod_{t = 1}^N \min(u_t,v_t) \, \langle Qx, y\rangle,
\end{equation}
for $u,v \in [0,1]^N$ and $x,y \in H$; see, for instance, \cite[Section~2]{RaSuZe07}.
Here, $Q : H \to H$ is a non-negative, symmetric linear operator with finite trace.

\begin{theorem} \label{thm:sheet.adapted.TR}
    Let $Q \in L(H)$ be a positive, symmetric linear operator with finite trace.
    Let $W$ be an $H$-valued $N$-parameter Brownian sheet. 
    Then $\law_\mathbb P(\lawad(W))$ is a transport-regular measure on $\Pc_2^{N}(\mathbb R^d)$ with full support.
\end{theorem}

\begin{proof}
    For $t = 1,\dots, N$, we write $H_t := L_2^{N-t}(H)$ using the convention that $L_2^0(H) = H$.
    For fixed $u_{1:t}, v_{1:t} \in [0,1]^t$ we have
        \begin{multline*}
        \mathbb E\big[ \langle W(u_{1:t},\cdot),f \rangle_{H_t} \langle W(v_{1:t},\cdot), g \rangle_{H_t} \big]
        \\
        =
        \int_{[0,1]^{N-t}\times[0,1]^{N-t}} \prod_{s = 1}^N \min(u_s,v_s) \langle Qf(u_{t+1:N}), g(v_{t+1:N}) \rangle \, d(u_{t+1:N},v_{t+1:N}).
    \end{multline*}
    Therefore, we can interpret $W$ as an $H_t$-valued, $t$-parameter Brownian sheet whose covariance operator $Q_t$ is given by
    \[
        (Q_t f)(u_{t+1:N}) = \int_{[0,1]^{N-t}} \prod_{s = t + 1}^N \min(u_s,v_s) \, Qf(v_{t+1:N}) \, dv_{t+1:N},
    \]
    where $u_{t+1:N} \in [0,1]^{N-t}$ and $f \in H_t$.
    Since $H_t$ can be identified with $L_2^{N-t}(\R) \otimes H$, $Q_t$ can be identified with the tensor product of the positive, symmetric trace class operators $\tilde Q$ and $Q$ on the separable Hilbert spaces $L_2^{N-t}(\R)$ and $H$, respectively.
    Therefore, $Q_t$ is a positive, symmetric trace class operator on $H_t$. 
    Indeed, $Q_t$ is symmetric, positivity follows from
    \begin{align*}
        \langle Q_t f, f \rangle_{H_t} =
        \int_{[0,1]^{N-t}} \Big\langle Q \int_{[w,1]} f(u_{t+1,N}) \, du_{t+1,N}, \int_{[w,1]} f(u_{t+1,N}) \, du_{t+1,N} \Big\rangle\, dw,
    \end{align*}
    where $[w,1] := [w_{N-t},1] \times \dots \times [w_N,1]$, for all $f \in H_t$, and it is trace class since
    \begin{align*}
        {\rm tr}(Q_t) &= \sum_{n,m  \in \mathbb N} \langle Q_t e_n \otimes f_m, e_n \otimes f_m\rangle_{H_t} \\
        &= \Big( \sum_{n \in \N} \int_{[0,1]^{N-t} \times [0,1]^{N-t}} \prod_{s = t+1}^N \min(u_s,v_s) e_n(u_{t+1:N})e_n(v_{t+1:N}) \, d(u_{t+1:N},v_{t+1:N}) \Big)
        \\ &\qquad \cdot \Big( \sum_{m \in \mathbb N} \langle Qf_m, f_m \rangle \Big)
        \\
        &= {\rm tr}(\tilde Q) \, {\rm tr}(Q) < \infty,
    \end{align*}
    for orthonormal bases $(e_n)_{n \in \mathbb N}$ of $L_2^{N-t}(\R)$ and $(f_m)_{m \in \mathbb N}$ of $H$.

    For fixed $u_{1:t-1} \in (0,1]^{t-1}$ we have that
    \[
        u_t \mapsto \mathbb E[\langle W(u_{1:t},\cdot), f \rangle_{H_t} \langle W(u_{1:t},\cdot), g \rangle_{H_t}] = \prod_{s = 1}^t u_s \langle Q_tf, g \rangle_{H_t},
    \]
    and thus $u_t \mapsto W(u_{1:t-1},u_t,\cdot)$ is a Wiener process with positive covariance operator.
    By \Cref{thm:Q-Wiener.TR}, we deduce that
    \begin{equation} \label{eq:TR.sheet.t}
        \text{a.s.}\quad 
        (u_t \mapsto W(u_{1:t},\cdot))_\# \lambda 
        \text{ is a transport-regular measure on } 
        H_t = L_2^{N-t}(H).
    \end{equation}
    By \Cref{thm:char.TR}, $\lawad(W)$ is almost surely transport regular in $\Pc_2^N(H)$. 
    Since $\law(W)$ is a non-degenerate Gaussian measure on $L_2^N(H)$ it is transport regular, see e.g.\ \cite[Theorem 6.2.10]{AGS}.
    Another application of \Cref{cor:LawInvSubDiffRandom.adapted} yields that $\law_\mathbb P(\lawad(W)) \in \Pc_2^{N+1}(H)$
    is transport regular, as claimed.

    Finally, we show that the support of $P$ is the whole space.
    Since $\law(W)$ as Gaussian measure on $L^N_2(H)$ is non-degenerate, we have that the support of $\law(W)$ is the whole space.
    As $\lawad : L^N_2(\H) \to \Pc_2^N(H)$ is continuous by \Cref{prop:iteratedSkorohod} and surjective by \Cref{prop:ReprLawadAsRV}, we conclude that the support of $\law_\mathbb P(\lawad(B))$ coincides with $\Pc_2^N(H)$.
\end{proof}

\subsection{Proofs of \Cref{sec:BrenierAndLiftsIntro}}

\begin{proof}[Proof of \Cref{Thm:MainMongeIntro}]
    \Cref{thm:sheet.adapted.TR} provides a transport-regular measure $\Lambda$ on $\Pc_2^N(H)$ with full support, and yields the claim.
\end{proof}

\begin{proof}[Proof of \Cref{cor:MainMongeIntro}]
    By \Cref{Thm:MainMongeIntro} there exists a transport-regular measure $\Lambda$ on $\Pc_2^N(H)$ with full support.
    Therefore, $\{ Q \in \Pc_2^N(H) : Q \ll P \}$ is dense and consists of transport-regular measures as consequence of \Cref{lem:apx.char.TR}.
\end{proof}

\begin{proof}[Proof of \Cref{thm:N-coupling_Brenier}]
    First, let $\phi$ be any $\mc$-convex function.
    Since $P \ll \Lambda$ and the latter is transport regular, we have by \Cref{lem:apx.char.TR} that $\# \partial_\mc \phi(p) \le 1$ for $P$-a.e.\ $p \in \Pc_2^{N}(H)$.
    Because of \Cref{prop:subdiff_char_ad} and \Cref{lem:mc_diff_char_ad} the adapted Lions lift satisfies $\# \partial \overline\phi(X) \le 1$ a.s.

    Next, recall that by \Cref{lem:W2_MC_connection_iterated} and \Cref{thm:ftot.mc}, $\Pi \in \cpl(P,Q)$ is $\W_2^2$-optimal if and only if there exists an $\mc$-convex $\phi$ with $\Pi(\partial_\mc \phi) = 1$.
    By the first part we find that, for $P$-a.e.\ $p$, $\partial_\mc \phi(p) = \{ \lawad(\nabla \overline\phi(Y)) \}$ with $Y \simad p$.
    Hence, $\W_2^2$-optimality of $\Pi$ is equivalent to
    \[
        \Pi \sim (\lawad(X), \lawad(\nabla \overline\phi(X)))
    \]
    for some $\mc$-convex $\phi : \Pc_2^N(H) \to (-\infty,+\infty]$ with $\partial_\mc \phi(p) \neq \emptyset$ for $P$-a.e.\ $p$.
    The claim then follows by \Cref{thm:MC_trafo_ast_ad}.
\end{proof}

\begin{proof}[Proof of \Cref{thm:only-coupling_Brenier}]
    The first claim follows from \Cref{thm:N-coupling_Brenier}.
    To see the second claim, note that $(P,Q) \in \Pc_2^N(H) \times \Pc_2^N(H)$ has a unique $\W_2$-optimal coupling and this coupling is concentrated on the graph of a bijection if and only if $(P,Q)$ is a strict Monge pair.
    For each fixed $Q \in \Pc_2^N(H)$, the set $\{ P \in  \Pc_2^N(H) : (P,Q) \text{ is strict Monge} \}$ is a dense $G_\delta$-set.
    It follows from the Kuratovski--Ulam theorem (see e.g.\ \cite[Theorem~8.41]{Ke95}) that the sets
    \begin{align*}
        \{ (P,Q) \in  \Pc_2^N(H) \times \Pc_2^N(H) : (P,Q) \text{ is strict Monge} \},\\
        \{ (P,Q) \in  \Pc_2^N(H) \times \Pc_2^N(H) : (Q,P) \text{ is strict Monge} \},
    \end{align*}
    are both comeager. Hence, the same is true for their intersection, which yields the claim.
\end{proof}

\section{Application to adapted transport}\label{sec:AW}
\subsection{Outline of the framework}

The aim of this section is to apply the results obtained above to $H$-valued stochastic processes, equipped with the adapted Wasserstein distance. 

As noted in the introduction, the set $\Pc_2(H^N)$ of laws of $N$-step processes is not complete wrt $\AW_2$.
It is shown in \cite{BaBePa21} that the completion of $(\Pc_2(\H^N),\AW_2)$ consists in the stochastic processes with \emph{filtration}. In detail, we use $\FFP_2$ to denote the class of all 5-tuples \[ \fp{X}= (\Omega^X, \F^X, \P^X, (\F^X_t)_{t=1}^N, (X_t)_{t=1}^N),\]  
where $X$ is an adapted square integrable process. In analogy to \eqref{eq:AW_intro_plain}, the (squared) adapted Wasserstein distance of two filtered processes is defined as $$\AW_2(\fp X, \fp Y):=\inf_{\pi\in \cplbc(\fp X, \fp Y)} \int |X-Y|^2 \, d\pi.$$
Here $\pi$ is a causal coupling of $\fp X, \fp Y$ if $\pi\in \cpl(\P^X, \P ^Y)$ and 
\[\pi(\Omega^X \times B|\F^X_t \otimes \F^Y_N)= \pi(\Omega^X \times B|\bar \F^X_N \otimes \F^Y_N), \quad B\in \F^Y, t \leq N.\] 
Bi-causality of $\pi$, that is 
 $\pi \in  
\cplbc(\fp X, \fp Y)$ means that this holds also when  the roles of $\fp X, \fp Y$ are exchanged.

We say that $\fp X, \fp Y $ are $\AW$-equivalent, in signs $\fp X\sim_\AW \fp Y$ if $\AW_2(\fp X,\fp Y)=0$ and write $\FP_2$ for $\FFP_2/ \sim_\AW$. Using these notations $(\FP_2,\AW_2)$ is the completion of $(\Pc_2(\H^N), \AW_2)$. Moreover, $(\FP_2, \AW_2)$ is a Polish geodesic space. We note that $\AW_2(\fp X,\fp Y)=0$ can be expressed in a number of equivalent ways: it is tantamount to $\fp X, \fp Y$ having the same probabilistic properties in the sense of \cite{HoKe84}, to  $\fp X, \fp Y$ having Markov lifts with the same laws (see \cite{BePfSc24}) and to $\fp X, \fp Y$ having the same adapted law in the sense defined in \eqref{eq:last_lawad} below.

Every filtered process $\fp X$ has a representative in 
\begin{align}
    L_{2,ad}^N(\H)  := \{ X \in L_2([0,1]^N,\lambda; \H^N ) : X_t \text{ is $\F_t$-measurable}  \}, .
\end{align}
where  $ (\F_t)$ denotes the coordinate filtration as above, see e.g.\ \cite{BePfSc24}. 
Moreover we have for
 $X,Y\in L_{2,ad}^N(\H)$  
\begin{align}
    \AW_2(X,Y)= \inf_{X \sim_\AW X', Y \sim_\AW Y'}  \| X_1'-X_2'\|_2 
\end{align}
and the infimum is attained. It is therefore without loss of generality to work with filtered  stochastic processes $X\in L_{2,ad}^N$ and we shall do so from now on.

The stochastic interpretation of the framework developed in the previous sections is that we consider an $\H$-valued random variable  
$X$ together with a filtration $(\F_t)_{t=1}^N$. We assume that $X$ is $\F_N$-measurable and interpret $(\F_t)_{t=1}^N$ as a model of how information about $X$ is gradually revealed. (Alternatively we could identify $X$ with its Doob martingale.)

In order to apply our results for $\Pc_2^N(H)$ to the adapted Wasserstein setting we need to consider the relationship  of  $L_2^N(H)$ and $L_{2,ad}^N(H)$ in some detail. 

\medskip
 
 The aim of this section is to translate these results to the classical setting of $H$-valued stochastic processes $X=(X_t)_{t=1}^N$ that are adapted to the filtration $(\F_t)_{t=1}^N$.

Following \cite{BaBePa21},  we  define the information process of an adapted process $X \in L_{2,ad}^N(\H)$ by backward induction on $t$ as
\begin{align*}
&\ipp_N(X) = X_N, \\
&\ipp_{t\,}(X) \,= (X_t, \law^{\F_t}(\ipp_{t+1}(X))), \qquad t = N-1,\dots,1.
\end{align*}
The information process takes values in the following nested space: Set $(\Zc_N, d_{\Zc_N}) := (\H, | \cdot |)$ and 
\[
\Zc_t := \H \times \Pc_2(\Zc_{t+1}) \qquad d_{\Zc_t}^2 := | \cdot |^2 + W_{2,\Zc_{t+1}}^2, 
\]
where $W_{2,\Zc_{t+1}}$ is the 2-Wasserstein distance w.r.t.\ the underlying metric $d_{\Zc_{t+1}}$. 
The adapted distribution of $X$ is then defined as 
\begin{align}\label{eq:last_lawad}
\lawad(X) := \law(\ipp_1(X)) \in \Pc_2(\Zc_1).
\end{align}

\begin{theorem}\label{thm:FPisometry}
The space of filtered processes is isometrically isomorphic to $\Pc_2(\Zc_1)$. That is
\[
\AW_2(X,Y) = \W_2(\lawad(X), \lawad(Y)).
\]
In particular $X\sim_\AW Y $ if and only  $X\simad Y $.
Moreover, for every $P \in \Pc_2(\Zc_1)$, there exists $X \in L_{2,ad}^N(H)$ such that $\lawadd(X)=P$. 
\end{theorem}
\begin{proof}
    The first assertion is due to \cite[Theorem 1.3]{BaBePa21}.
    The second claim follows analogously as in \Cref{prop:ReprLawadAsRV}.
\end{proof}

\subsection{Embedding into $\Pc_2^N(\H^N)$}\label{sec:embed}
In order to apply the results from the previous sections, we need to embed $\Pc_2(\Zc_1)$ into $\Pc^N_2(\H^N)$. We start to outline this procedure at the level of processes / random variables.  

Every process $X = (X_t)_{t=1}^N \in L_{2,ad}^N(\H)$ is (by forgetting the adaptedness constraint) in particular an element of $L_2^N(\H^N)$, i.e.\ there is a natural embedding 
\begin{align}
    \iota_{ad} : L_{2,ad}^N(\H) \to L_2^N(\H^N).
\end{align}
If it is notionally convenient, we suppress $\iota_{ad} $ and in particular consider $L_{2,ad}^N(\H)$ as a subspace of $L_2^N(\H^N)$.

On the other hand we can interpret $X \in L_2^N(  \H^N)$ as a process $X = (X_t)_{t=1}^N$ that is not necessarily adapted. We can naturally assign to it the adapted process $(\E[X_t|\F_t])_{t=1}^N$. This defines the orthogonal projection
\begin{align}
    \pr_{ad} : L_2^N(  \H^N) \to L_{2,ad}^N(\H) : X \mapsto (\E[X_t|\F_t])_{t=1}^N.
\end{align}

\medskip

Next, we outline this procedure at the level of adapted distributions.  
To that end, we introduce the auxiliary spaces 
\[
E_t := \Pc_2^{N-t}(\H^{N+1-t}) \qquad t=1,\dots,N.
\]
Note that $\Pc_2(E_1) = \Pc_2^N(\H^N)$. We recall the from \Cref{sec:AdLift} that the ip of a random variable is given by $\ip_t(X) = \law^{\F_{t:1}}(X)$ and further note that $\ip_t(X_{t:N})$ takes values in $E_t$. 

It is crucial that $\ipp_t(X)$ and $\ip_t(X_{t:N})$ contain precisely the same information. In particular, $\ipp_1(X)$ and $\ip_1(X)$ contain the same information and hence the difference of $\lawad$ and $\lawadd$ is only of notational nature. This is made rigorous in the next proposition:

\begin{proposition}\label{prop:translation}
For $t \in \{1,\dots,N\}$, there exist mappings 
\begin{align*}
    \iota_t &: \Zc_t \to E_t \\
    j_t &: E_t \to \Zc_t
\end{align*} 
with the following properties:
\begin{enumerate}
    \item $j_t$ is continuous and the left-inverse of $\iota_t$, i.e. $j_t \circ \iota_t = \id_{\Zc_t}$,
    \item $\iota_t$ is an isometric embedding,
    \item $\iota_t(\ipp_t(X)) = \ip_t(X_{t:N})$ for every $X \in L_{2,ad}^N$,
    \item $j_t(\ip(X_{t:N})) = \ipp_t(\pr_{ad}(X))$ for every $X \in L_2^N(H^N)$.
\end{enumerate}
Moreover, we have: 
\begin{enumerate}
\setcounter{enumi}{4}
    \item $j_{1\#}\lawad(X) = \lawadd(\pr_{ad}(X)) $ for every $X \in L_2^N(\H^N)$,
    \item $\iota_{1\#}\lawadd(X) = \lawad(X)$ for every $X \in L_{2,ad}^N(H)$.
\end{enumerate}

\end{proposition}
\begin{proof}
We start with the construction of the mappings. To this end, we write $\delta$ for the map 
\begin{align*}
    \delta : \mathcal{X} \to \Pc_2(\mathcal{X}) : x \mapsto \delta_x
\end{align*}
which assigns to a point $x$ the Dirac measure at $x$ and we also consider its $t$-fold iteration $\delta^t : \mathcal{X} \to \Pc_2^t(\mathcal{X})$.

 Next, we define the mappings  $ \iota_t :  \Zc_t \to B_t$ by backward induction on $t$. For $t=N$ we set 
$
\iota_N := \id : \H \to \H
$
and inductively, for $t <N$,
\begin{align}\label{eq:iota_t}
    \iota_t : \Zc_t \to E_t, \qquad \iota_t(x,p) := \delta^{N-t}_x \otimes {\iota_{t+1}}_\# p.
\end{align}

The mappings $j_t$ are then defined by $j_N := \id : H \to H$ and by induction for $t <N$ (writing $\pr_t, \dots, \pr_N$ for the projections from $H^{N+1-t}$ onto $H$) defined by
\begin{align}\label{eq:j_t}
  j_t : E_t  \to \Zc_t , \qquad  j_t(P) = ( \mean{\pr_{t\#}(I^{N-t-1}(P))}, j_{t+1\#}(\Pc^{N
  -t}[\pr_{t+1:N}](P))),
\end{align}
where $\mean{\mu} := \int x \,\mu(dx)$ for $\mu \in \Pc_2(H)$.

Next, we show the claims (1) to (4) simultaneously by induction on $t$. For $t=N$, they are trivial as $\iota_N$ and $j_N$ are both the identity. 

Suppose that the claim is true for $t+1$. We start with claim (1). As the intensity,  mean
and pushforwards with continuous maps, are all continuous operations, $j_t$ is continuous. For $(x,p) \in \Zc_t = H \times \Pc_2(\Zc_{t+1})$ we have using that $j_{t+1} \circ \iota_{t+1} = \id_{\Zc_{t+1}}$
\begin{align*}
    j_t(\iota_t(x,p)) = j_t(\delta_x^{N-t} \otimes \iota_{t+1}\#p) = (\mean{I^{N-t-1}( \delta_x^{N-t} )  } , j_{t+1\#}\iota_{t+1\#}p) = (x,p).
\end{align*}
Concerning (2), it is easy to check from \eqref{eq:iota_t} and from the definition of the Wasserstein distance that $\iota_t$ is an isometry. It is an embedding because $j_t$ is its continuous left-inverse.

In order to show (3) observe that
\[
\iota_t(\ipp_t(X)) = \iota_t(X_t, \law^{\F_t}(\ipp_{t+1}(X))) = \delta^{N-t}_{X_t} \otimes \law^{\F_t}( \iota_{t+1}( \ipp_{t+1}(X)  ) )  
\]
and using the inductive claim this is further equal to 
\[
\delta^{N-t}_{X_t} \otimes \law^{\F_t} ( \ip_{t+1}(X) ) = \delta^{N-t}_{X_t} \otimes \law^{\F_{N-1:t}}(X_{t+1:N}) = \law^{\F_{N-1:t}}(X_{t:N}),
\]
where the last equality is true because $X_t$ is $\F_t$-measurable. 

Claim (4) follows from (3) and the fact that $j_t$ is the left inverse of $\iota_t$. This finishes the induction.

The claims (5) and (6) follow from (3) and (4) for $t=1$.
\end{proof}

\begin{corollary}\label{cor:ad_is_lawadinv}
We have $\Pc_2(\Zc_1) = \{ \lawadd(X) : X \in L_{2,ad}^N(H) \}$. 
In particular, $L_{2,ad}^N(H)$ is an adapted-law invariant subspace and the projection $\pr_{ad}$ is adapted-law invariant, i.e.\ if $X \simad X'$, then $\pr_{ad}(X) \simad \pr_{ad}(X')$.
\end{corollary}
\begin{proof}
The first assertion follows from \Cref{prop:translation} and the fact that the space of filtered processes and $\Pc_2(\Zc_1)$ are isomorphic (see \Cref{thm:FPisometry}). The second assertion follows from \Cref{prop:translation} (5).
\end{proof}
We note that it is more generally true that the projection onto an adapted-law invariant subspace is adadpted-law invariant. 

\begin{corollary}\label{cor:iptranslation2}
For $t \le N$, there exists a measurable map $g_t : \Zc_{1:t} \to \Pc_2^{N-t}(H^N) $ such that $\ip_t(X) = g_t(\ipp_{1:t}(X))$ for every $X \in L_{2,ad}^N(H)$.
\end{corollary}
\begin{proof}
 For every $s \le t$, $X_s$ is the first component of $\ipp_s(X)$. 
 Hence, there is a projection $\pr : \Zc_{1:t-1} \to \H^{t-1}$ such that $\pr(\ipp_{1:t-1}(X))=X_{1:t-1}$ for every $X \in L_{2,ad}^N(H)$. The map $\iota_t : \Zc_t \to E_t$ defined in \Cref{prop:translation} satisfies $\iota_t(\ipp_t(X))= \ip_t(X_{t:N})$ for every $X \in L_{2,ad}^N(H)$. Thus,
 \[
 g_t : \Zc_{1:t} \to \Pc_2^{N-t}(H^N), \quad g(z_{1:t}) := \delta^{N-t}_{\pr(z_{1:t-1})} \otimes \iota_t(z_t)
 \]
 satisfies $\ip_t(X) = g_t(\ipp_{1:t}(X))$ for every $X \in L_{2,ad}^N(H)$.
\end{proof}

It was central in the previous sections (in particular in the context of $\mc$-subdifferential and $N$-Monge couplings) to establish that a random variable $Y$ is of the form $Y = \xi(\ip(X))$. The next assertion shows that in the present setting the mapping $\xi$ translates to an adapted mapping. 

\begin{proposition}\label{prop:xi_is_adapted}
Let $X,Y \in L_{2,ad}^N(\H)$ and let $\tilde\xi : A_{1:N} \to H^N $ be measurable such that $Y = \tilde\xi(\ip(X))$. 
Then there is an adapted map $\xi : \Zc \to H^N$ such that $Y = \xi(\ipp(X))$, i.e., $Y_t = \xi_t(\ipp_t(X))$ for every $t \le N$.
Moreover, there is an adapted map $S : \Zc \to \Zc$ such that $\ipp(Y) = S(\ipp(X))$. 
\end{proposition}

\begin{proof}
Fix $t \le N$. By \Cref{prop:N-Monge.ip}, the map $\tilde\xi$ induces a map 
$
T_t : \prod_{s=1}^t   \Pc_2^{N-s}(H^N)  \to \Pc_2^{N-t}(H^N)  
$
such that $\ip_t(Y) = T_t(\ip_{1:t}(X))$. The map $g_t$ defined in \Cref{cor:iptranslation2} satisfies $\ip_t(X) = g_t(\ipp_{1:t}(X))$. We set $h_t(P) := \mean{ I^{N-t-1}(\Pc^{N-t}[\pr_t](P)}$ where $\pr_t : H^N \to H$ is the projection onto the $t$-th component. Note that 
$$
h_t(\ip_t(Y)) = \mean{ I^{N-t-1}(\law^{\F_{N-1:t}} (Y_t) } = \mean{ I^{N-t-1}(\delta^{N-t}_{Y_t})} = Y_t.
$$
We then define the desired map $\xi_t$ as 
$
\xi_t := h_t \circ T_t \circ g_t
$
and note that
\[
\xi_t(\ipp_t(X)) = h_t ( T_t (g_t( \ipp_{1:t}(X)))) = h_t(T_t(\ip_{1:t}(X))) = h_t(\ip_t(Y)) =  Y_t. 
\]
For the second claim, we set $S_t = j_t \circ T_t \circ g_t$ and note that 
\[
    S_t(\ipp_{1:t}(X)) = j_t(T_t(g_t(\ipp_{1:t}(X)))) = j_t(T_t(\ip_{1:t}(X)) = j_t(\ip_t(Y)) = \ipp_t(Y). \qedhere
\]
\end{proof}

\subsection{Lions lift for filtered processes}
\label{sec:LionsLiftFP}
From now on, we occasionally write $\mathcal{Z}_t^{(N)}$ instead of $\mathcal{Z}_t$ to make the dependence on the number of time steps explicit when needed. As the space of filtered processes is isometrically isomorphic to $\Pc_2(\Zc_1^{(N)})$, defining a lift for functions $\psi : \Pc_2(\Zc_1) \to (-\infty,+\infty]$ is equivalent to defining a lift for functions on the space of filtered processes with $N$ time periods.

\begin{definition}
Let $\psi : \Pc_2(\Zc_1^{(N)}) \to (-\infty,+\infty]$. Then its adapted Lions lift is defined as the function 
\[
\overline{\psi} : L_{2,ad}^N(\H) \to (-\infty,+\infty], \quad \overline{\psi}(X) = \psi(\lawadd(X)).
\]  
\end{definition}

In order to establish a connection to convex analysis and Brenier-type results for the adapted Wasserstein distance, we consider the adapted maximal covariance 
\begin{align*}
\AMC(X,Y) := \sup_{ X' \simad X, Y' \simad Y  }  \E[ X'  \cdot Y' ]. 
\end{align*}
This adapted max-covariance functional is connected to the iterated max-covariance on $\Pc_2(\Zc^{(N)}_1)$. To define this functional we first set
\begin{align*}
    \langle x,y \rangle_{\Zc_N} &:= \langle x,y \rangle_H \\
    \langle (x,p), (y,q) \rangle_{\Zc_t} &:= \langle x,y \rangle_H + \langle p, q \rangle_{\Zc_{t+1}}
\end{align*}
and then define the max-covariance functional on $\Pc_2(\Zc_1)$ as
\[
\mc(P,Q) = \sup_{\Pi \in \cpl(P,Q)} \int \langle (x,p), (y,q) \rangle_{\Zc_1} \, d\Pi(x,p,y,q).
\]
\begin{proposition}
For $X,Y \in L_{2,ad}^N(\H)$ we have 
\begin{enumerate}
    \item $\AW_2^2(X,Y) = \|X\|_2^2 + \|Y\|_2^2 - 2\AMC(X,Y) $
    \item $\W_2^2(\lawadd(X),\lawadd(Y)) = \E|X|^2 + \E|Y|^2 - 2\mc(\lawadd(X),\lawadd(Y)) $
    \item $\AMC(X,Y) = \mc(\lawadd(X),\lawadd(Y))$
     
\end{enumerate}
\end{proposition}

\begin{proof}
The claims (1) and (2) are straightforward by completing squares. (3) follows then from \Cref{thm:FPisometry} by invoking (1) and (2).
\end{proof}

Transport theory on  $\Pc_2(\Zc_1)$ with $\mc$ costs then yields the notions of $\mc$-transform and $\mc$-subgradient in the present setting:
\begin{definition}\label{def:MC_subdiff_FP}
    Let $\psi : \Pc_2(\Zc_1) \to (-\infty,+\infty]$ be proper. We define its $\mc$-transform as 
    \[
        \psi^{\mc}(Q) = \sup_{ P \in \Pc_2(\Zc_1)}  \mc(P,Q) -  \psi(P). 
    \]
    A function $\psi : \Pc_2(\Zc_1) \to (-\infty,+\infty]$ is called $\mc$-convex if there exists a proper function $\vartheta : \Pc_2(\Zc_1) \to (-\infty,+\infty]$ such that $\psi = \vartheta^\mc$.
    Moreover, the $\mc$-subdifferential of $\psi$ at $P \in \Pc_2(\Zc)$ is defined as  
    \[
        \partial_{\mc} \psi(P) = \{ Q\in \Pc_2(\Zc_1) : \psi(R) \ge \psi(P) + \mc(R,Q) - \mc(P,Q) \text{ for every } R \in \Pc_2(\Zc_1)\},
    \]
    and we write
    \[
        \partial_\mc \phi := \{ (P,Q) : Q \in  \partial_{\mc} \phi(P), P \in \Pc_2(\Zc_1) \}. 
    \]
\end{definition}

\medskip

Analogous to the theory on $\Pc_2^N(H)$, we have the following crucial connection between the $\mc$-transform and the convex conjugate on $L_{2,ad}^N(\H)$. 

\begin{lemma}\label{lem:FP_XY}
Let $P,Q \in \Pc_2(\Zc_1)$ and $Y \in L_{2,ad}^N(\H)$ with $\lawadd(Y)=Q$ be given. Then,
\[
\mc(P,Q) = \sup_{ \substack{ X \in L_{2,ad}^N(\H) \\
\lawadd(X)=P }} \E[X \cdot Y]. 
\]
\end{lemma}
\begin{proof}
By using that $\iota_1$ is an isometry (see \Cref{prop:translation}), the corresponding claim in the setup of $\Pc_2^N(\H^N)$ (see \Cref{lem:DPP}) and the self-adjointness of $\pr_{ad}$ we find
\begin{align*}
    \mc(P,Q)=\mc(\iota_{1\#}P,\iota_{1\#}Q) = \sup_{ \substack{ X \in L_{2}^N(\H^N) \\\lawad(X)=P }} \E[X \cdot Y] = \sup_{ \substack{ X \in L_{2,ad}^N(\H) \\\lawadd(X)=P }} \E[X \cdot Y]. 
\end{align*}    
\end{proof}

\begin{proposition}\label{prop:MCcvx_FP}
Let $\psi : \Pc_2(\Zc_1) \to (-\infty,+\infty]$ be proper. Then we have 
\begin{align}\label{eq:mc_lift}
\overline{\psi^\mc} = \overline{\psi}^\ast.
\end{align}
In particular, the convex conjugate of an adapted-law invariant function on $L_{2,ad}^N(\H)$ is adapted-law invariant.

Moreover, the following are equivalent:
\begin{enumerate}
    \item $\overline{\psi}$ is lsc convex,
    \item $\psi$ is $\mc$-convex.
\end{enumerate}   
\end{proposition}
\begin{proof}
For $Y \in L_{2,ad}^N(H)$ we write $Q =\lawadd(Y)$. By \Cref{lem:FP_XY} we find
\begin{align*}
   \overline{\psi\,}^\ast(Y) &= \sup_{X \in L_{2,ad}^N(\H) } \E[ X \cdot Y ] - \overline{\psi}(X) = \sup_{ P \in \Pc_2(\Zc_1)}  \sup_{X \simad P} \E[X \cdot Y]   -  \psi(P) \\
   &= \sup_{P \in \Pc_2(\Zc_1) } \mc(P,Q) - \psi(P)  = \psi^{\mc}(Q) = \overline{\psi^{\mc}}(Y).
\end{align*}
Next, suppose that $\overline{\psi}$ is lsc convex. Then by first applying the Fenchel--Moreau theorem and then \eqref{eq:mc_lift} twice we find
\[
\overline{\psi} = \overline{\psi}^{\ast\ast} = \overline{\psi^\mc}^\ast = \overline{\psi^{\mc\mc}}.
\]
Hence, $\psi = \psi^{\mc\mc}$, which shows that $\psi$ is $\mc$-convex. 

Conversely assume that $\psi$ is $\mc$-convex. Then $\psi = \psi^{\mc}$ for some proper function $\vartheta : \Pc_2(\Zc_1) \to (-\infty,+\infty]$. By \eqref{eq:mc_lift}, we have $\overline{\psi} = \overline{\vartheta^\mc} = \overline{\vartheta}^\ast$, which shows that $\overline{\psi}$ is lsc convex. 
\end{proof}

\begin{remark}\label{rem:different_exts}
Next, we discuss how to extend adapted-law invariant functions on $L_{2,ad}^N(\H)$ to adapted-law  invariant functions on $L_{2}^N(\H^N)$. This  allows us to also translate results on the structure of the subdifferentials of $\mc$-convex functions from \Cref{sec:MC_cx_ad} to the present setting. The are two ways to extend a given adapted-law invariant function $\overline{\psi} : L_{2,ad}^N(H) \to (-\infty,+\infty]$.
\begin{enumerate}
    \item We set it $+\infty$ on non-adapted processes, i.e. writing $\chi$ for convex indicators
    \[
    \overline{\phi}(X) = \overline{\psi}(X) + \chi_{L_{2,ad}^N(H)}(X).
    \]
    \item We use the projection onto adapted processes, i.e.
    \[
\overline{\phi}(X) := \overline{\psi}(\pr_{ad}(X)).
\]
\end{enumerate}
As $L_{2,ad}^N(H)$ and $\pr_{ad}$ are adapted-law invariant (see \Cref{cor:ad_is_lawadinv}), this defines indeed adapted-law invariant functions. Moreover, these two ways are dual to each other via the convex conjugate on $L_2^N(H^N)$, i.e.
\[
(\overline{\psi} \circ \pr_{ad})^\ast = \overline{\psi}^\ast + \chi_{L_{2,ad}^N(H)}.
\]
\end{remark}

\begin{remark}\label{rem:dfb_equiv_extend}
If we are in the situation that $\overline{\phi}(X) := \overline{\psi}(\pr_{ad}(X))$, we have (when suppressing the embedding $\iota_{ad}$ and considering $L_{2,ad}^N(H)$ as subspace of $L_2^N(H^N)$)
\[
\partial \overline{\phi}(X) = \partial\overline{\psi}(X).
\]
Here, the subdifferential on the left hand side in meant in $L_{2}^N(H^N)$, whereas the one the right hand side is meant in the subspace $L_{2,ad}^N(H)$. In particular, this asserts that $\partial \overline{\phi}(X) \subset L_{2,ad}^N(H)$. Moreover, for $X \in L_{2,ad}^N(H)$, $\overline\phi$ is Gateaux (Frechet) differentiable at $X$ if and only if $\overline\psi$ is Gateaux (Frechet) differentiable at $X$. 
\end{remark}

\begin{proposition}\label{prop:subdiff_char_FP}
Let $\psi : \Pc_2(\Zc_1) \to (-\infty, +\infty]$ be proper and let $X,Y \in L_{2,ad}^N(\H)$. If $X \simad P$ and $Z \simad Q$ we have
\[
Z \in \partial \overline{\psi}(X) \iff Q \in \partial_{\mc} \psi(P) \text{ and } \E[X \cdot Z] = \mc(P,Q). 
\]
\end{proposition}
\begin{proof}
The proof is line by line as the proof of \Cref{prop:Subdiff_MCvsL}.    
\end{proof}

\begin{lemma}\label{lem:mc_diff_char_FP}
    Let $\psi:\Pc_2(\Zc_1)\to (-\infty, \infty]$ be $\MC$-convex and $P\in \dom(\psi)$. Then the following are equivalent: 
    \begin{enumerate}
        \item $\psi$ is $\MC$-differentiable at $P$.
        \item For all $X$ with $X\simad P$ we have $\#\partial \overline \psi (X)=1$.
        
    \end{enumerate}
\end{lemma}
\begin{proof}
This follows line by line as in the proof of \Cref{lem:mc_diff_char_ad}.
\end{proof}

\begin{corollary}\label{cor:dfb_ext_equiv}
Let $\psi:\Pc_2(\Zc_1)\to (-\infty, \infty]$ be $\MC$-convex and $P\in \dom(\phi)$. Set $\phi := \psi \circ \Pc(\iota_1)$. Then the $\mc$-subdifferentials coincide in the sense that 
\[
\partial_\mc\phi(\iota_{1\#}P) = \{ \iota_{1\#} Q : Q \in \partial_\mc\psi(P) \}.
\]
In particular, $\psi$ is $\mc$-differentiable at $P$ if and only if $\phi$ is $\mc$-differentiable at $\iota_{1\#}P$.  
\end{corollary}
\begin{proof}
Let $R \in \partial_\mc\phi(\iota_{1\#}P) $.  Let $X \simad \iota_{1\#}P$ (and hence $X \in L_{2,ad}^N(H)$) and $Y \simad R$ such that $\E [ X \cdot Y] = \mc(\iota_{1\#}P, R)$. By \Cref{prop:subdiff_char_ad}, we have  $Y \in \partial\overline{\phi}(X)$. By \Cref{rem:dfb_equiv_extend}, we have $Y \in \partial \overline{\psi}(X)$. In particular, $Y \in L_{2,ad}^N(H)$ and hence $R = \lawad(Y)$ is of the formn $\iota_{1\#}Q$ for some $Q \in \Pc_2(\Zc_1)$. Then \Cref{prop:subdiff_char_FP} yields $Q \in \partial_\mc\psi(P)$. 

Conversely, let $Q \in \partial_\mc\psi(P)$. Then there are $X,Y \in L_{2,ad}^N(H)$ such that $\mc(P,Q)=\E[X \cdot Y]$. By \Cref{prop:subdiff_char_FP}, we have $Y \in \partial\overline{\psi}(X)$, by \Cref{rem:dfb_equiv_extend} we have $Y \in \partial\overline{\phi}(X)$. Then, $\iota_{1\#} Q \in \partial_\mc\phi(\iota_{1\#}P)$ by \Cref{prop:subdiff_char_ad}. 
\end{proof}

\begin{proposition}\label{prop:LdiffXi_FP}
Let $\psi:\Pc_2(\Zc_1)\to (-\infty, \infty]$ be $\MC$-differentiable at $P$. Then there exists  an adapted map  $\xi : \Zc \to \H^N$ such that for all $X\simad P$ 
\[
\partial\overline{\psi}(X) = \{\xi( \ipp(X) )\}.
\]
\end{proposition}
\begin{proof}
We set $\phi := \psi \circ \Pc(\iota_1)$ and note that  $\overline{\phi}(X) = \overline{\psi}(\pr_{ad}(X))$ and that $\phi$ is $\mc$-differentiable at $\iota_1\#P$ by \Cref{cor:dfb_ext_equiv}. By \Cref{prop:LdiffXi_ad} there exists a measurable function  $\tilde \xi : A_{1:N} \to \H^N$ such that $ \partial\overline{\phi}(X) = \{\tilde \xi( \ip(X) )\} $ for all $X\simad P$.  By \Cref{prop:xi_is_adapted}, there is an adapted map $\xi: \Zc \to H^N$ such that $\tilde\xi(\ip(X)) = \xi(\ipp(X))$. Using \Cref{rem:dfb_equiv_extend} we find
\[
\partial\overline{\psi}(X) = \partial\overline{\phi}(X) = \{\xi( \ipp(X) )\}. \qedhere
\]
\end{proof}

\begin{proposition}\label{prop:AJ_FP}
    For $P,Q \in \Pc_2(\Zc_1)$ the following are equivalent:
    \begin{enumerate}
        \item $\mc(\cdot,Q)$ is $\mc$-differentiable at $P$.
        \item There is an adapted map $\xi : \Zc \to \H^N$ such that for some $X \simad P$
        \[
            \big\{ Y \simad Q : \E[X \cdot Y] = \MC(P,Q) \big\} = \{ \xi(\ipp(X)) \}.
        \]
        \item For some $X \simad P$, the function $\overline{\mc(\cdot,Q)}$ is Frechet differentiable at $X$.
    \end{enumerate}
    Moreover, whenever one (and hence all) of the above conditions holds, statements (2) and (3) are valid for every $X \simad P$.
\end{proposition}
\begin{proof}
We derive this result from \Cref{prop:AJ_adapted}, that is the corresponding result on $\Pc^N_2(H^N)$. 
For this note that the extension of $\mc(\cdot ,Q)$ 
to a function on $\Pc_2^N(H^N)$ in the way described in \Cref{rem:different_exts}(2) is precisely $\mc(\cdot, \iota_{1_\#}Q)$ (here $\mc$ denotes the max-covariance on $\Pc_2^N(H^N)$). 

Assertion (1)  in \Cref{prop:AJ_adapted} and assertion (1) in \Cref{prop:AJ_FP} are equivalent by \Cref{cor:dfb_ext_equiv}. Moreover, the assertion (2)  in \Cref{prop:AJ_adapted} and assertion (2) in \Cref{prop:AJ_FP} are equivalent due to \Cref{prop:xi_is_adapted}. Finally, assertion (3) in \Cref{prop:AJ_adapted} and assertion (3) in \Cref{prop:AJ_FP} are equivalent by \Cref{rem:dfb_equiv_extend}.
\end{proof}

\begin{proposition}\label{prop:MCdfb_Gdelta_FP}
Let $\psi : \Pc_2(\Zc_1) \to (-\infty,+\infty]$ and suppose that $\cont(\psi)\neq\emptyset$. Then $\psi$ is $\mc$-differentiable on a dense $G_\delta$ subset of the closure of $\dom(\psi)$. 
\end{proposition}
\begin{proof}
    The proof follows the proof of \Cref{cor:MCdfb_Gdelta} line by line.
\end{proof}

\subsection{Naturally filtered Processes}

A process is called natural if its filtration contains no more information about its future evolution than is provided by its past trajectory.
This is made precise in the following definition.
\begin{definition}
    $X$ is called naturally filtered if for every $t \le N$
    \[
    \law(X \, |  \, \F_t) = \law(X  \, | \, X_{1:t}).
    \]
\end{definition}
Naturally filtered processes are already determined by the law of the processes itself and it is not necessary to consider the adapted law in this case. This is made precise in the next lemma.

\begin{lemma} \label{lem:plain_char}
    Let $X, Y \in L_{2,ad}^N(H)$. Then, we have:
    \begin{enumerate}
        \item $X$ is naturally filtered if and only if there is an adapted map $f$ such that $\ipp(X)=f(X)$.
        \item If $X,Y$ are naturally filtered, then $\lawad(X) = \lawad(Y)$ if and only if $\law(X)=\law(Y)$.
    \end{enumerate}
\end{lemma}
\begin{proof}
See e.g.\ \cite[Section 3]{BePfSc24} and \cite[Lemma 2.3]{Pa22}.
\end{proof}

As the law of the random variables determines a natural filtered processes, it is also meaningful to consider the adapted Wasserstein distance between these laws,  defined in \eqref{eq:AW_intro_plain} in the introduction. The following lemma clarifies that these concepts are consistent.

\begin{lemma}\label{lem:plainAW_bc}
We have 
    \begin{enumerate}
        \item If $X, Y \in L_{2,ad}^N(\H)$ are naturally filtered, $\law(X,Y) \in \cplbc(\law(X),\law(Y))$.
        \item If $\mu,\nu \in \Pc_2(H^N)$ and $\pi \in \cpl(\mu,\nu)$, there are naturally filtered $X,Y \in L_{2,ad}^N(\H)$ such that $\law(X,Y)=\pi$.
        \item If $X, Y \in L_{2,ad}^N(\H)$ are naturally filtered, $\AW_2(X,Y) = \AW_2(\law(X),\law(Y))$.
    \end{enumerate}
\end{lemma}
\begin{proof}
For (1) and (2), see e.g.\ \cite[Lemma 5.3]{BePfSc24}. For (3) see \cite[Theorem 1.2]{BaBePa21}.
\end{proof}
Next, we give an embedding $J$ of the (adapted distributions of) natural filtered processes $(\Pc_2^N(\H^N),\AW_2)$ into set of (adapted distributions of) general processes. For $N=2$ periods, writing $d\mu(x_1,x_2)=d\mu_1(x_1)\mu_{x_1}(x_2)$, it reads as follows:
\[
 J(\mu) =  (x_1 \mapsto (x_1, \mu_{x_1}))_\# \mu_1.
\]
For $N>2$ it is an iteration of such maps (see \cite{BaBaBeEd19b, BaBeEdPi17}) for notational simplicity we define the map $J$ in processes language as 
\begin{align}\label{eq:J}
J : \Pc_2(\H^N) \to \Pc_2(\Zc_1) : \mu \mapsto \lawadd(X) \text{ where $X$ naturally filtered with $\law(X)=\mu$}.
\end{align}

\begin{proposition}\label{prop:plainGdelta}
The map $J : (\Pc_2(\H^N),\AW_2) \to (\Pc_2(\Zc_1),\W_2)$ is an isometric embedding and its range is dense $G_\delta$. 
\end{proposition}
\begin{proof}
See e.g.\ \cite[Section 5]{BaBeEdPi17}.     
\end{proof}

\begin{corollary}
\label{cor:xitoT_natural}
Let $X, Y \in L_{2,ad}^N(\H)$. If $X$ is naturally filtered and  $Y = \xi(\ipp(X))$, then there is an adapted map $ \mathsf{T} : \H^N \to \H^N$ such that $Y =\mathsf{T}(X)$. 
\end{corollary}
\begin{proof}
This follows immediately from \Cref{prop:xi_is_adapted} and \Cref{lem:plain_char} (1).     
\end{proof}

\begin{proposition}\label{prop:biMongeGd_proc}
 For every $\nu \in  \Pc_2(\H^N)$, the set of $\mu \in  \Pc_2(\H^N) $  such that $\cplopt^{bc}(\mu,\nu) = \{\pi\}$ and $\pi$ is induced by an adapted map is comeager in $\Pc_2(\H^N)$.  
\end{proposition}
\begin{proof}
Fix $\nu \in  \Pc_2(\H^N)$ and write $Q = J(\nu) \in \Pc_2(\Zc_1)$, where $J$ is the map defined in \eqref{eq:J}. By \Cref{prop:MCdfb_Gdelta_FP}, the set of $P \in \Pc_2(\Zc_1) $ such that $\MC(\cdot, Q)$ is $\mc$-differentiable in $P$ is dense $G_\delta$ in $\Pc_2(\Zc_1) $. By \Cref{prop:plainGdelta}, also the set 
\[
D := \{ \mu \in \Pc_2(\H^N) : \MC(\cdot, Q) \text{ is $\mc$-differentiable in $J(\mu)$ } \}
\]
is dense $G_\delta$ in $\Pc_2(\H^N)$. 

Let $\mu \in D$. By \Cref{prop:AJ_FP} there is an adapted map $\xi : \Zc \to \H^N$ such that such that for every $X \simad J(\mu)$ 
\begin{align}\label{eq:prf:prop:biMongeGd_proc_1}
            \big\{ Y \simad Q : \E[X \cdot Y] = \MC(P,Q) \big\} = \{ \xi(\ipp(X)) \}.
\end{align}
 By \Cref{lem:plain_char}, there is an adapted map $f : \H^N \to \Zc$ such that $\ipp(X)=f(X)$. We define 
 \[
 \mathsf{T} := \xi \circ f : \H^N \to \H^N.
 \]
Note that  $\mathsf{T}$ is adapted as concatenation of adapted maps. Now, let $\pi \in \cplbc(\mu,\nu)$ be any optimizer. We observe that $\pi=(\id,\mathsf{T})_\#\mu$. Indeed, by \Cref{lem:plainAW_bc}, there are $X,Y \in L_{2,ad}^N(H)$ natural such that $\law(X,Y)=\pi$. By \eqref{eq:prf:prop:biMongeGd_proc_1}, we have $Y = \xi(\ipp(X))= \xi(f(X))= \mathsf{T}(X)$. Hence, the optimal $\pi$ is unique and Monge. 
\end{proof}

\begin{proposition}\label{prop:biMongeGd_proc}
 The set of pairs $(\mu,\nu) \in \Pc_2(\H^N) \times \Pc_2(\H^N) $  such that  $\cplopt^{bc}(\mu,\nu) = \{\pi\}$ and $\pi$ is induced by a bi-adapted map is  comeager in $\Pc_2(\H^N) \times \Pc_2(\H^N)$.
\end{proposition}
\begin{proof}
 We write $e : \H^N \times \H^N : (x,y) \mapsto (y,x)$ and 
\begin{align*}
    A &= \big\{  (\mu,\nu) : \cplopt^{bc}(\mu,\nu) = \{\pi\}  \text{ and $\pi$ is Monge}  \big\}, \\
    B &= \big\{  (\mu,\nu) : \cplopt^{bc}(\mu,\nu) = \{\pi\}  \text{ and $e_\#\pi$ is Monge}  \big\}.
\end{align*}
 By \Cref{prop:biMongeGd_proc}, for every $\nu \in \Pc_2(\H^N)$, the set $A^\nu := \{ \mu \in \Pc_2(\H^N) : (\mu,\nu) \in A\}$ is comeager.  For the same reason, for every $\mu \in \Pc_2(\H^N)$, the set $B_\mu = \{  \nu \in \Pc_2(\H^N) : (\mu,\nu) \in B \}$ is  comeager. It follows from  the Kuratovski--Ulam theorem (see e.g.\ \cite[Theorem~8.41]{Ke95}) that $A$ and $B$ are comeager in $\Pc_2(\H^N) \times \Pc_2(\H^N)$. Hence, $A \cap B$ is comeager  in $\Pc_2(\H^N) \times \Pc_2(\H^N)$. If $(\mu,\nu) \in A \cap B$, then $\cplopt^{bc}(\mu,\nu) = \{\pi\}$ and both $\pi$ and $e_\#\pi$ are induced by an adapted map. Hence, $\pi$ is induced by a bi-adapted map (see e.g.\ \cite[Lemma A.2]{BePaSc21c}).
\end{proof}

\subsection{Lift of functions on $\Zc_1$ and proof of \Cref{thm:AW_Brenier_Intro}}
For transport of measures $P,Q \in \Pc_2(\Zc_1^{(N)})$, the dual potentials are functions on $\Zc^{(N)}_1 = \H \times \Pc_2(\Zc^{(N)}_2) = \H \times \Pc_2( \Zc^{(N-1)}_1 ) $. For this reason, we also need to consider adapted Lions lifts of such functions, that is 
\begin{align}\label{eq:LiftZ1}
\overline{\psi} : \H \times L_{2,ad}^{N-1}(\H) \to (-\infty,+\infty], \quad \overline{\psi}(x_1,X_{2:N}) = \psi(x_1,\lawadd(X_{2:N})).
\end{align} 
We write \[
V_N:= \H \times L_{2,ad}^{N-1}(\H)
\]
and note that if $X \in L_{2,ad}^N(\H)$ and $u_1 \in [0,1]$ is fixed, then  
$$
X(u_1,\cdots) = (X_1(u_1),X_2(u_1,\cdot),\dots,X_N(u_1,\cdots)) \in V_N.
$$ 
Hence, we can consider $X \in L_{2,ad}^N(\H)$ also as $V_N$-valued random variable on $([0,1],\lambda)$ and apply a function $\overline{\psi} : \H \times L_{2,ad}^{N-1}(\H) \to (-\infty, +\infty]$ to it, in this case we write $\overline{\psi}(X(u_1,\cdot))$.

It is easy to see from the results in \Cref{sec:LionsLiftFP} that via this Lions lift defined in \eqref{eq:LiftZ1} the $c$-transform for the cost $c((x,p),(y,q)) = \langle(x,p),(y,q) \rangle_{\Zc_1}$ coincides with the convex conjugate on $V_N$. Writing $\langle\cdot,\cdot\rangle_{V_N}$ for the scalar product on $V_N$ we also have the analogue of \Cref{prop:subdiff_char_FP}, namely
\begin{align}\label{eq:subdiff_char_FP_fancy}
    (y_1,Y_{2:N}) \in \partial \overline{\phi}((x_1,X_{2:N})) \iff \begin{cases}
       (Q,y_1) \in \partial_{c} \phi((P,x_1)) \text{ and } \\ 
       \langle (x_1,X_{2:N}), (y_1,Y_{2:N}) \rangle_{V_N} = c((x_1,P),(y_1,Q)).  
    \end{cases}
\end{align}

\begin{proof}[Proof of \Cref{thm:AW_Brenier_Intro}]
We write $P = \lawadd(X)$ and $Q= \lawadd(Y)$. Suppose that there is an adapted-law invariant function $\overline{\phi} : V_N \to (-\infty,+\infty]$ such that $Y(u_1, \ldots) \in \partial \overline{\phi}( X(u_1, \ldots))$ for a.e.\ $ u_1\in [0,1]$. Writing $\phi$ for the corresponding function on $\Zc_1$ we have 
\[
\phi(\lawadd(X(u_1,\cdot))) + \phi^c(\lawadd(Y(u_1,\cdot))) = \overline{\phi}(X(u_1,\cdot)) + \overline{\phi}^\ast(Y(u_1,\cdot)) = \langle X(u_1,\cdot) , Y(u_1,\cdot) \rangle_{V_N}.
\]
By integration over $\lambda(du_1)$ we find
\[
\int \phi \, dP + \int \phi^c \, dQ = \E[ X \cdot Y].
\]
On the other hand \Cref{thm:FPisometry} and the fundamental theorem of optimal transport yield that
\[
\AMC(X,Y) = \MC(P,Q) = \inf_{\phi} \int \phi \, dP + \int \phi^c \, dQ .
\]
Hence, $(X,Y)$ is optimal. 

Conversely, assume that $(X,Y)$ is optimal. Then on the probability space $([0,1],\lambda)$ (denoting the elements of it with $u_1$) the pair of random variables 
\[
( \lawadd(X(u_1,\cdot)), \lawadd(Y(u_1,\cdot))  )
\]
is optimal for the transport problem between $P$ and $Q$ with cost $c$. By the fundamental theorem of optimal transport the exists a $c$-convex function $\phi : \Zc_1 \to (-\infty,+\infty]$ such that for $\lambda$-a.e. $u_1$ we have 
\[
\lawadd(Y(u_1,\cdot)) \in \partial_c \phi(  \lawadd(X(u_1,\cdot))). 
\]
Writing $\overline{\phi}$ for the lift of $\phi$ as in \eqref{eq:LiftZ1}, display \eqref{eq:subdiff_char_FP_fancy} yields that $Y(u_1,\cdot) \in \partial\overline{\phi}(X(u_1,\cdot))$ for almost every $u_1$. 

This finishes the proof of the first equivalence. The second claim concerning naturally filtered processes follows immediately as being naturally filtered is a property of the adapted law (i.e.\  if $X \simad J(\mu)$ and $Y \simad J(\nu)$ they are necessarily naturally filtered). 
\end{proof}

\subsection{Failure of an adapted Brenier theorem for absolutely continuous measures}
Finally we provide an example showing that a naive attempt to extend Brenier's theorem to the adapted Wasserstein  fails already in the first non-trivial instance of one-dimensional state space and one time interval:

\begin{example}\label{ex:NoPlainAdaptedBrenier}
There exist absolutely continuous, compactly supported probabilities on $\R^2$ such that the adapted transport problem for the quadratic costs admits no Monge type solution:

\medskip

Let $\mu_1=\nu_1$ be the uniform probability measure on 
$ 
\big[-\tfrac{1}{\sqrt{2}},\,\tfrac{1}{\sqrt{2}}\big],
$ 
and define $\mu, \nu$ on $\R^2$ by
$$ \mu(dx_1,dx_2):= \mu(dx_1) q_{a(x_1)}(dx_2)), \quad \nu(dy_1,dy_2):= \nu(dy_1) q_{b(y_1)}(dy_2))$$
where we write $q_x$ for the uniform distribution on $[x,x+1]$ and  \[
a(x):=\begin{cases}
\sqrt{2}\,x-1, & x\ge 0,\\[4pt]
\sqrt{2}\,x+1, & x<0,
\end{cases}
\qquad
b(y):=\tfrac{y}{\sqrt{2}}.
\]
By the dynamic programming principle the squared adapted Wasserstein distance of $\mu, \nu$ is given by
\begin{align}
    \AW^2_2(\mu,\nu) &= \inf_{\pi_1\in \cpl(\mu_1,\nu_1)} \int (x_1-y_1)^2 + \W_2^2(q_{a(x_1)}, q_{b(y_1)}) \, d\pi_1(x_1,y_1)=\\
    &= \inf_{\pi_1\in \cpl(\mu_1,\nu_1)} \int c(x_1,y_1) \, d\pi_1(x_1,y_1),   \label{eq:easy1dtransport}
\end{align}
where \(c(x,y) = (x-y)^2+\bigl(a(x)-b(y)\bigr)^2
\). The one dimensional transport problem  \eqref{eq:easy1dtransport} admits an elementary solution. 
Indeed, 
the optimal transport plan between $\mu_1$ and $\nu_1$ is of Monge type and is uniquely given $\mu_1$-a.e.\ by the transport map
\[
T_1(x)=
\begin{cases}
2x-\tfrac{1}{\sqrt{2}}, & x\geq0,\\[4pt]
2x+\tfrac{1}{\sqrt{2}}, & x<0.
\end{cases}
\] and  the dual maximizers are given by
\[
 \phi_1(x)=\tfrac{1}{2}-x^2,\qquad 
\psi_1(y)=\tfrac{1}{2}y^2,\qquad x,y\in\left[-\tfrac{1}{\sqrt{2}},\,\tfrac{1}{\sqrt{2}}\right].
\]
It is straightforward to verify this using the complementary slackness conditions.

It follows that the unique optimal bi-adapted transport between $\mu$ and $\nu$ is induced by the map $T(x_1,x_2)=(T_1(x_1), x_2+b(T_1(x_1)) - a(x_1))$. 

Crucially, the map $T_1$ is not invertible on $ 
\big[-\tfrac{1}{\sqrt{2}},\,\tfrac{1}{\sqrt{2}}\big],
$ hence there exists not optimal transport map from $\nu_1$ to $\mu_1$ and hence also no Monge solution for the adapted transport problem from $\nu$ to $\mu$.
\end{example}

\appendix

\section{Auxiliary results}\label{sec:app}
In this section $U : [0,1] \to \R$ defines the identity and $U = (U_t)_{t=1}^N : [0,1]^N \to \R^N$ defines the coordinate process. We equip $[0,1]$ and $[0,1]^N$ with the Lebesgue measure $\lambda$.

\begin{lemma} \label{lem:funny.disint}
    Let $\mathcal X, \mathcal Y$ be Polish and $f : \mathcal X \to \mathcal Y$ measurable.
    Then, there is a Borel map
    \[
        T : \mathcal P(\mathcal X) \times (0,1) \to \mathcal X
    \]
    such that $T(\mu,U) \sim \mu$ and, for all $\mu,\nu \in \Pc(\mathcal X)$ with $f_\# \mu = f_\# \nu$, $f(T(\mu,U))= f(T(\nu,U))$ a.s.
\end{lemma}

\begin{proof}
    Using a Borel isomorphism we assume that $\mathcal X = \mathcal Y = \R$.
    By \cite[Lemma B.1]{La20}  there exists a jointly measurable version of disintegration of measure, i.e., there is a Borel map $S : \Pc(\mathcal X) \times \mathcal Y \to \Pc(\mathcal X)$ such that for $\mu \in \Pc(\mathcal X)$
    \begin{align*}
        \mu(dx) &= \int S(\mu,y;dx) \, f_\# \mu(dy), \\
        S(\mu,y;f^{-1}(\{y\})) &= 1 \quad \text{for } f_\# \mu\text{-a.e.\ $y$}.
    \end{align*}
    Further, let $R = (R_1,R_2) : (0,1) \to (0,1)^2$ be Borel measurable with $R_\# \lambda = \lambda^2$.
    Writing $Q_\nu$ for the quantile function of $\nu \in \mathcal P(\R)$, we define
    \[
        T(\mu,u) := Q_{S(\mu,Q_{f_\# \mu}(R_1(u)))}(R_2(u)),
    \]
    for $(\mu,u) \in \Pc(\mathcal X) \times (0,1)$.
    It follows directly from inverse transform sampling that $T(\mu,U) \sim \mu$.
    Next, let $\mu,\nu \in \mathcal P(\mathcal X)$ with $f_\# \mu = f_\# \nu =: \eta$, then we have
    \[
        S(\mu,y;f^{-1}(\{y\})) = 1 = S(\nu,y;f^{-1}(\{y\})) \quad \text{for } \eta\text{-a.e.\ $y$.}
    \]
    We deduce that for $\eta$-a.e.\ $y$
    \[
        f(Q_{S(\mu,y)}(U)) = y = f(Q_{S(\nu,y)}(U))\quad\text{a.s.}
    \]
    We conclude that
    \[
        f(T(\mu,U)) = Q_\eta(R_1(U)) = f(T(\nu,U)) \quad \text{a.s.} \qedhere
    \]
\end{proof}

\begin{lemma}\label{lem:ip-meas-rv}
    Let $X,Y$ be random variables on $(\Omega,\mathcal F,\P, (\F_t)_{t = 1}^N)$.
    Then, there exist random variables $X',Y' : (0,1)^N \to \R^d$ on $((0,1)^N,\B((0,1)^N),\lambda^N, (\sigma(U_{1:t}))_{t = 1}^N)$ such that
    \begin{itemize}
        \item for $t = 1,\dots, N-1$, $X'$ has a $(\law^{\mathcal F_{N-1:t}}(X'),U_{t+1:N})$-measurable version,
        \item $(X',Y') \simad (X,Y)$.
    \end{itemize}
\end{lemma}

\begin{proof}
    We show the claim by induction.
    More specifically, we show the following:

    There exists a measurable map $S_N = (X_N,Y_N) : \Pc^N(\R^d \times \R^d) \times [0,1]^N \to \R^d \times \R^d$ with
    \begin{itemize}
        \item for $P,Q \in \Pc^N(\R^d \times \R^d)$, $u,v \in (0,1)^N$ and $t = 0,\dots,N-1$ with $X_N(P,u_{1:t},U_{t+1:N}) \simad X_N(Q,v_{1:t},U_{t+1:N})$, then $X_N(P,u_{1:t},U_{t+1:N}) = X_N(Q,v_{1:t},U_{t+1:N})$ a.s.,
        \item for $P \in \Pc^N(\R^d \times \R^d)$, $(X_N(P,U_{1:N}),Y_N(P,U_{1:N})) \simad P$.
    \end{itemize}

    First, let $N = 1$. Then the claim follows from \Cref{lem:funny.disint} with $\mathcal X = \R^d \times \R^d$, $\mathcal Y = \R^d$ and $f(x_1,x_2) = x_1$.

    Next, assume the claim holds true for $N - 1$.
    We invoke \Cref{lem:funny.disint} with $\mathcal X = \mathcal P^{N-1}(\R^d \times \R^d)$, $\mathcal Y = \mathcal P^{N-1}(\R^d)$ and $f = \mathcal P^{N-1}[\pr_1]$ to find a Borel map $T : \Pc^N(\R^d \times \R^d) \times (0,1) \to \Pc^{N-1}(\R^d \times \R^d)$ with $T(P,U_1) \sim P$ and, for $P,Q \in \Pc^N(\R^d \times \R^d)$ with $f_\# P = f_\# Q$, then $f(T(P,U_1)) = f(T(Q,U_1))$ a.s.
    We set
    \[
        X_N(P,u_{1:N}) := X_{N-1}(T(P,u_1),u_{2:N}) \text{ and }
        Y_N(P,u_{1:N}) := Y_{N-1}(T(P,u_1),u_{2:N}),
    \]
    which defines a measurable map $S_N = (X_N,Y_N) : \Pc^N(\R^d \times \R^d) \times (0,1)^N \to \R^d \times \R^d$.
    By construction we have for fixed $(P,u_1) \in \Pc^N(\R^d \times \R^d) \times (0,1)$
    \begin{align*}
        S_N(P,u_1,U_{2:N}) &= (X_N(P,u_1,U_{2:N}), Y_N(P,u_1,U_{2:N})) \\ 
        &= (X_{N-1}(T(P,u_1),U_{2:N}), Y_{N-1}(T(P,u_1),U_{2:N})) \simad T(P,u_1).
    \end{align*}
    Hence, $S_N(P,U_{1:N}) \simad (T(P,\cdot))_\# \lambda = P$ which also yields $X_N(P,U_{1:N}) \simad f_\# P$ since $f = \Pc^{N-1}[\pr_1]$.
    Therefore, if $X_N(P,U_{1:N}) \simad X_N(Q,U_{1:N})$ then $f_\# P = f_\# Q$ and thus $f(T(P,U_1)) = f(T(Q,U_1))$ almost surely.
    Consequently, we get
    \[
        X_N(P,U_{1:N})) = X_{N-1}(T(P,U_1),U_{2:N}) = X_{N-1}(T(Q,U_1),U_{2:N}) = X_N(Q,U_{1:N}) \quad\text{a.s.}
    \]
    If additionally $u,v \in (0,1)^N$ are such that $X_N(P,u_{1:t},U_{t+1:N}) \simad X_N(Q,v_{1:t},U_{t+1:N})$, then we also have that $X_{N-1}(T(P,u_1),u_{2:t},U_{t + 1:N}) \simad X_{N-1}(T(Q,v_1),v_{2:t},U_{t + 1:N})$.
    Using the inductive assumption we obtain that a.s.
    \begin{align*}
        X_N(P,u_{1:t},U_{t+1:N}) &= X_{N-1}(T(P,u_1),u_{2:t},U_{t+1:N}) \\
        &= X_{N-1}(T(Q,v_1),v_{2:t},U_{t+1:N}) = X_N(Q,v_{1:t}, U_{1:N}).  
    \end{align*}
    This completes the inductive step.

    Finally, we let $(X',Y') :=  S_N(P,U_{1:N})$. We note that for $t = 1,\dots, N-1$
    \[
        X' = \E[ X' | \law^{\F_{N-1:t}}(X'),U_{t+1:N}] \quad\text{a.s.},
    \]
    hence, $(X',Y')$ have the desired properties.
\end{proof}

\begin{proof}[Proof of \Cref{lem:DPP}]
    First note that
    \begin{align*}
        c^{(N)}(P,P') &= \inf_{\Pi \in \cpl^N(P,P')} \int c(x,x') \, dI^{N-1} \Pi(x,x') \\
        &\le \inf_{Y \simad P} \E[c(Y,X')] \le \inf_{ \substack{ T : \Omega \to \Omega \\ {\rm isomorphisim} } } \E[c(X \circ T, X')],        
    \end{align*}
    the equality follows as in \Cref{prop:DPP_CplN_MC}, the first inequality by \Cref{lem:cplN} and last inequality follows from \Cref{lem:iso_pres_ip}.

    We write $U=(U_t)_{t=1}^N$ for the coordinate process on $[0,1]^N$. The distributions
    \[
        \mu := \law(U_1,\dots,U_{N-1},(U_N,X)) \text{ and }
        \nu := (U_1,\dots,U_{N-1},(U_N,X'))
    \]
    have atomless disintegrations, that is, $\lambda = \law(U_1) = \law(U_{t + 1} | \F_t)$ as well as 
    \begin{align*}
        \mu_{u_{1:N-1}} :=& \law(U_N,X|U_{1:N-1} = u_{1:N-1}),\\
        \nu_{u_{1:N-1}} :=& \law(U_N,X'|U_{1:N-1} = u_{1:N-1}),
    \end{align*}
    are $\lambda^N$-almost surely atomless and thus w.l.o.g.\  atomless. 
    In what follows we view $\mu$ and $\nu$ as laws of the $N$-time step stochastic processes $(U_1,\dots,U_{N-1},(U_N,X))$ and $(U_1,\dots,U_{N-1},(U_N,X'))$, respectively, and use the denseness of biadapted Monge couplings in the set of bicausal couplings to show the remaining inequality.
    Therefore we have to relate the adapted optimal transport problem between $\mu$ and $\nu$ with $c^{(N)}(P,P')$.
    To this end, we define
    \[
        \bar\mu_{U_{1:t}} := \law^{\F_{N-1:t}}(X) = \law^{\F_{N-2:t}}(\mu_{U_{1:N-1}}) \text{ and }
        \bar \nu_{U_{1:t}} := \law^{\F_{N-1:t}}(Y)= \law^{\F_{N-2:t}}(\nu_{U_{1:N-1}}).
    \]
    Due to the dynamic programming principle of optimal transport under bicausality constraints
    \[
        \inf_{\pi \in \cpl_{\rm bc}(\mu,\nu)} \int c(x,x') \, d\pi((u_1,\dots,u_{N-1},(u_N,x)),(v_1,\dots,v_{N-1},(v_N,x'))) = V^{(0)},
    \]
    see, for example, \cite[Theorem 3.2]{AcKrPa24}, where $\cpl_{\rm bc}(\mu,\nu)$ denotes the set of bicausal couplings, see \cite[Definition 2.4]{AcKrPa24}, and $V^{(0)}$ is inductively given by
    \begin{gather*}
        V^{(N-1)}(u_{1:N-1},v_{1:N-1}) := \inf_{\pi \in \cpl(\mu_{u_{1:N-1}},\nu_{v_{1:N-1}})} \int c(x,x') \, d\pi((u_N,x),(v_N,x')), 
        \\
        V^{(t)}(u_{1:t},v_{1:t}) := \inf_{\pi \in \cpl(\lambda,\lambda)} \int V_{t + 1 } \, d\pi, \quad
        V^{(0)} := \inf_{\pi \in \cpl(\lambda,\lambda)} \int V_1 \, d\pi.
    \end{gather*}
    By comparing the definition of $c^{(t)}$ and $V^{(t)}$, we have for all $u,v$ in a $\lambda$-full set
    \begin{align*}
        V^{(N-1)}(u_{1:N-1},v_{1:N-1}) = c^{(N-1)}(\bar \mu_{u_{1:N-1}},\bar \nu_{v_{1:N-1}}), \quad
        V^{(t)}(u_{1:t},u'_{1:t}) = c^{(t)}(\bar \mu_{u_{1:t}},\bar \nu_{v_{1:t}}),
    \end{align*}
    and thus $V^{(0)} = c^{(N)}(P,P')$.
    By \cite{BePaSc21c}, bicausal transport plans given by biadapted Monge maps are dense in $\cpl_{\rm bc}(\mu,\nu)$, since $\mu$ and $\nu$ both have atomless successive disintegration as argued above.
    Let $T = (T^1,\dots,T^N)$ be a biadapted Monge map from $\mu$ to $\nu$, that is, $T_\# \mu = \nu$ and
    \begin{gather*}
        T^t : [0,1]^t \to [0,1],  \quad T^t(u_{1:t-1},\cdot)_\# \lambda = \lambda, \\
        T^N = (R^N,S^N) : [0,1]^N \times \mathcal X \to [0,1]^N \times \mathcal X, \quad T^N(u_{1:N-1},\cdot)_\# \mu_{u_{1:N-1}} = \nu_{u_{1:N-1}},
    \end{gather*}
    and $(T^1(u_1), \dots, T^t(u_{1:t})) : [0,1]^t \to [0,1]^t$ is bijective, for $t = 1,\dots N-1$.
    In particular, $R := (T^1,\dots,T^{N-1},R^N)$ is a biadapted Monge map from $\lambda$ to itself, and $S^N = Y \circ R$ $\lambda$-almost surely.
    Consequently, we obtain
    \[
        c^{(N)}(P,P') = V^{(0)} \ge \inf_{ \substack{ R : [0,1]^N \to [0,1]^N, \\ {\rm bi-adaptiert, R_\#\lambda=\lambda} } } \int c(X(u),X'(R(u))) \, d\lambda(u),
    \]
    which provides the missing inequality and shows the first claim.
\end{proof}

\bibliographystyle{abbrv}
\bibliography{joint_biblio}

\end{document}